\documentclass[10pt]{amsart}
 \usepackage{amsmath}
 \usepackage{times}

\usepackage{aurical}
\usepackage[OT1]{fontenc}
\usepackage{float}
\usepackage{subfig}
\usepackage{subfloat}

\usepackage{amsthm}
\usepackage{amssymb}

\usepackage{amscd}
\usepackage[all]{xy}
\usepackage{color}
\usepackage{verbatim}
\usepackage{mathtools}
\usepackage{stmaryrd}

\usepackage[plainpages=false,colorlinks,hyperindex,pdfpagemode=None,bookmarksopen,linkcolor=red,citecolor=blue,urlcolor=blue]{hyperref}
\usepackage{pdflscape}

\usepackage[OT2,T1]{fontenc}
\DeclareSymbolFont{cyrletters}{OT2}{wncyr}{m}{n}
\DeclareFontFamily{OT1}{rsfs}{}

\usepackage{verbatim}
\usepackage{graphicx}
\usepackage{nomencl}
\DeclareFontEncoding{OT2}{}{} %
  
\newcommand{\Spec}{\mathrm{Spec}}

\newcommand{\cyclo}{\mathrm{cyclo}}

\newcommand{\PGL}{\mathrm{PGL}}
\newcommand{\Out}{\mathrm{Out}}
\newcommand{\sigmastable}{compatible } 
\newcommand{\Lie}{\mathrm{Lie}}
\newcommand{\LL}{\mathbf{L}}
\newcommand{\Z}{\mathbf{Z}}

\newcommand{\Ad}{\mathrm{Ad}}
\newcommand{\SL}{\mathrm{SL}}
 \renewcommand{\AA}{\mathbf{A}}
\newcommand{\Spin}{\mathrm{Spin}}

\newcommand{\twist}{(p)}
\newcommand{\SO}{\mathrm{SO}}

\newcommand{\Frob}{\mathrm{Frob}}
\newcommand{\Aut}{\mathrm{Aut}  }

\newcommand{\Afinite}{\mathbf{A}_{\mathrm{f}}}
\newcommand{\Adele}{\mathbf{A}}
\newcommand{\C}{\mathbf{C}}
\newcommand{\Tate}{\mathrm{T}}
\newcommand{\PP}{\mathbf{P}}

\newcommand{\ab}{\mathrm{ab}}

\newcommand{\Hom}{\mathrm{Hom}}
 
\newcommand{\F}{\mathbf{F}}
\newcommand{\cF}{\mathscr{F}}

\newcommand{\GG}{\mathbf{G}}

\newcommand{\LT}{{}^L \hat{T}}

\newcommand{\LH}{{}^L \hat{H}}
\newcommand{\LG}{{}^L \hat{G}}
\newcommand{\CG}{{}^C \hat{G}}
\newcommand{\cG}{{}^c \hat{G}}
\newcommand{\cgroupH}{{}^c \hat{H}}
\newcommand{\CH}{{}^C \hat{H}}

\newcommand{\Hecke}{\mathscr{H}}

\newcommand{\Q}{\mathbf{Q}}

 \newcommand{\adele}{\mathbf{A}}
\newcommand{\OO}{\mathcal{O}}
\newcommand{\BB}{\mathbf{B}}
\newcommand{\TT}{\mathbf{T}}

\newcommand{\R}{\mathbf{R}}
\newcommand{\G}{\mathbf{G}}
 \newcommand{\HH}{\mathbf{H}}

\newcommand{\Sp}{\mathrm{Sp}}

\DeclareFontShape{OT1}{rsfs}{n}{it}{<-> rsfs10}{}
\DeclareMathAlphabet{\mathscr}{OT1}{rsfs}{n}{it}
\newcommand{\GL}{\mathrm{GL}}

\swapnumbers

\newtheorem{theorem}[subsection]{Theorem}
\newtheorem*{theorem*}{Theorem}
\newtheorem*{theoremFMT}{First Main Theorem}
\newtheorem*{theoremSMT}{Second Main Theorem}

\newtheorem{lemma}[subsection]{Lemma}

\newtheorem{prop}[subsection]{Proposition}
 \newtheorem*{prop*}{Proposition}
\theoremstyle{definition}

\newtheorem{definition}[subsection]{Definition}

\newtheorem*{defn*}{Definition}
\newtheorem*{lemma*}{Lemma}
\newtheorem*{corollary*}{Corollary}
\newtheorem*{conjecture*}{Conjecture}

\newcommand{\cH}{\Hecke}
\newcommand{\Fun}{\mathrm{Fun}}
\definecolor{darkgreen}{RGB}{0,150,0}

\newcommand{\Br}{\mathrm{Br}}
\newcommand{\FBr}{\widetilde{\Br}}

\newcommand{\NBr}{\mathrm{br}}

\renewcommand{\sl}{\mathfrak{sl}}
\newcommand{\so}{\mathfrak{so}}
\renewcommand{\sp}{\mathfrak{sp}}

\newcommand{\Gal}{\mathrm{Gal}}

\newcommand{\vhE}{{}^{\mathit{vh}}\!E}
\newcommand{\hvE}{{}^{\mathit{hv}}\!E}
\newcommand{\Lpsi}{{}^L \hat{\psi}}
\newcommand{\Cpsi}{{}^C \hat{\psi}}
\newcommand{\WittRing}{\Lambda}
\newcommand{\WittField}{\Lambda[p^{-1}]}
\newcommand{\Ga}{\mathbb{G}_{\mathrm{a}}}
\newcommand{\Gm}{\mathbb{G}_{\mathrm{m}}}

\newcommand{\Cent}{\mathit{Z}}

\numberwithin{equation}{subsection}

\newcommand{\git}{/\!\!/}

\newcommand{\Tcan}{\TT^{\mathrm{can}}}
\newcommand{\inverserho}{\rho^{-}}

\newcommand{\cleanup}[1]{\marginpar{\color{blue} \tiny #1}}

\begin{document} 
 
\title{  Functoriality, Smith theory, and the Brauer homomorphism}
\author{David Treumann and Akshay Venkatesh}

\begin{abstract}
 
  If $\sigma$ is an automorphism of order $p$ of the semisimple group $\mathbf{G}$,
  there is a natural correspondence between {\em mod $p$} cohomological automorphic forms
  on $\mathbf{G}$ and $\mathbf{G}^{\sigma}$. We describe this correspondence in the global and local settings.

\end{abstract}
\maketitle

\section{Introduction} \label{Secintro}

  \subsection{}
  Let $\G$ be a semisimple group over a number field $F$, with Langlands-dual
$\LG$, 
and let  $k$ be an algebraically closed field of positive characteristic $p$. 
By ``mod $p$ automorphic forms'' for $\G$ we shall mean Hecke eigenclasses
in the cohomology of congruence subgroups with $k$-coefficients. 
We  {\em make no assumption that these cohomology classes lift
to characteristic zero}, i.e. there may be  no automorphic form in the classical sense associated to this eigenclass.

Now let $\sigma$ be an order $p$ automorphism of $\mathbf{G}$, defined over $F$, with a connected fixed point subgroup $\GG^\sigma$.  The main goal of this paper is to show that there is a close relationship between 
mod $p$ automorphic forms on $\G$ and mod $p$ automorphic forms on $\G^{\sigma}$: we 
construct a homomorphism (\S \ref{nBr}) 
\[
\begin{array}{c}
\mbox{ Hecke algebra for $\G$ at $v$,} \\
\mbox{with $k$ coefficients} 
\end{array}
\stackrel{\psi_v}{\longrightarrow} 
\begin{array}{c}
\mbox{ Hecke algebra for $\G^{\sigma}$ at $v$,}\\
\mbox{with $k$ coefficients}
\end{array}
\]
which is a slight variant of the ``Brauer homomorphism'' of modular representation theory.
We prove  (see Theorem \ref{thm:realtheorem}):

\begin{theoremFMT} If a mod $p$ automorphic form for $\G^{\sigma}$  has Satake parameters
 $\{a_v\}$ then there exists a mod $p$ automorphic form for $\G$ with Satake parameters
 $\{\psi_v^*(a_v)\}$. \end{theoremFMT}

 What is the relationship
between the parameters of these forms at ramified places?   This is answered by Theorem \ref{thm:behavioratramified},
based on the the notion of ``linkage'' (Definition \ref{linkagedef}) of local representations.   Roughly speaking,
these results suggest that  local functoriality should be realized by Tate cohomology. 
 
One of course wants to compute $\psi_v$. This is accomplished by
the Theorem of \S\ref{BrauerSatake}. 
The formulation in \S\ref{BrauerSatake} is not the natural one
 from the point of view of the Langlands program. The remainder of the paper addresses this issue, which we now describe:
 
 \subsection{Torsion functoriality} 
 \label{torsionfunctoriality} Write $\HH=\GG^{\sigma}$.  
One wants to know if the ``lift'' furnished by the theorem is a ``functorial lift'' in the sense of Langlands: 
is $\psi_v$  induced by a $L$-homomorphism
 of $L$-groups $$\Lpsi:  \LH \rightarrow \LG?$$ For the precise meaning of the word ``induced'' here,
 see \S \ref{subsec:sigmadual}. 
 In the setting of our paper, it is natural to construct
 the dual groups as algebraic groups {\em over $k$}, rather than over $\C$,
 and we shall always follow that convention  (\S \ref{subsec:dgalg}). 
 In that case the Theorem amounts to functoriality for $\psi$ and we
 call such
 an $\Lpsi$ a {\em $\sigma$-dual homomorphism}.

 A basic example is that of cyclic base change.  If $\GG =\mathrm{Res}_{E/F} (\mathbf{H} \otimes_F E)$ 
where $E/F$ is a cyclic extension of degree $p$, and $\sigma$ induced by a generator of $\Aut(E/F)$, 
then we may take $\Lpsi$ to be the homomorphism obtained by restriction on the Galois group, 
and the result amounts to ``cyclic base change.'' It doesn't follow from
usual base change, even for   $\mathbf{H} = \mathrm{SL}_2$ and $F=\Q(i)$,  as our result applies to characteristic $p$ torsion classes. 
Indeed,  this result itself seems interesting: we do not know of any other examples of where functoriality can be established, in such a strong form, for torsion classes (compare \cite{CV}).

Our second main  theorem verifies that such $\Lpsi$ exists at least  in some generality:

  \begin{theoremSMT}
  Suppose that $\GG$ is simply connected and $\HH$ is semisimple\footnote{In this case, $\HH$ is automatically connected.}. Then there exists a $\sigma$-dual homomorphism $\LH \rightarrow \LG$
  with the possible exception of cases\footnote{The remaining cases could presumably be similarly treated,
but the authors were left exhausted by the existing proof.
 } where   $(\Lie(\GG), \Lie(\HH))$ contains a factor of the form
 $(\mathfrak{e}_6, \mathfrak{sl}_3^3  \mbox{ or } \mathfrak{sl}_6 \times \mathfrak{sl}_2 \mbox{ or } \mathfrak{sp}_8)$.
 \end{theoremSMT} 
 In summary, then, the two theorems together give a functoriality for  mod $p$ automorphic forms, from $\HH$ to $\GG$. 
 
We prove this second theorem using the Theorem of \S\ref{BrauerSatake}, some general arguments,  and  finally some case-by-case verification for some exceptional cases.   Because the case-by-case verifications  are lengthy,
and the remainder of the proof is already long,
we present only the most interesting cases (e.g. the order $5$ inner automorphism of $\mathrm{E}_8$) in full, and simply summarize some salient data in the remaining cases. 

  In many instances the validity of the Theorem  is related to delicate
 properties of the relevant groups in characteristic $p$. For
 example,  
when $\G$ is a split form of $\mathrm{Spin}_8$ and $\sigma$ is the pinned triality with fixed points $\mathrm{G}_2$, 
the $\sigma$-dual homomorphism $\mathrm{G}_2 \rightarrow \mathrm{PSO}_8$ is not the ``standard'' eight-dimensional representation
but (Proposition \ref{crit4} and Lemma \ref{pinnedcase})  its twist by the exceptional  isogeny $\mathrm{G}_2 \rightarrow \mathrm{G}_2$ that exists only in characteristic $3$. 
Also the Galois component of the $L$-group enters in an interesting way, especially when $p>2$:
roughly speaking, the action of $\Lpsi$ on the ``Galois component'' of $\LH$ must 
exactly compensate the 
  difference between half-sum shifts for $\HH$ and $\GG$  (a simple example is 
the inner   automorphism of $\mathrm{G}_2$ of order $3$, see Proposition \ref{crit6}).

If we relax the assumption   that  $\GG$ is  simply connected  such an $\Lpsi$ need not exist: a counterexample is given by taking $\sigma$ to be an inner automorphism of order $3$
 of $\GG=\PGL_2$. The difficulties arising here are related to the ambiguity of square roots in  the Satake transform. 
 A plausible solution is described in \S \ref{subsec:Cgroup}:
  one can replace the $L$-group by  a different extension  $\hat{G} \rightarrow \cG \rightarrow \Gal(\bar{F}/F)$ but without the ``pinned splitting'' of $\LG$,
  with respect to which the Satake isomorphism can be formulated canonically (Theorem \ref{cSatake}). 
It  tempting to speculate then that there is always a $\sigma$-dual map $\cgroupH \rightarrow \cG$.  This variant $\cG$ is a subgroup of the ``$C$-group'' discussed in \cite{BG}
and is closely related to the ideas of Deligne \cite{Deligne}.  Since this story is somewhat orthogonal to our main goals we do not examine it further in this paper.  It would also be interesting to relax the assumption that $\HH$ is connected.  In that case it appears that the First Main Theorem, suitably interpreted, remains valid, but we have not investigated if it has an $L$-group formulation.

 \subsection{Discussion}
 
In the local setting, this correspondence is related to the ``Brauer correspondence''
of modular representation theory.  
In the global setting, the correspondence   can be viewed as a kind of ``mod $p$ Eisenstein series;''
like Eisenstein series, many phenomena related to the Langlands program simplify but yet do not become trivial.

Indeed, Eisenstein series are, in a sense, dual to the operation of restricting an automorphic form to the boundary. 
Here we observe that the symmetric space for $\mathbf{G}^{\sigma}$, embedded in the symmetric space
for $\mathbf{G} $, behaves {\em with respect to  characteristic $p$ homology} like a kind of ``interior boundary.'' 
A better-known example of this phenomenon is that ``restriction to supersingular points''
 gives, for classical modular forms, a geometric construction of the Jacquet--Langlands correspondence modulo $p$.  
 
The main technical tool to prove the ``interior boundary'' property is  Smith theory, or $\Z/p$-equivariant localization.   It is related to
   the prior paper \cite{T} of the first-named author.   If we call $Y$ (resp. $X$) the locally symmetric space
   for $\G^{\sigma}$ (resp. $\G$)    and put $\Gamma = \langle \sigma \rangle$, then the inclusion
$Y \hookrightarrow X$ induces (almost) an isomorphism on equivariant cohomology
    $H^*_{\Gamma}(X) \rightarrow H^*_{\Gamma}(Y)$, and what remains is ``just'' to  
    pass from equivariant cohomology to usual cohomology,  and to understand Hecke actions.

  Given a sufficiently good chain-level understanding of the $\sigma$-action on the cohomology of $X$, for instance a \sigmastable triangulation of $X$, our method gives an explicit recipe for lifting automorphic forms on $Y$ to $X$.  The recipe can be presented as a spectral sequence (see proof of Theorem \ref{thm:smith}). 
    We do not expect it to degenerate
and indeed the differentials seem to carry interesting information.  It will be interesting to study this further.

   One cannot be too optimistic about ``lifting'' the method to characteristic zero
in any direct way. The proof uses special properties of the Frobenius at various points;
in fact the homomorphism of dual groups mentioned above need not lift to characteristic zero, 
as in the example above, or the inner examples of \cite{T}.
Nonetheless  the results here {\em could} be lifted to characteristic zero, using more machinery, in the following situation:
 \begin{itemize}
  \item[(i)] There is a $\sigma$-dual homomorphism  $\LH \rightarrow \LG$ and it  lifts to characteristic zero,
  \item[(ii)] On $\HH$ one may associate  Galois representations to mod $p$ forms;
  \item[(iii)]   On $\GG$ one has available {\em modularity lifting theorems}.   
  \end{itemize} 
then one can in fact 
deduce the corresponding functoriality {\em in characteristic zero}  (i.e., for cohomological automorphic forms)  from the above Theorem.
The recent work of Scholze and Calegari--Geraghty \cite{Scholze, CG} enlarges  the list of possibilities where (ii) and (iii) apply. 
We do not pursue this here.

We note some related  work. One inspiration for this paper was trying
to understand the ideas behind the Glauberman correspondence \cite{Glauberman}. 
In a Langlands setting the closest paper appears to be the recent work of Clozel \cite{Clozel};
it   is closely related to the  ideas of the current paper, specialized to the case of $\G$ a definite
quaternion algebra over a totally real number field; it moreover  makes intriguing use of this idea in the context of an  infinite $p$-adic tower (see \S \ref{openquestions}).  The paper \cite{T} of the first author, already mentioned,
studies a similar story in the setting of the local geometric Satake correspondence, when $\sigma$ is inner. 
L. Clozel has pointed out that the arguments of \S \ref{BrauerSatakeproof}
resemble some of the constructions in the theory of twisted endoscopy, as in \cite{KottwitzShelstad}, but
in our case these constructions are on the dual side and in characteristic $p$.  The paper of Kionke \cite{Kionke} applies the Smith inequalities to $p$-adic analytic towers of locally symmetric spaces for $\GG$.
Finally 
we draw attention to the paper of Ash \cite{Ash}, which uses Smith theory
to produce homology for $\GL_n$ over certain fields.

  \subsection{Open questions} \label{openquestions} 
  We mention five interesting open questions:
 
  \begin{itemize}
  \item[(i)] The ramified correspondence: we formulate a conjecture in \S \ref{sec:behavioratramified}
  relating local functoriality to Tate cohomology.  This is a problem solely in the representation theory of $p$-adic groups. 
  N. Ronchetti has obtained some evidence for this, 
  in the setting of cyclic base change for $\GL_n$ and supercuspidal representations. 
  \item[(ii)] Behavior in a $p$-adic tower:   In the case of a definite quaternion algebra,
  Clozel \cite{Clozel}  formulates a theory of ``automorphic forms over $\Q(\zeta_{p^{\infty}})$.''
  Can one make a similar theory for cohomological forms on an arbitrary group, using the ideas of the current paper? 
    \item[(iii)]  Generalization of second main theorem: replacing the $L$-group by the $c$-group that we define in \S \ref{Satakeparam},
    establish the existence of $\sigma$-dual homomorphisms in  all cases --- that is, without any restriction that $\GG$ be simply connected.
    (Even the situation when $\HH$ is a torus would be of interest.)  
    \item[(iv)] Spectral sequences: Study more carefully the higher differentials in the spectral sequence of Theorem \ref{thm:smith}. In our context, 
    these higher differentials cannot always be zero, and it would be interesting to understand their arithmetic importance. 
   \item[(v)]   In general the locally symmetric space $Y$ for $\GG^\sigma$ is only part, a union of connected components, of the full space of $\sigma$-fixed points of the locally symmetric space $X$ of $\GG$.  The other components appear to be locally symmetric spaces for different $F$-forms of $\GG^\sigma$.  Smith theory realizes their cohomology as a subquotient of the cohomology of $X$, but we have not investigated the compatibility with Hecke actions.
  \end{itemize}

 \subsection{Plan of the paper} 
 \S \ref{sec:notation} summarizes some of our notations, and \S \ref{sec:Tate} some basic facts about Tate cohomology for cyclic groups.
 
 \S \ref{sec:Brauer} and \S \ref{sec:cyclic}
 describe the Brauer homomorphism and describe the proof of the first main theorem.
 
 \S \ref{sec:behavioratramified} describes the situation at ramified places. The results of this section are also used in the later parts of the paper. 
 
 \S \ref{Satakeparam} consists of ``folklore results'' on the Satake transform. We have stated and proved them here because
 we do not know of a reference with characteristic $p$ coefficients. However, the formulation of Theorem \ref{cSatake} may be of independent interest. 
 
\S \ref{BrauerSatakeproof}  computes the unramified Brauer homomorphism in terms of Satake parameters. 
Rather than compute directly, we deduce the result  by applying the results of \S \ref{sec:behavioratramified} to unramified representations.

\S \ref{section:prelim} collects some preliminaries to the proof of our second main theorem. \S \ref{sec9} proves
it in many cases when $\sigma$ is an inner automorphism, \S \ref{casebash3} proves it in many cases
including when $\sigma$ is a pinned automorphism, and \S \ref{casebash4} 
handles all remaining cases by direct computation.

\subsection{Acknowledgements}

We thank Raphael Rouquier and Zhiwei Yun for sharing their ideas about the Glauberman correspondence,
and Brian Conrad and Mark Reeder for taking time to answer several questions about algebraic groups.   
Finally we thank Laurent Clozel for several helpful discussions as well as comments on the manuscript. 

{\tiny \tableofcontents}

\section{Notation} \label{sec:notation} 

\subsection{Notation used throughout} \label{twistdef}
Let $p$ be a prime number and let $k$ be an algebraic closure of the field with $p$ elements.  Let $\WittRing$ denote the ring of Witt vectors of $k$ (thus, $p$ is a uniformizer for $\WittRing$) and let $\WittField$ denote the fraction field of $\WittRing$.  For $x \in \WittRing$ let $\overline{x} \in k$ denote its reduction mod $p$.

The symbol $\sigma$ denotes a generator for a group of order $p$.  We write $\langle \sigma \rangle$ for this group.

Let $F$ be a number field, $\Gamma_F$ its absolute Galois group, $\OO_F$ its ring of integers, $\Adele$ its adele ring, $\Afinite$ its ring of finite adeles, and set $F_{\infty} = F \otimes_{\Q}\R$.  If $v$ is a place of $F$ then $F_v$ denotes the completion of $F$ at $v$. 
By $\mathrm{cyclo}: \Gamma_F \rightarrow \mathbf{F}_p^* \subset k^*$ we mean the cyclotomic character, i.e. 
the action of $\Gamma_F$ on $p$th roots of unity in $F$. 

If $W$ is a $k$-vector space we denote by $W^{\twist}$ the Frobenius-twist of $W$, i.e. the space with the same underlying vectors
but scalar multiplication modified: if $\cdot$ is the scalar multiplication in $W$ then the scalar multiplication $*$ in $W^{\twist}$ is given by
$\lambda* w = \lambda^{1/p} \cdot w$; equivalently, $W^{\twist} = W \otimes_{(k,\mathrm{Frob})} k$, so that $w \lambda \otimes \mu = w \otimes \lambda^p \mu$
for $\lambda \in k$. 

When we write homology $H_*$ or cohomology $H^* $ of a topological space, we will always understand the coefficients to be taken in $k$, unless otherwise specified.

If $X =\mathrm{Spec}(R)$ is an affine algebraic variety over $k$ and $G$ an algebraic group acting on $X$, we denote by $X \git G$ the geometric quotient, i.e.
the spectrum of $R^G$.

\subsection{Nonabelian cohomology} \label{sec:NAC}
If $\sigma$ is an order $p$ automorphism of a group $M$, we let $H^1(\sigma;M)$ denote the nonabelian cohomology of $\langle \sigma \rangle$ with coefficients in $M$, i.e. cocycles $j:\langle \sigma\rangle \to M$ modulo coboundaries. Elements of $H^1(\sigma;M)$ may be equivalently regarded as elements $j(\sigma).\sigma \subset M.\sigma \subset M \rtimes \langle \sigma \rangle$ modulo $M$-conjugacy, or elements $j(\sigma) \in M$ up to twisted conjugacy.

If $N \subset M$ is a $\sigma$-stable subgroup, then we have the ``long exact sequence'' of nonabelian cohomology \cite[\S\S I.5.4--I.5.5]{GaloisCohomology} 
\[
N^{\sigma} \hookrightarrow M^{\sigma} \to (M/N)^{\sigma} \to H^1(\sigma;N) \to H^1(\sigma;M)
\]
If $N$ is normal in $M$ and nilpotent of order prime to $p$, then $H^1(\sigma;M) \to H^1(\sigma;M/N)$ is a bijection.

\subsection{Algebraic groups and level structures}
\label{subsec:agals}

Let $\G$ be a connected reductive algebraic group over $F$.  If $v$ is a place of $F$, then $\G(F_v)$ is a locally compact topological group that we denote by $G_v$.  If $v$ is a finite place, $K_v$ will denote an open compact subgroup of $G_v$.  If $v$ is an archimedean place, $K_v$ denotes a maximal compact subgroup of $G_v$,
and finally $K_{\infty}$ denotes a maximal compact subgroup of $G_{\infty} = \G(F_{\infty})$. 

By a level structure for $\G$, we mean an open compact subgroup $K \subset \G(\Afinite)$ of the form $\prod_v K_v$.   Note for such a level structure $K_v$ is a ``standard'' maximal compact for almost all $v$ (i.e. is obtained by taking $\OO_v$-points of an integral model of $\GG$ over $\OO_F$).  

When a level $K$ is fixed then for $V$ a set of places of $F$, we denote by $G_V$ the restricted product $\prod'_{v \in V} G_v =  \{ (g_v)_{v \in V} \mid g_v \in K_v \text{ for all but finitely many $v$}\}$.  When $V$ is a set of finite places, we denote by $K_V$ the product $\prod_{w \in V} K_w$
and by $K^{(V)}$ the complementary subgroup
\begin{equation}\label{Kupperdef}  K^{(V)} :=  K_{\infty}  \prod_{w \notin V} K_w \end{equation}
where the product is taken over finite places $w$ not belonging to $V$. 

\subsection{Canonical torus}
\label{subsec:cantorus}

If $F$ is any field, and $\G$ is connected and reductive over $F$, there is a canonical algebraic torus defined over $F$ attached to $\G$.  We denote it by $\Tcan_G$.  
We will follow \cite[\S 1.1]{DeligneLusztig}. 

If $\GG$ is quasisplit we may describe $\Tcan_G$ as the limit $\varinjlim \BB/\mathrm{R}_u(\BB)$  over $F$-rational Borel subgroups of $\GG$, of the quotient torus of the Borels. In general, we pass to an extension over which $\GG$ is quasisplit, and then descend the torus thus constructed; 
 then there may be no inclusion of $\Tcan_G$ into $\GG$ defined over $F$.

Let $X^*$ denote the character lattice of $\Tcan_G \times_F \overline{F}$ and $X_*$ the dual lattice.  Then $(X^*,X_*)$ supports a canonical based root datum $\Psi(\G)$.

Automorphisms of $\overline{F}/F$ induce an action
$
\theta_G:\Gamma_F \to \Aut(\Psi(\G)) \subset \Aut(X^*)
$.  The permutation representation of $\Gamma_F$ on simple roots determines an \'etale algebra over $F$, and if $\G$ is semisimple the $F$-rational points of $\Tcan_G$ are naturally identified with the units in this algebra.

\subsection{Dual groups and $L$-groups}
\label{subsec:dgalg}
{\em By default, we will regard all dual groups as reductive algebraic groups \emph{over $k$}, i.e. in characteristic $p$. }

The dual root datum to $\Psi(\G)$ determines a pinned reductive algebraic group over $k$, which we denote by $\hat{G}$.  Recall that a ``pinning'' is data $(\hat{T},\hat{B},\{X_i\})$ where $\hat{T} \subset \hat{B} \subset \hat{G}$ are a maximal torus and Borel subgroup of $\hat{G}$, and each $X_i$ is a nonzero vector in a simple root space of $\mathrm{Lie}(\hat{B})$.  A pinning determines a splitting $\mathrm{Out}(\hat{G}) \to \mathrm{Aut}(\hat{G})$, and an identification $\mathrm{Out}(\hat{G}) \simeq \mathrm{Aut}(\Psi(\G))$.  

In fact the dual group and its pinning can be constructed over the prime field $\mathbf{F}_p$.  It follows that $\hat{G},\hat{B},\hat{T}$ can be equipped with Frobenius endomorphisms that are defined over $k$, which we will denote by $\mathrm{Frob}$.

The construction $\Psi(\G) \to \Psi(\G)^\vee \to (\hat{G},\hat{B},\{X_i\})$ is functorial, and one obtains a $\Gamma_F$ action $\Gamma_F \to \mathrm{Out}(\hat{G}) \subset \mathrm{Aut}(\hat{G})$.  We let $\LG$ denote the semidirect product $\hat{G} \rtimes \Gamma_F$.

We regard $\hat{G}$ and $\LG$ as algebraic groups over $k$ (the latter with an infinite component group), and denote their groups of $k$-points by $\hat{G}(k)$ and $\LG(k)$.

If $\alpha$ is a root in $\Psi(\G)$, it determines a coroot for $\hat{G}$. We will use the notation $\alpha_*$ for this coroot (although
a couple of times we will abuse notation and drop the subscript). Similarly, if $\alpha^{\vee}$ is a coroot for $\Psi(\G)$, then
we use $\alpha^{\vee}_*$ for the associated root in $\hat{G}$.

 Suppose $\HH$ is another algebraic group with $L$-group $\LH$.  As $k$ is algebraically closed and $\hat{H},\hat{G}$ are scheme-theoretically reduced, any algebraic morphism $\LH \to \LG$ is determined by its induced morphism on $k$-points $\LH(k) \to \LG(k)$.  For our purposes (Langlands functoriality), we may therefore usually ignore the difference between them.  (Note however that, as $k$ has positive characteristic, the map $\LH \to \LG$ can have nontrivial fibers as a map of schemes and yet induce an injection, or even an isomorphism, $\LH(k) \to \LG(k)$.)

\subsection{Local $L$-groups}
\label{subsec:llg}
In the construction of \S\ref{subsec:dgalg}, we may replace $F$ by $F_v$ for any finite place $v$, producing a group $\hat{G} \rtimes \Gamma_{F_v}$.  When $\G$ splits over an unramified extension of $F_v$, we work with the smaller group $\LG_v := \hat{G} \rtimes  \langle \Frob_v\rangle$.  Here $\langle \Frob_v\rangle$ denotes the discrete infinite cyclic group that topologically generates the unramified quotient of $\Gamma_{F_v}$.

\subsection{Parabolics and Levis in $\GG$ and $\hat{G}$}
\label{parabelly}
 
A $\GG(\overline{F})$-conjugacy class of parabolic subgroups $\PP \subset \GG \times_F \overline{F}$ distinguishes a subset $\Delta_P$ of the simple roots of $\Psi(\GG)$: the set of $\alpha$ for which $-\alpha$ is a root of $\PP$ (for a conjugation action of $\Tcan$ induced by  an arbitrary Borel subgroup $\BB \subset \PP$ and splitting of $\BB \to \Tcan$, all defined over $\overline{F}$).

We define a corresponding Levi subgroup $\hat{L}$ of $\hat{G}$ --- the subgroup generated by $\hat{T}$ and the coroot homomorphisms $\SL_2 \to \hat{G}$ corresponding to $\alpha \in \Delta_P$.

Then the abelianization $\mathbf{L}^{\ab}$ of the quotient Levi  $\mathbf{L}$ of $\PP$ is dual, as a torus, to the center of $\hat{L}$
-- with notations as above, the character group of $\mathbf{L}^{\ab}$ is identified with $X^*(\Tcan)/ \langle \alpha^{\vee}: \alpha \in \Delta_P \rangle$, whereas
the co-character group of $Z(\hat{L})$ is identified with the orthogonal complement
of $\{\alpha^{\vee}_*: \alpha \in \Delta_P \}$ in $
X_*(\hat{T}) $.

\subsection{Frobenius maps}
We will use   the notation  ``Frob'' for the Frobenius endomorphism of any $k$-group scheme
equipped with a descent to  $\mathbf{F}_p$. 
This applies, in particular, to any group of the form $\hat{G}$; the fixed
point subgroup of $\Frob$ is exactly the discrete set of $\mathbf{F}_p$-points
of $\hat{G}$.

\subsection{Class field theory}
\label{subsec:cft} \label{chartori}
We'll recall  part of the Langlands correspondence for tori \cite{Langlandstori}. 
 
Let $\TT$ be an algebraic torus over a number field $F$.  We say $\TT$ is unramified at $v$ if $T_v := \TT(F_v)$ splits over an unramified extension of $F_v$.  A homomorphism $T_v \to k^*$ is unramified if it is trivial on the maximal compact subgroup of $T_v$.  An idele class character $\TT(F) \backslash \TT(\Afinite) \to k^*$ is said to be unramified at $v$ if its restriction to $T_v$ is unramified.

Set $\LT = \hat{T} \rtimes \Gamma_F$ as in \S\ref{subsec:dgalg}.  When $\TT$ is unramified at $v$ set $\LT_v = \hat{T} \rtimes \langle \Frob_v\rangle$ as in \S\ref{subsec:llg}.  Let $A_v$ denote the 
$F_v$-points of the maximal split subtorus $\mathbf{A}_v \subset \mathbf{T}$. (We leave in the subscript $v$ because this depends on the place $v$). Then when $\TT$ is unramified at $v$ restriction gives an isomorphism (see \cite[\S 9.5]{Borel}):
\begin{equation} \label{res-iso}
\{\text{unramified characters of $T_v$}\} \stackrel{\sim}{\to} \{\text{unramified characters of $A_v$}\}
\end{equation}
There is also a natural surjective homomorphism $\hat{T} \to \hat{A}_v$ that identifies $\hat{A}_v$ with the $\Frob_v$-coinvariants of $\hat{T}$. 
Therefore, to an unramified character of $T_v$ is associated an element of the coinvariants $\hat{T}_{\Frob_v}$;  
put another way, this gives a natural bijection between unramified characters of $T_v$ and conjugacy classes of splittings $\langle \Frob_v\rangle \to \LT_v$.  
 The element of $\hat{T}_{\Frob_v}$ associated to an unramified character $\chi$ will be called the ``Langlands parameter'' of $\chi$.
 
We can describe this more directly: Since $X^*(\TT)^{\Frob_v} \rightarrow X^*(\AA)$ is an injection with finite cokernel, and $k^*$ is divisible,
we obtain a surjection $X^*(\TT)^{\Frob_v} \otimes k^* \twoheadrightarrow X^*(\AA) \otimes k^*$.  
Thus each  element of $X^*(\TT)^{\Frob_v} \otimes k^*$ gives an unramified character $\chi$, and every unramified character $\chi$ arises thus, although
possibly not uniquely.  Explicitly, if $\alpha \in X^*(\TT)^{\Frob_v}$ and $\lambda \in k^*$, the unramified
character associated to $\alpha \otimes \lambda$ is given by $t \in T_v \mapsto \lambda^{v (\alpha(t))}$.

One obtains the parameter of $\chi$ via the maps 
\begin{equation}   \label{XT} X^*(\TT) \otimes k^* =  X_*(\hat{T}) \otimes k^* = \hat{T}(k)\end{equation} 
 (i.e.,   the character $\chi$
is parameterized by the splitting  $\Frob_v \mapsto t_{\chi} \Frob_v$,
where $t_{\chi} \in \hat{T}(k)$ is the element thus produced).

\subsection{Hecke algebras}
\label{subsec:Heckealgebras}
Let $G$ be a locally compact, totally disconnected group.  If $S$ is a discrete set with a continuous left $G$-action and compact stabilizers,  let $\Fun_G(S\times S)$ (pronounced ``funguses'') denote the set of $k$-valued functions on $S \times S$ that are invariant for the diagonal action of $G$, and whose support is a union of finitely many $G$-orbits.  $\Fun_G(S\times S)$ has an algebra structure with multiplication given by
\begin{equation}
\label{eq:heckemult}
(h_1 \ast h_2)(x,z) = \sum_{y \in S} h_1(x,y)h_2(y,z)
\end{equation}
If $k[S]$ denotes the vector space spanned by $S$, there is a left action of $\Fun_G(S \times S)$ on $k[S]$ given by
\[
h \ast s = \sum_{t \in S} h(s,t) t
\]
If $S$ has finitely many $G$-orbits, $\Fun_G(S \times S)$ has a two-sided unit and the action on $k[S]$ identifies $\Fun_G(S \times S)$ with the ring of $G$-endomorphisms of $k[S]$.

The standard example is when $K \subset G$ is an open compact subgroup and $S = G/K$.  In that case, $\Fun_G(G/K \times G/K)$ can be identified with finitely supported functions on the double coset space $K \backslash G /K$, via $h(K,gK) = h(KgK)$.  We abbreviate this case by $\cH(G,K)$.   We will also use the notation $\cH(G, K; \mathbf{F}_q)$ for that subalgebra of $\cH(G, K)$ consisting
of functions valued in $\mathbf{F}_q \subset k$. 

 The theories of left- and right- $\cH(G,K)$-modules are equivalent via the anti-involution $KgK \leftrightarrow Kg^{-1}K$; nevertheless we wish to record some explicit formulas for these actions with some attention paid to the difference between left and right: 

The identification $V^K = \Hom_G(k[G/K],V)$ gives the $K$-invariants of a left $G$-module the structure of a right $\cH(G,K)$-module.  When $V$ is a left $G$-module, an explicit formula for this action is
\begin{equation}
\label{eq:vivaldi}
v \ast h = \sum_{gK \in G /K} g^{-1} v h(K,gK)\end{equation}
When $X$ is a set with a right $G$-action, the $k$-vector space $k[X]$ spanned by $X$ carries a left $G$-module structure extending linearly $g \cdot x = xg^{-1}$.  Then \eqref{eq:vivaldi} specializes to the following right $\cH(G,K)$-action on $k[X/K] \simeq k[X]^K$:
\begin{equation}
\label{eq:Heckerightact}
xK \ast h = \sum_{gK \in G/K} xgK \, h (K,gK) \qquad \text{ for $xK \in X/K$}
\end{equation}

\subsection{Hecke actions on homology and cohomology} \label{subsec:Heckeactionsonhomology}
Suppose that:
\begin{itemize}
\item[(i)] $G$, as in \S\ref{subsec:Heckealgebras}, is a locally compact, totally disconnected group, and $K \subset G$ an open compact subgroup.
\item[(ii)] $X$ is a  locally compact Hausdorff topological space with   continuous right $G$-action, such that the restriction of this action to $K$
is free and proper; 
  \end{itemize}
 
 Note that the assumptions force $K$ to be profinite.  Let $K_i \subset K $ be a collection of open normal subgroups with $\bigcap_i K_i = \{e\}$; then the natural map 
\begin{equation} \label{homeo} \pi: X \rightarrow \varprojlim X/K_n \end{equation} 
is a homeomorphism. In fact, it is easily verified to be a continuous bijection.
Now, we need to check that the image of any 
closed set  $Z \subset X$ is also closed in $\varprojlim X/K_n$.
Choose $y \notin \pi(Z)$. We want an open set containing $y$ and disjoint from $\pi(Z)$. 
We may find a  compact neighbourhood $A$ of $y$ in $X/K$ with  $\pi^{-1}A$ also compact.
  Then $\pi(Z) \cap A = \pi(Z \cap \pi^{-1}(A))$ is compact inside $\varprojlim X/K_n$.
Thus there is an
  open set $N \ni y$ that's disjoint from $\pi(Z ) \cap A$, and then
$A^{\circ} \cap N$ is the required open set.

\medskip 

 Then there is a right action of $\cH(G,K)$ on the $k$-homology of $X/K$ and a left action on the $k$-cohomology of $X/K$.  
We give an explicit construction of it:
 
 The set of singular $m$-simplices  $\Hom(\Delta^m,X)$ has a right $G$-action.  Since the $K$-action is free on $X$ it is free also on $\Hom(\Delta^m,X)$ and the quotient is naturally identified with the set of singular $m$-simplices in $X/K$: Each map $X/K_n \rightarrow X/K$ is a covering space %
 and  by \eqref{homeo} each $m$-simplex in $X/K$ lifts to $ X$
uniquely up to $\varprojlim K/K_n = K$. 
 Thus, we get a right action of $\cH(G,K)$ on \[
C_m(X/K) = k[\Hom(\Delta^m,X)/K] = k[\Hom(\Delta^m, X)]^K
\] for each $m$ by applying the discussion around \eqref{eq:Heckerightact} with $S = \Hom(\Delta^m,X)$.  As the face maps are maps of $\cH(G,K)$-modules, the action descends to a right action of $\cH(G,K)$ on $H_m(X/K;k)$.  
The identification of $H^m(X/K)$ with the dual vector space to $H_m(X/K)$ gives it a left $\cH(G,K)$-action.

\subsection{Borel-de Siebenthal theory}
\label{DTT1}
The section title refers to the general structure theory for semisimple subgroups of $\GG$ of the same rank as $\GG$.  In characteristic $0$, such subgroups are the fixed points of inner automorphisms of $\GG$.  We will recall that aspect of the theory that is important for us, when the inner automorphism has prime order $p$.

Let $\GG$ be a simply-connected algebraic group over $\overline{F}$, and let $\BB_G \supset \TT$ be a maximal torus and Borel in $\GG$, and let $\Delta(\GG) = \{\alpha_i\}$ denote the corresponding set of simple roots.  Let $\alpha_0$ be the highest root.  Let $\omega_i^{\vee}: \mathbb{G}_m \rightarrow \TT/\Cent(\G)$
denote the $i$th fundamental weight, so that  the coefficient of $\alpha_i$ in $\alpha_0$ is $\langle \omega_i^\vee,\alpha_0 \rangle$.

\begin{theorem*}
Suppose $\langle \omega_i^\vee,\alpha_0\rangle = p$.  Then there is a unique connected and semisimple subgroup $\HH \subset \GG$, containing $\TT$, with the following properties:
\begin{enumerate}
\item $-\alpha_0$ and $\alpha_j$ are roots of $\TT$ on $\HH$.  Meanwhile, $\alpha_i$ is not a root of $\HH$.
\item The set $\Delta(\HH) = \{-\alpha_0\} \cup \{\alpha_j\}_{j \neq i}$ is a system of simple roots for $\HH$, which can be realized as a subset of the extended Dynkin diagram $\tilde{\Delta}(\GG)$ by removing vertex $i$.
\end{enumerate}
\end{theorem*}

Recall the method of proof (\cite{BdS}, \cite[\S 2]{Reeder}):    One sees that if $\HH$ exists, the quotient $\Cent(\HH)/\Cent(\G)$ must have an element of order $p$ represented by  the image under $\omega_i^\vee$ of a primitive $p$th root of unity $\zeta$.  One then constructs $\HH$ as the centralizer of a representative $s \in \TT$ for $\omega_i^\vee(\zeta)$.

The automorphism group of the extended Dynkin diagram of $\widetilde{\Delta}(\G)$ permutes the $\omega_i^\vee$ with $\langle \omega_i^\vee,\alpha_0\rangle = p$.  If $\omega_i^\vee$ and $\omega_j^\vee$ are in the same orbit, the corresponding subgroups $\HH$ are $\GG(\overline{F})$-conjugate.  This provides a complete list of $\G(\overline{F})$-conjugacy classes of \emph{inner} automorphisms of order $p$ whose fixed points are semisimple.

\subsection{Endoscopic subgroups}
\label{sec:DTT2}
If $\hat{G}$ is an algebraic group of low characteristic (for us, the dual group of a simply connected group $\GG$), there are some exceptional semisimple subgroups of $\hat{G}$, of full rank, not captured by Borel-de Siebenthal.  Because they arise in duality with the centralizers of $\GG$ it is tempting to call such subgroups ``endoscopic subgroups.''\footnote{But note that in our story these groups appear on the Galois side.  In the theory of endoscopy they appear on the automorphic side, where they are not always subgroups.}  In this section we recall a parametrization of these subgroups parallel to that of \S\ref{DTT1}. 

Let $\hat{T}_G \subset \hat{B}_G \subset \hat{G}$ be the pinned maximal torus and Borel subgroup in $\hat{G}$ over $k$, as in \S\ref{subsec:dgalg}.  Let $\Delta(\hat{G}) = \{\alpha_{i,*}^\vee\}$ denote the corresponding set of simple roots, and let $\alpha_{0,*}^\vee$ denote the highest \emph{short} root.   (We are following the notation for roots of dual groups
discussed in \S \ref{subsec:dgalg}). 
Let $\alpha_{i,*}$ and $\alpha_{0,*}$ denote the associated coroots to the roots $\alpha_{i,*}^\vee$ and $\alpha_{0,*}^\vee$ --- thus $\alpha_{0,*}^\vee$ is the highest coroot of $\hat{G}$.  Let $\langle \omega^\vee_{i,*},\alpha_{0,*}\rangle$ denote the coefficient of $\alpha_{i,*}$ in $\alpha_{0,*}$.

\begin{theorem*}\label{DTT2}
Suppose $\langle \omega_{i,*}^\vee,\alpha_{0,*} \rangle = p$.  Then there is a unique connected and semisimple subgroup
$\hat{H}_0$ of $\hat{G}$, containing $\hat{T}_G$, with the following properties:
\begin{enumerate}
\item $-\alpha_{0,*}^\vee$ and $\alpha_{j,*}^\vee$ for $j \neq i$ are roots of $\hat{T}_G$ on $\hat{H}_0$.  Meanwhile $\alpha_{i,*}^\vee$ is not a root of $\hat{H}_0$.
\item The set $\Delta(\hat{H}_0) := \{-\alpha_{0,*}^\vee\} \cup \{\alpha_{j,*}^\vee\}_{j \neq i}$ make a system of simple roots for $\hat{H}_0$, which can be realized as a subset of the extended Dynkin diagram of $\hat{G}$ by removing $\alpha_{i,*}^\vee$ from $\widetilde{\Delta}(\hat{G})$.
 \end{enumerate}
\end{theorem*}

A case-by-case proof is given in \cite[Theorem 3.3]{T} but it does not give as much precision on the root systems, so we sketch a proof now: 

 \proof 

To specify a smooth subgroup scheme of $\hat{G}$ or $\GG$ we must specify a ``quasi-closed'' subset $S$ of roots  for $\GG$ or  $\hat{G}$   --- see \cite[\S 3]{BorelTits} 
but also \cite[Corollary 5]{Sopkina} for a reference in a form more convenient to us.  Let us for simplicity suppose $\GG$ does not have type $\mathrm{G}_2$, though that case can be verified either directly or by a similar argument.
If $\GG$ does not have type $\mathrm{G}_2$, then ``quasi-closed'' means (loc. cit.) the following: If $\alpha, \beta \in S$ and $i, j  > 0$ are such that $i \alpha + j \beta$ is a root for $G$, and furthermore the commutator coefficient $N_{\alpha\beta,ij}$ (in the notation of \cite{Sopkina}) is not divisible by $p$, then in fact $i \alpha + j \beta \in S$.  

In view of the computations
of \cite{VP} we may rephrase this condition  for symmetric $S$ (i.e. $S=-S$) as follows: if $\alpha$ and $\beta$ are in $S$ and $\Psi$ is the root system they span, then either $\Psi \cap S = \Psi$, or $\Psi \cap S$ is the set of long roots of $\Psi$, or else $\Psi \cap S$ is the set of short roots of $\Psi$.  The last possibility occurs only when $p = 2$.

To prove the Theorem, we prove that if $S$ is the quasi-closed set of roots that arise from $\HH$ as in Theorem \ref{DTT1}, the set $\{\alpha^\vee_*:\alpha \in S\}$ is also quasi-closed.  As $\HH$ is reductive, $S$ is symmetric, and according to the criterion for symmetric $S$ this is automatic if $\GG$ is simply laced, or $p = 2$.  The only remaining cases are $\mathrm{F}_4$ and $\mathrm{G}_2$ with $p = 3$, which can be verified by hand.
 \qed

 A more uniform ``Tannakian''  approach is proposed  in loc. cit., but to carry out the proposal requires a better understanding of Smith theory for perverse sheaves \cite[Conjecture 4.18]{T}.

\section{Tate cohomology} \label{sec:Tate}

\subsection{Definition of Tate cohomology}
\label{subsec:deftate}
Let $M$ be an abelian group with an action of the cyclic group $\langle \sigma \rangle$ of order $p$ generated by $\sigma$.  Set $\Tate^0(M) := \ker(1-\sigma)/\mathrm{Im}(N)$ and $\Tate^1(M) := \ker(N)/\mathrm{Im}(1-\sigma)$, where $N = 1+\sigma + \cdots + \sigma^{p-1}$ is the ``norm.''
In other words, $\Tate^i(M)$ is the cohomology of the 2-periodic chain complex whose differentials alternate between $1-\sigma$ and $N$.  Because of this each short exact sequence of $\sigma$-modules induces a long exact sequence
\[
\Tate^0(M') \to \Tate^0(M) \to \Tate^0(M'') \to \Tate^1(M') \to \Tate^1(M) \to \Tate^1(M'') \to \Tate^0(M')
\]

\subsection{Tate on smooth functions on $\ell$-spaces}
\label{subsec:Ccinfty}

If $X$ is a Hausdorff, locally compact, totally disconnected space (an ``$\ell$-space'', in the terminology of \cite{BZ}) write $C^{\infty}_c(X;\WittRing)$ or $C^{\infty}_c(X;k)$ for the space of $\WittRing$- or $k$-valued functions on $X$ that are locally constant (``smooth'') and compactly supported.  If $\sigma$ acts continuously on $X$ then we may form $\Tate^i(C^{\infty}_c(X;?))$.  These groups can be computed in terms of $X^{\sigma}$, as follows:
\begin{enumerate}
\item For $i=0$ or $1$, restricting to fixed points descends to an isomorphism
\[
\Tate^i(C_c^{\infty}(X;k)) \stackrel{\sim}{\to} C^\infty_c(X^{\sigma};k)
\]
\item The map
\[
\Tate^0(C^{\infty}_c(X;\WittRing)) \to C^{\infty}_c(X^{\sigma};k)
\]
give by restricting to fixed points and reducing mod $p$ (in either order) is an isomorphism, while $\Tate^1(C^{\infty}_c(X;\WittRing)) = 0$.
\end{enumerate}
We are going to prove a more general result in the Proposition below.

\subsection{Tate on sheaves on $\ell$-spaces}
\label{subsec:Tateonsheaves}
Let $X$ be as in \S\ref{subsec:Ccinfty}, and let $\cF$ be a sheaf of $k$- or $\WittRing$-modules on $X$.  Write $\Gamma_c(X;\cF)$ for the space of compactly supported sections of $\cF$.  For instance, if $\cF$ is the constant sheaf with stalk $k$ or $\WittRing$ then $\Gamma_c(X;\cF) = C^{\infty}_c(X;k)$ or $C^{\infty}_c(X;\WittRing)$.  The assignment $\cF \mapsto \Gamma_c(X,\cF)$ is a covariant exact functor \cite[\S 1.3]{Bernstein}. 

If $\sigma$ acts on $X$ and $\cF$ is $\sigma$-equivariant then $\sigma$ can be regarded as a map of sheaves
$\cF\vert_{X^{\sigma}} \stackrel{\sigma}{\to} \cF\vert_{X^{\sigma}}$ and we may define
\[
\begin{array}{c}
\Tate^0(\cF\vert_{X^{\sigma}}) = \ker(1-\sigma)/\mathrm{Im}(N)\ \  \mbox{(a sheaf on $X^{\sigma}$)}\\
\Tate^1(\cF\vert_{X^{\sigma}}) = \ker(N)/\mathrm{Im}(1-\sigma)   \ \ \mbox{(a sheaf on $X^{\sigma}$)}
\end{array}
\]
A compactly supported section of $\cF$ can be restricted to a compactly supported section of $\cF\vert_{X^{\sigma}}$. This map preserves the $\sigma$-actions inducing a map
\begin{equation}
\label{eq:steve}
\Tate^i(\Gamma_c(X;\cF)) \to \Gamma_c(X^{\sigma};\Tate^i(\cF))
\end{equation}

\noindent
{\bf Proposition.}  {\it The map \eqref{eq:steve} is an isomorphism.}

\medskip 
In the proof we'll use ``existence of fundamental domains:'' if $\sigma$ acts freely on a compact $\ell$-space $X$, there is a fundamental domain, i.e.
a closed and open subset $F \subset X$ so that $X$ is the disjoint union
of $\sigma^i F$ for $0 \leq i \leq p-1$. Indeed, take a cover
of $X$ by finitely many closed-and-open sets $U_i$ so that
$\sigma(U_i) \cap U_i = \emptyset$, take the algebra of sets
generated (under intersection and complement)  by the $U_i$ and their images under $\langle \sigma \rangle$. The minimal nonempty
elements of that algebra give a finite \sigmastable partition of $X$,
each block of the partition being disjoint from its $\sigma$-image. The desired domain now follows
by taking representatives for the $\sigma$-orbits on the blocks.

\begin{proof}
There is a short exact sequence (cf. \cite[1.16]{BZ}) 
\[
0 \to \cF' \to \cF \to \cF'' \to 0
\]
of $\sigma$-equivariant sheaves where $\cF''$ is supported on $X^{\sigma}$ and the stalks of $\cF'$ vanishes along $X^{\sigma}$.  It suffices to treat the cases $\cF = \cF''$ and $\cF = \cF'$.  For $\cF''$, the Proposition follows from the exactness of $\Gamma_c$ applied to the 2-periodic chain complex defining $\Tate^*$. 
So it is enough to check vanishing of $\Tate^i(\Gamma_c(U;\cF'))$  when $U^{\sigma}$ is empty.

We check for $\Tate^0$, the other case being similar. 
If $f \in \Gamma_c(U,\cF')$ is $\sigma$-invariant, choose a $\sigma$-invariant compact
$U' \subset U$ containing the support of $f$, take a fundamental domain $F' \subset U'$
for the $\sigma$-action, and note 
$f = \mathrm{N} (f'),$ where $f'$ is a section which agrees with $f$ on $F'$ and is zero off $F'$.  \end{proof}
 
 \subsection{Tate on rings}
\label{Taterings}
There is an algebraic relative of \S \ref{subsec:Ccinfty}:

If $A$ is any commutative unital $k$-algebra with a $\sigma$-action,   set $\bar{A}  = \Tate^0 A = A^{\sigma}/NA$.
It has a {\em ring} structure because $N \cdot A$ is an ideal in $A^{\sigma}$. 
Then we have a bijection
\begin{equation} \label{Tateforrings}  \Hom(A, k)^{\sigma} \stackrel{\sim}{\longrightarrow} \Hom(\bar{A}, k). \end{equation}
defined by restriction of characters.

\begin{proof} 
For short, let us say an ``extension'' of a character $\psi: \bar{A} \rightarrow k$
is a character $\chi: A \rightarrow k$ such that $\chi$, when restricted to $A^{\sigma}$,
factors as $A^{\sigma} \rightarrow \bar{A} \stackrel{\psi}{\rightarrow} k$.

There is a ring homomorphism $ A \rightarrow \bar{A}$ given by 
$a \mapsto (a a^{\sigma} \dots a^{\sigma^{p-1}})$.
When we restrict this to $A^{\sigma}$, it gives the $p$th power of the tautological map $A^{\sigma} \rightarrow \bar{A}$.

 Thus  given a character $\chi$ on $\bar{A}$ the formula
$$ \tilde{\chi}(a) =  \chi(a a^{\sigma} \dots a^{\sigma^{p-1}})^{1/p}$$
defines an extension to $A$.  This is the unique extension of $\chi$ to $A$: it is clearly the only possible $\sigma$-fixed extension, and in fact any extension of $\chi$ to $A$ must be $\sigma$-fixed.  To see this, note that as $\chi$ is trivial on $NA$ we must have $\sum_{i = 0}^{p-1} \tilde{\chi}^{\sigma^i} = 0$ for any extension $\tilde{\chi}$.  By linear independence of characters \cite[Ch. V \S 6.1, Theorem 1]{Bourbaki},  it follows that $\tilde{\chi}^{\sigma^i} = \tilde{\chi}$, i.e. $\tilde{\chi}_1 \in \Hom(A,k)^{\sigma}$.  
\end{proof}

The proof shows, more generally, that given any commutative integral domain $B$, 
any homomorphism $ \chi: \bar{A} \rightarrow B$  has the property that $\chi^p$ extends uniquely to $A$
(same argument, replacing $B$ by its quotient field to invoke the linear independence of characters).

\section{The Brauer homomorphism} \label{sec:Brauer}

\subsection{$\sigma$-actions, $\sigma$-plain subgroups}
\label{subsec:sigmaactions}
Let $G$ be a locally compact, totally disconnected group, and $K \subset G$ an open compact subgroup. 
Suppose $\sigma$ acts on $G$ with $\sigma(K) = K$ and $\sigma^p = 1$.  Write $G^{\sigma}$ and $K^{\sigma}$ for the fixed subgroups.

If $X$ is a right $G$-space on which $K$-acts freely, then $\cH(G,K)$ acts on the chains and cochains of $X/K$ and $\cH(G^{\sigma},K^{\sigma})$ acts on the chains and cochains of $X^{\sigma}/K^{\sigma}$.  Note the difference between $(X/K)^{\sigma}$ and its subspace $X^{\sigma}/K^{\sigma}$ --- the former usually does not carry a $\cH(G^{\sigma},K^{\sigma})$-action.

We wish to relate the Hecke modules $H^*(X/K)$ and $H^*(X^{\sigma}/K^{\sigma})$.  The relationship becomes much simpler under a technical hypothesis on $K$.  We say that $K \subset G$ is \emph{$\sigma$-plain} if  both of the following conditions hold:
\begin{itemize}
\item[(a)]  the inclusion
\[
G^{\sigma}/K^{\sigma} \hookrightarrow (G/K)^{\sigma}:gK^{\sigma} \mapsto gK
\]
is a bijection, or equivalently if $G^{\sigma}$ acts on $(G/K)^{\sigma}$ with a single orbit,
or equivalently if  $H^1(\sigma, K) \rightarrow H^1(\sigma, G)$ has trivial fiber above the trivial class. 
\item[(b)] $K$ is virtually prime-to-$p$, i.e. there is a finite index subgroup $K' \subset K$  which is a projective limit of prime-to-$p$ finite groups. 
In particular  $H^1(\sigma, K)$ is finite (\S \ref{sec:NAC}). 
\end{itemize}

\subsection{The Brauer homomorphism}
\label{subsec:brauerhom}
Suppose $\sigma$ acts on $G$ and $K$ as in \S\ref{subsec:sigmaactions}.  
The algebra $\cH(G,K)$ has an action of $\sigma$ (i.e.
$h^{\sigma}(x^{\sigma} ) = h(x)$ for $x \in G/K \times G/K)$. 
Write $\cH(G,K)^{\sigma}$ for the $\sigma$-invariant part of $\cH(G,K)$.  If $K$ is $\sigma$-plain, then the \emph{Brauer homomorphism} is the map
\[
\Br:\cH(G,K)^{\sigma} \to \cH(G^{\sigma},K^{\sigma})
\]
just given by restricting $h$ from $ G/K \times G/K$ to $((G/K)^{\sigma} \times (G/K)^{\sigma}) = (G^{\sigma}/K^{\sigma}) \times G^{\sigma}/K^{\sigma}$.  Since the summands of \eqref{eq:heckemult} are invariant under the action of $\sigma$ on $G/K$, and $k$ has characteristic $p$, the Brauer map is an algebra homomorphism.  A similar construction is called the ``Brauer homomorphism'' in the modular representation theory of finite groups, and we call it by the same name here.

Set $N = 1+\sigma + \sigma^2 + \cdots + \sigma^{p-1}$, i.e. $N$ is the ``norm element'' in the group ring $k[\sigma]$.  Then $N\cdot(1-\sigma) = (1-\sigma)\cdot N =0$.  If $\sigma$ acts on a set $S$, 
and we thereby regard $k[S]$ as a $k[\sigma]$-module, there are canonical identifications
\begin{equation}
\label{eq:obama}
\ker(1-\sigma)/\mathrm{Im}(N) = k[S^\sigma], \qquad \ker(N)/\mathrm{Im}(1-\sigma) = k[S^\sigma]
\end{equation}
 Note these are the groups $\Tate^i$ from \S\ref{subsec:deftate}. The identification on the left sends $s \in S^{\sigma}$ to $s + \mathrm{Im}(N)$ and on the right it sends $s \in S^{\sigma}$ to $s + \mathrm{Im}(1-\sigma)$.  The Brauer homomorphism is compatible with these identifications in a sense we now describe (see also \eqref{Brauercompat} for a more general statement):

Suppose that $\tilde{S}$ is a right $G$-set with compatible $\sigma$-action such that $K$ acts freely,  $S = \widetilde{S}/K$  and $h \in \cH(G,K)^{\sigma}$.  
 Then the map $i:\tilde{S}^{\sigma}/K^{\sigma} \rightarrow S^{\sigma}$ is injective
as the $K$ action is free.    
Moreover, $\cH(G, K)^{\sigma}$ acts on $k[S^{\sigma}] \simeq \Tate^0 k[S]$ by means of \eqref{eq:obama},
i.e. each $h \in \cH$ acts via $\Tate^0 h$.  

We claim that, in fact, $k[\tilde{S}^{\sigma}/K^{\sigma}]$ is a $\cH(G, K)^{\sigma}$-direct
summand of $k[S^{\sigma}]$, and 
 \begin{equation} \label{Brauercompat0} \mbox{  The action of $\Tate^0 h$
 for $h \in \cH(G, K)^{\sigma}$ on $k[\tilde{S}^{\sigma}/K^{\sigma}]$
 coincides with 
 $\Br(h) \in  \cH(G^{\sigma}, K^{\sigma})$. }\end{equation} 
Moreover, the identical statement holds also for $\Tate^1$. 

 It is enough to show the same statement with ``direct summand'' replaced by ``submodule''; then one
 notes that the natural $k$-valued bilinear pairing on $k[S^{\sigma}]$  --- given by  $\langle \sum a_s s, \sum b_s s \rangle = \sum a_s b_s $ --- 
 has the property that $\langle a \ast \Tate^0 h , b \rangle = \langle a, b \ast \Tate^0 h' \rangle$ where $h \mapsto h'$ is the antiinvolution of the Hecke 
 algebra sending $K gK$ to $K g^{-1} K$.  Moreover, this bilinear pairing 
is  nondegenerate on $k[\tilde{S}^{\sigma}/K^{\sigma}]$. That shows that $k[\tilde{S}^{\sigma}/K^{\sigma}]$ is actually a summand. 
 
 Now to check \eqref{Brauercompat0}: 
For $\tilde{s} \in \tilde{S}^{\sigma}$ we have
$ \tilde{s} K^{\sigma} \ast \Br(h) = \sum_{g \in G^{\sigma}/K^{\sigma}} \tilde{s} g K^{\sigma} \, h(K, g K) $ and 
$ \tilde{s} K \ast h= \sum_{g \in G/K} \tilde{s} g  \, K  h(K, gK)$.
  Considered in $k[S]^{\sigma}$, these elements differ by 
 $$ \sum_{g \in G/K - (G^{\sigma}/K^{\sigma})} \tilde{s} g K \, h(K, gK)$$
 and our assumption that $K$ is $\sigma$-plain means that $\sigma$ acts
 freely on $G/K-(G^{\sigma}/K^{\sigma})$; in particular, the element above belongs
 to the image of $N$ from \eqref{eq:obama}.

\subsection{Normalized Brauer homomorphism} \label{nBr}
Suppose $\cH(G,K)$ and $\cH(G^{\sigma},K^{\sigma})$ are commutative integral domains.  Then according to \S\ref{Taterings}, 
the $p$th power of $\Br$ extends uniquely to a homomorphism
\[
\widetilde{\Br}:  \cH(G, K) \rightarrow \cH(G^{\sigma}, K^{\sigma}).
\]
This map is not $k$-linear but rather Frobenius semilinear. However, we may twist it to be linear:
$ \cH(G, K)$ has an $\mathbf{F}_p$-structure, i.e.
\[
\cH(G,K) = \Fun_G(G/K \times G/K; \mathbf{F}_p) \otimes_{\mathbf{F}_p} k
\]
 The {\em normalized Brauer homomorphism}, which we denote with a lower case ``b,'' is the unique $k$-linear homomorphism
\[ 
\NBr: \cH(G, K) \longrightarrow \cH(G^{\sigma}, K^{\sigma}),
\]
that agrees with $\widetilde{\Br}$ on $ \Fun_G(G/K \times G/K; \mathbf{F}_p)$.
An explicit formula for $\NBr$ is given by  
\begin{equation} 
\NBr(h)(K^{\sigma},gK^{\sigma}) =\left( \left(h \ast \stackrel{p}{\cdots} \ast h\right)(K,gK)\right)^{\frac{1}{p}}. \end{equation}

\begin{theorem}
\label{thm:smith} 
Let $G, K, X$ be as in \S \ref{subsec:Heckeactionsonhomology}.    Suppose that $\sigma$ acts compatibly on $G, K, X$,
so that $G^{\sigma}, K^{\sigma}, X^{\sigma}$ also satisfy the conditions of \S \ref{subsec:Heckeactionsonhomology}. 
Suppose in addition that $X/K$ has finite cohomological dimension, and that $K$ is $\sigma$-plain in the sense of \S\ref{subsec:sigmaactions}.
In this situation, as described above, 
\begin{eqnarray}
\label{eq:23}
H^*(X/K) & \text{\rm is a left $\cH(G,K)$-module} \\
\label{eq:24}
H^*(X^{\sigma}/K^{\sigma}) & \text{\rm is a left $\cH(G^{\sigma},K^{\sigma})$-module}
\end{eqnarray}

Then we have:
\begin{itemize}
\item[(a)] If we regard these as $\cH(G,K)^{\sigma}$-modules (via restriction for \eqref{eq:23} and via $\Br$ for \eqref{eq:24}), then every composition
factor of  \eqref{eq:24} is  also a composition factor of \eqref{eq:23}.   

\item[(b)] Suppose that we are in the setting of \S \ref{nBr} --- i.e., suppose that $\cH(G, K)$ and $\cH(G^{\sigma}, K^{\sigma})$ are both commutative integral domains --- so the normalized Brauer homomorphism $\NBr$ is defined. 
Suppose that $\chi: \cH(G^{\sigma}, K^{\sigma}) \rightarrow k$
is a character
that appears as  an eigenvalue of \eqref{eq:24} (i.e.
there exists an element of $H^*(X^{\sigma}/K^{\sigma})$ annihilated by all $h - \chi(h)$ for $h \in \cH(G^{\sigma}, K^{\sigma})$). 

Then also $\chi \circ \NBr$  appears as an eigenvalue of \eqref{eq:23}.  
\end{itemize}

  \end{theorem}

 \begin{proof}
 
 Note that (a) implies (b).  Indeed, suppose that $\chi$ is as in (b). 
By (a) we have 
\begin{equation}  \label{WWP} \mbox{ $\chi \circ \Br$ appears as an eigenvalue of \eqref{eq:23} when restricted
 to $\cH(G, K)^{\sigma}$. } \end{equation} 
We  have seen in \S \ref{Taterings} that $(\chi \circ \FBr)^{1/p}$
 is the unique extension of $\chi \circ \Br$ from $\cH(G, K)^{\sigma}$
 to $\cH(G, K)$. 
So, supposing \eqref{WWP},  we see  that $(\chi \circ \FBr)^{1/p}$
 appears as an eigenvalue of \eqref{eq:23}.   
 But this implies that $\chi \circ \NBr$ appears as an eigenvalue of \eqref{eq:23}: 
the isomorphisms
$$H^*(X/K) = H^*(X/K; \mathbf{F}_p) \otimes_{\mathbf{F}_p} k, \ \ \cH(G,K)  = \Fun_{G}(G/K \times G/K, \mathbf{F}_p) \otimes_{\mathbf{F}_p} k$$
yield semilinear actions of $\Aut(k)$ on $\cH(G,K)$ and $H^*(X/K)$. We have also
$$ \alpha^{\tau} (h^{\tau})  = \left( \alpha(h) \right)^{\tau}$$
for $\alpha \in \cH, h \in H^*(X/K), \tau \in \Aut(k)$.  So if $h \in H^*(X/K)$ corresponds to the eigenvalue
$(\chi \circ \FBr)^{1/p}$ then $h^\tau$ corresponds to the eigenvalue $\chi \circ \NBr$, where
$\tau$ is the Frobenius automorphism.

The proof of (a) is an application of ``fixed point localization'' methods of Smith, Borel, Quillen.  We give a treatment here that is well adapted to keeping track of the Hecke action.  The statement of the theorem for homology implies the statement for cohomology --- let us prove the homology version.

Consider the ``Smith double complex''
\[
C_*^{\mathrm{Smith}} := 
\left[
\xymatrix{
\cdots &C_*(X/K) \ar[l]_N &\ar[l]_{1-\sigma} C_*(X/K) &\ar[l]_N C_*(X/K) & \cdots \ar[l]_{1-\sigma}
}
\right]
\]
The map $\sigma$ is not a map of $\cH(G,K)$-modules but it is a map of $\cH(G,K)^{\sigma}$-modules, so $C_*^{\mathrm{Smith}}$ is a double complex of $\cH(G,K)^{\sigma}$-modules.  It leads to two spectral sequences of $\cH(G,K)^{\sigma}$-modules:
\begin{itemize}
\item The spectral sequence $\hvE$, in which the differential on the $0$th page is the horizontal differential, on the $1$st page is the vertical differential.  
\item The spectral sequence $\vhE$, in which  the differential on the $0$th page is the vertical differential and on the $1$st page is the horizontal differential.
\end{itemize}
 If $C_*$ is bounded, then both $\hvE$ and $\vhE$ converge to the homology of the total complex of $C_*^{\mathrm{Smith}}$.
Let us abbreviate the horizontal differential (which alternates between $1-\sigma$ or $N$)  by $d^h$ and the vertical differential (which is the standard singular differential on $C_*(X/K)$) by $d^v$.
We can compute higher differentials in these spectral sequences by the following standard device.  If $x \in \hvE^0_{ij}$ is an element that survives to $\hvE^r_{ij}$, and $(x_1,\ldots,x_r)$ is a sequence of elements with $x = x_1$ and $d^v(x_i) = d^h(x_{i+1})$ for $i < r$, then $d^v(x_r)$ is a representative for $d^r(x)$.

We complete the proof in three steps:

(1) {\it  Degeneration of $\hvE$.}
The $\hvE$ spectral sequence is analyzed as follows:

By equation \eqref{eq:obama} the first page $\hvE^1$ is naturally identified with $C_*((X/K)^{\sigma})$ i.e.
\[
\hvE^1_{ij} = C_j((X/K)^\sigma) \qquad d^1:C_j \to C_{j-1} \text{ is the singular differential}
\]
It follows that $\hvE^2_{ij} = H_j((X/K)^\sigma)$.  Now, if $\zeta \in \hvE^0_{ij} = C_j(X/K)$ has $d^0(\zeta) = 0$
then, by \eqref{eq:obama}, 
\[
\zeta = \zeta'  + d^0 \varepsilon
\]
for some $\zeta' \in C_j((X/K)^\sigma)$.
If $\zeta$ survives to $\hvE^2$, we must have, in addition, $d^v \zeta  \in \mathrm{Im}(d^0)$,
or equivalently $d^v \zeta' \in \mathrm{im}(d^0)$. 
But $d^v \zeta' \in C_{j-1}((X/K)^{\sigma})$; by another application of \eqref{eq:obama}, 
$d^v \zeta'$ is identically zero. 
In other words, every element of $\hvE^2$ is represented by a {\em cycle} $\zeta' \in C_{j}((X/K)^{\sigma})$.
Then  $(\zeta',d^v(\zeta') = 0, 0, 0, \ldots,0)$ is a sequence we may use to compute $d^r(\zeta') = 0$ for all $r \geq 2$.  Thus, $\hvE^2 = \hvE^\infty$. 

\medskip

(2) {\it  Compatibility with the Brauer homomorphism.}
 In other words, (1) shows that the homology of the total complex of $C_*^{\mathrm{Smith}}$ has a filtration (by $\cH(G,K)^\sigma$-submodules) whose associated graded is $H_*((X/K)^\sigma)$.  We claim that our assumptions imply that $X^{\sigma}/K^{\sigma}$ is a union of connected components of $(X/K)^{\sigma}$.   %

To prove the claim, let $Y \subset X$ be the inverse image of $(X/K)^{\sigma}$.  As $K$ acts freely, for each $y \in Y$ there is a unique $\kappa(y) \in K$ such that $y \kappa(y) = \sigma(y)$.  The map $\kappa$ is $K$-equivariant for the $\sigma$-twisted conjugation action on $K$.  The graph of $\kappa$ is the set of all $(y,k) \in Y \times K$ with $\sigma(y) = yk$ --- in particular it is a closed set.
The projection of this graph to $Y$ is a homeomorphism. 
So $\kappa$ is a continuous function
$Y \rightarrow K$. 
It descends to a continuous function $$Y/K = (X/K)^{\sigma} \to \mbox{$\sigma$-twisted conjugacy classes for $K$},$$
where we give the right-hand side the quotient topology.   But the space of $\sigma$-twisted conjugacy classes for $K$
is a finite set, because we assumed that $K$ is $\sigma$-plain, and because each $\sigma$-twisted conjugacy class is closed in $K$,
the topology on this finite set is the discrete topology.
 It follows that $X^{\sigma}/K^{\sigma}$ is a union of connected components of $(X/K)^{\sigma}$.
 
Thus, on the first page,  $\hvE^1_{ij} =C_i((X/K)^{\sigma})$ has $C_i(X^{\sigma}/K^{\sigma})$ as a vector space summand.  By equation \eqref{Brauercompat0} ---
applied with $\tilde{S}$ equal to the free $K$-set of singular $i$-simplices in $X$ ---   the action of $\cH(G,K)^{\sigma}$ on this summand factors through $\Br$
 and it is actually a $\cH(G, K)^{\sigma}$-submodule.  Passing to homology, we conclude that 
  $H_*(X^{\sigma}/K^{\sigma})$ is a $\cH(G,K)^{\sigma}$-submodule of $\hvE^2_{ij}$.

 \medskip
 
 (3) {\it  Convergence of $\vhE^r$}  In this last step, observe that $\vhE^1_{ij} = H_j(X/K)$ and that since $H_*(X/K)$ vanishes in large degrees, we have a convergent spectral sequence
\begin{equation}
\label{eq:vhE}
\vhE^1_{ij} = H_j(X/K) \implies H_*(\mathrm{Tot}(C^{\mathrm{Smith}}))
\end{equation}
of $\cH(G,K)^{\sigma}$-modules.
Therefore by (2), we obtain the desired statement: we have exhibited
$H_*(X^{\sigma}/K^{\sigma})$, 
as a composition factor of  $H_*(X/K)$, where both are regarded as modules under $\cH(G,K)^{\sigma}$. 
Indeed, even better: We can identify $\vhE^{2}_{ij}$ with the Tate cohomology 
$\Tate^i H_j(X/K)$;   and we have thus actually exhibited $  H_*(X^{\sigma}/K^{\sigma})$
as a subquotient of  $\Tate^* H_*(X/K)$. 
\end{proof}

\section{Cyclic group actions on locally symmetric spaces } \label{sec:cyclic}
\label{sec:lssam}

\begin{definition}
\label{def:langleGrangle}
Let $\G, K_{\infty}$ and level structure $K \subset \G(\Afinite)$ be as in \S\ref{subsec:agals}.  Let $[G]_K$ denote the double coset space 
\[
[G]_K := \G(F) \backslash \G(\Adele)/(K_\infty \times K)
\]
\end{definition}
If $K \cap \G(F)$ is torsion-free, the homology and cohomology of $[G]_K$ carry the action of the Hecke algebra $\cH(\G(\Afinite),K)$ described in  \S\ref{subsec:Heckeactionsonhomology}.  For general $K$, one should regard $[G]_K$ as an orbifold and take homology and cohomology in this sense, in which case a more careful discussion defines an action of $\cH$ as well.

The Hecke algebra $\cH(\G(\Afinite),K)$ is a restricted tensor product over finite places
\[
\cH(\G(\Afinite),K) = {\bigotimes_v}' \cH(G_v,K_v)
\]
where the restricted product is taken with respect to the identity element in $\cH(G_v, K_v)$. 
When $V$ is a set of finite places, we write $\cH(G_V,K_V) := \bigotimes'_{v \in V} \cH(G_v,K_v)$;
we sometimes abbreviate this to simply $\cH_V$. 
 
\subsection{Good places} \label{goodplace}
We call a place $v$ \emph{good} with respect to the algebraic group $\G$, level structure $K$ and prime $p$  if
\begin{enumerate}
\item[(i)] The residue characteristic of $\OO_v$ is not equal to $p$
\item[(ii)] $\G \times_F F_v$ is quasi-split over $F_v$ and split over an unramified extension of $F_v$.
\item[(iii)] $K_v$ is a hyperspecial subgroup of $G_v$.  In other words, $K_v$ is a maximal compact subgroup of the form $\mathfrak{G}(\OO_v)$, where $\mathfrak{G}$ is a reductive smooth model for $\mathfrak{G}\times_F F_v$ over $\OO_v$.
\end{enumerate}
For any $K$, all but finitely many places are good.
 At a good place, $\cH(G_v,K_v)$ is a commutative integral domain  and its characters are understood via the Satake isomorphism;
for this, see \S \ref{Satake}.

\subsection{Characters of the Hecke algebra occurring on cohomology}
  Suppose given a character $\chi$ of $\cH(G_V, K_V)$,  where $V$ is a set of finite good places. 
  We say ``$\chi$ appears in the cohomology
of $[G]_{K}$'' if there is $h \in H^*([G]_{K})$ such that $h$ transforms under $\cH(G_V, K_V)$ by 
$\chi$. 
 
 The following result shows that it is enough to consider ``sufficiently small'' level structures, in particular --- as long as $V$ excludes at least one finite place --- one may always assume that the relevant locally symmetric spaces are manifolds and not merely orbifolds.

\begin{prop}
\label{prop:isthisnec}
Suppose that $K = \prod K_v$ and $K' = \prod K'_v$ where
$K'_v \subset K_v$ for all $v$ with equality $K'_v=K_v$ for $v \in V$. 
 If $\chi$ appears in the cohomology of $[G]_{K}$
then it also  appears in the cohomology of $[G]_{K'}$. 
\end{prop}

\begin{proof}
 The finite group $K/K'$ acts on the cohomology of $[G]_{K'}$.  For all  $v\in V$ we have $K_v = K'_v$, and the actions of $\cH_v$ and $K/K'$ on $H^*([G]_{K'})$ commute.  The spectral sequence
\[
E_2^{ij} = H^i(K/K';H^j([G])_{K'}) \implies H^{i+j}([G]_K)
\]
is a spectral sequence of $\cH_V$-modules.  Thus, a character of $\cH_V$ that occurs in the cohomology of $[G]_K$ also occurs in $H^*(K/K';H^*([G])_{K'})$.  The bar model for the $K/K'$-cohomology of $H^*([G]_{K'})$ shows that the character mus appear in $H^*([G]_{K'})$  itself.  
\end{proof}

\subsection{$\sigma$-action} \label{sigmaax}
Now suppose that  $\sigma$ acts on $\G$ with order $p$, and set $\HH = \G^{\sigma}$.  Suppose that $\HH$ is connected. We may treat either $\G$ or $\HH$ as a special case of the setup of \S\ref{def:langleGrangle}, and we
  make the following parallel notations and assumptions:
\begin{itemize}

\item[(a)] Fix level structures $K$ for $\G$ and $U$ for $\HH$,  and suppose that $K$ is $\sigma$-stable with fixed
 points $U = K^{\sigma}$;
\item[(b)] Fix a maximal compact $K_{\infty} \subset \G(F_{\infty})$ in such a way that $K_{\infty}$ is $\sigma$-invariant
and $K_{\infty}$ intersects $\mathbf{H}(F_{\infty}$) in a maximal compact subgroup $U_{\infty}$.

 This is always possible: 
inside the disconnected group $\G(F_{\infty}) \rtimes \langle \sigma \rangle$, 
we may find a maximal compact subgroup that contains $U_{\infty} \times \langle \sigma \rangle$, and then
we just take its intersection with $\G(F_{\infty})$. 

  \item[(c)]Write $[G]_K = \G(F) \backslash \G(\mathbf{A}_F) / K_{\infty} K$ and $[H]_U = \HH(F) \backslash \HH(\mathbf{A}_F)/U_{\infty} U$.
  \item[(d)] $K$ is ``sufficiently small'', in that $\G(F) \cap K_{\infty} K$ is trivial; so also $\HH(F) \cap U_{\infty} U$ is trivial. By Proposition \ref{prop:isthisnec}, this will entail  no real loss of generality.

\end{itemize}

 \begin{prop} \label{gppd}  \label{subsec:sigmagood}
Say that a finite place $v$ is \emph{$\sigma$-good} with respect to $\G,K,\HH,U$ if
\begin{itemize}
\item[(a)] $v$ is good with respect to $K$ and $U$ in the sense of  \S \ref{goodplace}, and
\item[(b)] $K_v \subset G_v$ is a $\sigma$-plain subgroup in the sense of  \S \ref{subsec:sigmaactions}.
\end{itemize}

If $\HH$ is connected, then  all  but finitely many places of $F$ are $\sigma$-good. \end{prop}
 
 We remark that Brian Conrad and Gopal Prasad explained to us how to obtain a much sharper result
by reducing to a corresponding assertion for tori. 

\begin{proof}  
Let $\mathfrak{G}$ be a model of $\GG$ over $\OO_F$. 
We must check that the map $H^1(\sigma, \mathfrak{G}(\OO_v)) \rightarrow H^1(\sigma, G_v)$
has trivial fiber above the trivial class,  for almost all $v$ (it is easy to check the remaining conditions are valid for almost all $v$).

Consider the morphism of $\OO_F$-schemes $ g \mapsto g^{-1} \sigma(g)$ from $\mathfrak{G}$ to itself. 
Its image $I$ is constructible, i.e. a finite union of locally closed sets. On the other hand, it interects the generic fiber
of $\mathfrak{G}$ in a {\em closed} set $J$: in characteristic zero, the conjugacy class of $\sigma$ is closed ( \cite[Corollary 5.8]{Joyner}:
$\sigma$ is automatically semisimple, being of finite order).  Let $\overline{J}$ be the closure of $J$ inside $\mathfrak{G}$. The  symmetric difference $(\overline{J}   \backslash I) \cup (I  \backslash \overline{J}) $, considered as a subset of $\mathfrak{G}$, 
is a constructible set which does not intersect the generic fiber. The  projection 
of this symmetric difference to  $\Spec(\OO_F)$
is (being constructible and disjoint from the generic point)  a finite set of closed points.

 Let $T$ be the corresponding set of places, together with all
 places at which $\mathfrak{G}$ or $\mathfrak{G}^{\sigma}$ are not smooth and all places of residue characteristic dividing $p$. 
 In what follows, replace $\mathfrak{G}$ by its restriction to $\OO[\frac{1}{T}]$. Then, by choice of $T$,  
 the image of $g \mapsto g^{-1} \sigma(g)$ is a closed subset  $J'$ of $\mathfrak{G}$. 

 We claim that the claim holds for $v \notin T$. Indeed, 
 suppose given $g \in G_v$ with the property that $g^{-1} \sigma(g) \in \mathfrak{G}(\OO_v)$. 
 We need to verify that the $\OO_v$-scheme defined by
 $$ \mathfrak{X} = \{ x \in \mathfrak{G}:  \ x^{-1} \sigma(x) = g^{-1} \sigma(g) \}$$ 
has an $\OO_v$-point. 

Now $\mathfrak{X}$ has a $\overline{\mathbf{F}_v}$-point:  By assumption, 
$g^{-1} \sigma(g)$ yields a map $\Spec(\OO_v) \rightarrow \mathfrak{G}$
sending the generic point of $\Spec(\OO_v)$ to an element of $J'$. Because $J'$ is closed, 
the special point of $\Spec(\OO_v)$ is also sent to an element of $J'$, i.e. there
exists $y \in \mathfrak{G}(\overline{\mathbf{F}_v})$ with $y^{-1} \sigma(y) = g^{-1} \sigma(g)$
modulo $v$, as desired.

Therefore, $\mathfrak{X}$ also has a point over $\mathbf{F}_v$, because
$\mathfrak{X}(\overline{\mathbf{F}_v})$ is a torsor under $\HH(\overline{\mathbf{F}_v})$
and Steinberg's theorem  \cite[Theorem 1.9]{Steinberg} says that the Galois cohomology of the connected algebraic group $\mathbf{H}$ is trivial over the finite field $\mathbf{F}_v$. 

In other words there exists $x \in \mathfrak{G}(\OO_v)$ such that
 $$ (xg)^{-1} \sigma(gx) \in \Delta_v := \mathrm{ker} \left( \mathfrak{G}(\OO_v) \rightarrow \mathfrak{G}(\mathbf{F}_v) \right),$$
 i.e. it defines a class in $H^1(\sigma, \Delta_v)$. But $\Delta_v$ has pro-order that is relatively prime to $p$, so that class must vanish, i.e.
 there exists $\delta \in \Delta_v$ such that $(xg)^{-1} \sigma(gx) = \delta^{-1} \sigma(\delta)$. In other words,
 the class of $g^{-1} \sigma(g) = y^{-1} \sigma(y)$ where $y =\delta x^{-1} \in \mathfrak{G}(\OO_v)$, as desired. \end{proof}

\subsection{Analysis of connected components} \label{conn-analysis}

Let $V$ be any nonempty
finite  
  set of $\sigma$-good places.  
Write $K^{(V)} =K_{\infty} \prod_{w \notin V} K_w$.
We are going to apply the
discussion of
\S \ref{thm:smith},
with $$ X = \G(F) \backslash \G(\adele_F) / K^{(V)}$$ and the acting groups (``$G,K$'' from \S \ref{thm:smith})
\begin{eqnarray*}G_V  & = & \mbox{$G$ from \S \ref{thm:smith}} = \prod_{w \in V} \G(F_w), \\ 
 K_V  &=&   \mbox{$K$ from \S \ref{thm:smith}}   = \prod_{w \in V} K_w, \end{eqnarray*}

Since we assumed that   $\G(F) \cap K_{\infty} K$ is trivial, the group $K_V$ acts freely on $X$.  
The main issue is to precisely analyze how the fixed locus $X^{\sigma}$ is related to $\HH$.

\begin{prop*}
The natural map  $[H]_U  \rightarrow X^{\sigma}/K_V^{\sigma}$  maps $[H]_U$ homeomorphically onto a union of components of $X^{\sigma}/K_V^{\sigma}$.
 \end{prop*}
 \proof   There's  a map  \begin{equation} \label{Ymap2} \mathfrak{e}:  X^{\sigma}/K_V^{\sigma}  \rightarrow H^1(\sigma, \G(F)) \times H^1(\sigma, K^{(V)}) \end{equation}  constructed as follows:
For $g \in \G(\adele_F)$, the double coset $\G(F)g    K^{(V)}  \in X $ is $\sigma$-fixed if and only if one can find $\gamma \in \G(F)$ and $\kappa \in K^{(V)}$ such that 
$\sigma(g) = \gamma g \kappa$ inside $\G(\adele_F)$. 
Consider this equality at a place $w \in V$: it shows that actually
$\gamma =  \sigma(g_w) g_w^{-1}$, 
in particular, it satisfies   $\sigma^{p-1}(\gamma)  \dots  \sigma(\gamma) \gamma= e$. 
Then computing $\sigma^p(g)$ we see that also $\kappa \sigma(\kappa) \dots \sigma^{p-1}(\kappa)=e$. 
In other words $\sigma \mapsto \gamma^{-1}$ and $\sigma \mapsto \kappa$
define cocycles in $H^1(\sigma, \G(F))$ and $H^1(\sigma, K^{(V)})$; these classes depend only on the double coset. 

The map $\mathfrak{e}$ of \eqref{Ymap2} is  locally constant. In fact
choose $x \in X^{\sigma}$ and a representative $g \in \G(\adele_F)$.  
Let $U_{\infty}$ be a $\sigma$-fixed open  neighbourhood of $K_{\infty}$ inside $\G(F_{\infty})$ 
and let $U = U_{\infty} \cdot K$. 
 Suppose  $g$ is, as above, so that the double coset $\G(F) g K^{(V)}$ is $\sigma$-fixed, and $\gamma, \kappa$ are as above. Suppose that $g u$ also defines a $\sigma$-fixed element of $X$, i.e.
$$ \sigma(g)   \sigma(u)= \gamma'  g u \kappa' \implies \gamma g \kappa  \sigma(u) = \gamma' g u \kappa',$$
and in particular,
$$ \gamma  g   K \cap \gamma' g K \cdot  U_{\infty}  \cdot U_{\infty}^{-1}  \neq \emptyset.$$
Because the action of $\G(F)$ on $\G(\adele_F)/K_{\infty} K$ is   properly discontinuous
and free by assumption (d) of \S \ref{sigmaax}, 
this implies that $\gamma' =\gamma$ if $U_{\infty}$ is chosen sufficiently small.   (Recall that $K_{\infty}$ is chosen $\sigma$-invariant,
and so one may choose $U_{\infty}$ to be an arbitrarily small open neighbourhood of it.)

We also then have $\kappa =u \kappa' (\sigma(u))^{-1}$
and thus the corresponding classes in $H^1(\sigma, K^{(V)})$ are also equal. 
 Indeed, this is now clear for the projection to the latter component
 of $K^{(V)} \simeq K_{\infty} \times \prod_{w \notin V} K_w$;
 to handle the $K_{\infty}$ component  we observe that $H^1(\sigma, K_{\infty})$ is finite, and for each class in $H^1(\sigma, K_{\infty})$
the set of representing cocycles is closed; thus the induced topology on $H^1(\sigma, K_{\infty})$ is the discrete one. 
Thus, if  we take $U_{\infty}$ sufficiently small, the  classes  $H^1(\sigma, K_{\infty})$ corresponding to
 $\kappa_{\infty}$ and $\kappa'_{\infty}$ are then forced to be equal.

The natural $[H]_U \rightarrow X^{\sigma}/K_V^{\sigma}$ is  injective:
if the double cosets of $h,h' \in \mathbf{H}(\adele_F)$  map to the same point, 
we have $h = \gamma h' k$ with $ \gamma \in \G(F), k \in K^{(V)} K_V^{\sigma}$;
considering components at a place $w \in V$ we see that $h_w = \gamma h'_w k_w$, and in particular
$\gamma$ is $\sigma$-invariant; then $k$  too is $\sigma$-invariant and belongs to 
$(K_{\infty} K)^{\sigma} = U_{\infty} U$.    

Finally,  the image of $[H]_U \rightarrow X^{\sigma}/K_V^{\sigma}$
is, by definition, precisely the fiber  of $\mathfrak{e}$ above the trivial class, i.e. a union of connected components. 
This map  from $[H]_U$ to its image is now a  proper continuous bijection, so a homeomorphism.  

(We do not know a reference for the properness, which uses the reductivity of $\HH$; we outline this argument. Let $A^+$
be the positive cone in the connected real points of a maximal split torus of $\HH$.  Then reduction theory 
shows that there is a compact set $\Omega \subset \HH(\adele_F)$ such that $A^+ \cdot \Omega$ surjects to $\HH(F) \backslash \HH(\adele_F)$. 
We are reduced to verifying that $A^+ \rightarrow   \GG(F) \backslash \GG(\adele)$ is proper, and in turn 
that will follow from the properness of $A_G \rightarrow \GG(F) \backslash \GG(\adele)$, where $A_G$ is the  group of connected real points
of a maximal $F$-split torus inside $\GG$.  For the last statement, we use the action of the Weyl group
of $\GG$ to reduce to the case of $A_G^+ \rightarrow \GG(F) \backslash \GG(\adele)$, where $A_G^+$ is again a positive cone in $A_G$. That statement is again part of reduction theory.) \qed 
\medskip

 At this point we  are ready to prove the first Theorem from the introduction.

  \begin{theorem}
\label{thm:realtheorem}
Let $\GG, \HH, U,K, \sigma$ be as in \S \ref{sigmaax}, and $V$ a nonempty set of $\sigma$-good places.
Then $\cH(G_V, K_V)$ and $\cH(H_V, U_V)$ are both commutative integral domains, and in particular the 
 normalized Brauer homomorphism  $\NBr: \cH(G_V, K_V) \rightarrow \cH(H_V, U_V)$ of \S \ref{nBr} is defined. 
  If $\chi:\cH(H_V,U_V) \to k$ is an eigenvalue occuring in the cohomology of $[H]_U$ then the character $\tilde{\chi} = \chi \circ \NBr :\cH(G_V,K_V) \to k$  occurs in the cohomology of $[G]_K$.
\end{theorem}

\begin{proof}
The fact that $\cH(G_V, K_V)$ and $\cH(H_V, U_V)$ are commutative integral domains is well-known (at least in the context
where the coefficient ring is $\C$ rather than $k$, but the same proof works); we  summarize the proof in  Theorem \ref{Satake}.

  It suffices to prove the theorem when $V$ is finite. For suppose that $V$ is infinite and the theorem is false, 
i.e. $\tilde{\chi}$ doesn't occur in the cohomology of $[G]_K$.  Since that cohomology is finite-dimensional, there is certainly a finite subset $V' \subset V$ such that the restriction 
of $\tilde{\chi}$ to $\cH(G_{V'}, K_{V'})$ doesn't occur in the cohomology of $[G]_K$.  

Now suppose $V$ is finite. 
As we saw above, $H^*([H]_U)$ is a direct summand of $H^* (X^{\sigma}/K_V^{\sigma})$, where $X$ is as above.  Indeed, it is even  a $\cH(H_V, U_V)$-submodule, as is clear by inspection.  Now apply Theorem \ref{thm:smith} (and use   Theorem \ref{Satake}, applied to both $\GG$ and $\HH$, to check the conditions.)
\end{proof}

\section{Representation theory}
\label{sec:behavioratramified}

 Let $\G,K,\HH,U$  be as in 
 \S\ref{sigmaax}.   For any finite place $v$ the Hecke algebras $\cH(G_v,K_v)$ and $\cH(H_v,U_v)$ describe portions of the category of representations of $G_v$-modules and $H_v$-modules.  In this section we make precise a sense in which the Brauer homomorphism of \S\ref{subsec:brauerhom} ``lifts'' to a functor between categories of representations.  This is relevant both to understand the situation at ramified places and for the proof of the Theorem of \S\ref{BrauerSatake}. 

\subsection{Linkage and the Brauer homomorphism} \label{linkBrauer}
 Let $\G,K,\HH,U$  be as in 
 \S\ref{sigmaax}.  Fix a finite place $v$ of $F$, of residue characteristic $\neq p$.  We do not require that $K_v$ is maximal compact.  In particular, $\cH(G_v,K_v)$ is not necessarily commutative.  

We consider irreducible $k$-linear representations $\Pi$  of $G_v$. These are always understood to be continuous, i.e.
every vector in $\Pi$ has open stabilizer (often called ``smooth.'')  We will only consider {\em admissible} representations: 
 $\Pi^{K'}$ is finite-dimensional over $k$ for every compact open subgroup $K' \subset G$.
Say that such a representation is \emph{$\sigma$-fixed} if it is isomorphic to $\Pi \circ \sigma$.  

\begin{prop*}  If $\Pi$ is $\sigma$-fixed, then there is a unique action of $\sigma$ on $\Pi$ compatible with the $\sigma$-action on $G_v$.
\end{prop*}
\proof 

If $A$ is a $k$-linear isomorphism from $\Pi$ to itself that intertwines $\Pi$ with $\Pi \circ \sigma$, then we claim $A^p$ must be a scalar.  If that scalar is $\lambda$, then $\sigma = \lambda^{-1/p} A$ is a $\sigma$-action on $\Pi$ compatible with the $\sigma$-action on $G_v$.

To prove the claim, choose a prime-to-$p$ open subgroup $K_v^0 \subset K_v$.  Then $\Pi^{K_v^0}$ is a finite-dimensional irreducible representation of the Hecke algebra $\cH(G_v,K_v)$, so by Schur's lemma $A^p$ acts as a scalar on $\Pi^{K_v^0}$.  Since the image of the action map $G_v/K_v^0 \times \Pi^{K^0_v} \to \Pi$ generates $\Pi$, $A^p$ must act by the same scalar on the entirety of $\Pi$.
\qed
 
 \medskip

For a $\sigma$-fixed $\Pi$ with its action of $\sigma$, we may then consider the Tate cohomology  $\Tate^i \Pi$ for $i \in \{0,1\}$.  It carries an action of $H_v$.  

\begin{definition}\label{linkagedef} We say that an irreducible representation $\pi$ of $H_v$ is \emph{linked with} $\Pi$ if the Frobenius-twist $\pi^{\twist}$
(see \S \ref{twistdef})
occurs as a Jordan-Holder constituent of $\Tate^0(\Pi)$ or $\Tate^1(\Pi)$. 
\end{definition} 

As a motivating example, which may explain the role of the Frobenius-twist:  take $\GG = \HH^p$ and $\sigma$ to act by cyclic permutation. 
Then the irreducible representation $\pi_v$ of $H_v$ is linked with the irreducible representation
$\pi_v^{\otimes p}$ of $G_v \simeq H_v^p$.

The notion of linkage is a representation theoretic version of the Brauer homomorphism:  Let $\Pi$ be a $\sigma$-fixed representation of $G_v$. 
We may apply $\Tate^*$ to the $\cH(G_v,K_v)$-module $\Pi^{K_v}$.  The $\sigma$-equivariant inclusion map $\Pi^{K_v} \to \Pi$ induces $\Tate^*(\Pi^{K_v}) \to \Tate^*(\Pi)$, which in fact takes values in the $\cH(H_v,U_v)$-module $\Tate^*(\Pi)^{U_v}$.  

We now suppose that $H_v/U_v = (G_v/K_v)^{\sigma}$, as in \S\ref{subsec:sigmaactions}, i.e. $K_v$ is $\sigma$-plain 
in the notation of that section.   Then the (unnormalized) Brauer homomorphism $\Br:\cH(G_v,K_v)^{\sigma} \to \cH(H_v,U_v)$ is compatible with linkage in that
the diagram
\begin{equation} \label{Brauercompat}
\xymatrix{
\Tate^*(\Pi^{K_v})  \ar[d]_{\Tate^*(h)} \ar[r] & \Tate^*(\Pi)^{U_v} \ar[d]^{\Br(h)} \\
\Tate^*(\Pi^{K_v}) \ar[r] & \Tate^*(\Pi)^{U_v}
}
\end{equation}
commutes for any $h \in \cH(G_v,K_v)^{\sigma}$.
 
\begin{proof} We give the proof for $\Tate^0$. 
If $x \in \Pi^{K_v}$ is $\sigma$-fixed, then the image of $x + N({\Pi^{K_v}})$ in $\Tate^0(\Pi)^{U_v}$ is $x + N(\Pi)$, and to verify \eqref{Brauercompat} we have to show that the equation
\[
\Br(h)\ast\left(x + N(\Pi)\right) = (h \ast x) + N(\Pi)
\]
holds.  The left-hand side is
\begin{eqnarray*}
& & \sum_{gU_v \in H_v / U_v} \Br(h)(U_v, gU_v) g(x +N(\Pi))\\
& = & \sum_{g U_v \in H_v / U_v} h(K_v, gK_v)(gx + N(\Pi))
\end{eqnarray*}
and the right-hand side is 
\[
\left(\sum_{g K_v \in G_v /K_v} h(K_v,gK_v)gx \right) + N(\Pi)
\]
so \eqref{Brauercompat} reduces to checking
\[
\sum_{g K_v \in G_v /K_v - H_v/U_v} h(K_v,gK_v)gx \in N(\Pi)
\]
Since we have assumed $(G_v/K_v)^{\sigma} = H_v/U_v$, $\sigma$ acts freely on the set indexing the sum, which therefore does belong to $N(\Pi)$.  A similar computation shows \eqref{Brauercompat} holds for $x + (1-\sigma)(\Pi^{K_v}) \in \Tate^1(\Pi^{K_v})$.
\end{proof}

 \subsection{Conjectures} \label{sec:linkedrep}

It seems very reasonable to believe that
\begin{quote} 
{\it Let  $\Pi$ be a $\sigma$-fixed irreducible admissible representation of $G_v$. 
Then $\Tate^* \Pi$ is of finite length as an $H_v$-representation.}
\end{quote}
 
The conjecture is motivated by the analogy with Eisenstein series formulated in the introduction.  If the functor $\Tate^i \Pi$ should be seen as an analog of the Jacquet functor, then the Conjecture is a counterpart to the fact that the Jacquet functor carries admissibles to admissibles  \cite[Theorem 3.3.1]{Casselman}. The analogy, together with computations we have carried out in the case of depth zero base change for $\GL_n$, 
suggests another conjecture, which we will leave in a slightly less precise form:

\begin{quote}
{\it Linkage is compatible with the Langlands functorial transfer associated to a $\sigma$-dual homomorphism $\Lpsi: \LH \rightarrow \LG$ {\rm (\S\ref{torsionfunctoriality}, \S\ref{subsec:sigmadual})}}
\end{quote}

In particular, if  the $\sigma$-fixed representation $\Pi$ of $G_v$ is linked with the representation $\pi$ of $H_v$, 
we should expect $\Lpsi$ to carry the Langlands parameter of $\pi$ to the Langlands parameter of  $\Pi$. In other words: Just as the Jacquet functor realizes functoriality between an $L$-group and a Levi subgroup, we expect that the Tate cohomology functor should realize functoriality for the $\sigma$-dual homomorphism of \S\ref{subsec:sigmadual}.

\subsection{Ramified places}
\label{subsec:ramified}
Fix a finite set $V$ of places of $F$ and a level structure $K \subset \G$, 
where each place $v \in V$ is good (\S \ref{goodplace}) with respect to $K$.  
 
Let $S$ be a finite set of finite places, disjoint from $V$, and put $G_S = \prod_{w \in S} G_w$. 
Consider all level structures $K'$ that agree with $K$ away from the set $S$, that is to say, 
\begin{equation} \label{K'def} K' = \prod_{v \in S} K'_v \cdot \prod_{v \notin S} K_v.\end{equation}
The $V$-Hecke algebra $\cH(G_V,K_V)$ is a commutative integral domain acts on the cohomology of 
$[G]_{K'}$. Let $\chi: \cH(G_V, K_V) \rightarrow k$ be a homomorphism. 
We may form the $G_S$-module
\[
\pi(\chi) := \chi \text{-component of }\varinjlim_{K'} H^*([G]_{K'})
\]
where by $\chi$-component we mean  in fact the localization at the maximal ideal defined by $\chi$, 
i.e. the generalized eigenspace corresponding to $\chi$. Strictly speaking, as we have defined it, this depends on both $\chi$ and $V$,
but we have suppressed the dependence on $V$ in the notation. 

The precise determination of $\pi(\chi)$ is an interesting and difficult question; it is the subject of the mod $p$ Langlands correspondence \cite{emerton}.
In any case, $\pi(\chi)$ and all of its irreducible subquotients are admissible: if we take $K'$ small enough that
$K'_S:=\prod_{v \in S} K'_v$ has pro-order that is prime-to-$p$, then  
  $\pi(\chi)^{K'_S}$ is identified with the $\chi$-component of cohomology of $H^*([\G]_{K'})$.  If we shrink $K'$ further,
we may ensure that $[\G]_{K'}$ is a manifold and has finite-dimensional cohomology.

 We are ready to formulate the exact relationship between linkage and the functoriality associated to a $\sigma$-dual homomorphism:

\begin{theorem}
\label{thm:behavioratramified}
Let $\G,\HH,  K, U,\sigma$ be as in \S \ref{sigmaax}, and suppose $\GG$ is semisimple and $\HH$ is connected.
Let $V$ be a finite set of $\sigma$-good places (see Proposition \ref{gppd}) and $S$ a finite set of 
finite places disjoint from $V$ and all primes above $p$. 

Let $\chi: \cH(H_V, U_V) \rightarrow k$ be a character,   $\psi = \chi \circ \NBr$.
 Let $\pi = \pi_{\chi}$
and $\Pi= \Pi(\psi)$ be the representations of (respectively) $H_S$ and $G_{S}$ attached to $\chi$ as in \S\ref{subsec:ramified}.  Then any irreducible subquotient of the $H_{S}$-module $\pi(\chi)$ is linked with an irreducible subquotient of the $G_{S}$-module $\Pi( \psi)$.
\end{theorem}

  Note a minor weakness compared to Theorem \ref{thm:realtheorem}: the set $V$ above is 
required to be finite. It is likely this can be relaxed, and it seems harmless in practice.
More  seriously one could ask for a more precise statement --- for example, a complete determination of one space in terms of the other, but we do not pursue this here.
  
\begin{proof} 
 Let $K' \leqslant G_S$ be an open compact subgroup as above (see around \eqref{K'def}), now  assumed $\sigma$-stable;
 let $U' = (K')^{\sigma} \leqslant H_S$.

 We proceed just as in \S \ref{conn-analysis} and Theorem \ref{thm:realtheorem} but
 in cohomology rather than homology. That furnishes an embedding of 
 $$H^*([H]_{U'}) \mbox{ as a subquotient of } \Tate^* H^*([G]_{K'})$$
 equivariantly  for the action of $\cH(H_V, U_V)$. 
In fact, we can do this compatibly at all levels at once, thus embedding
$\varinjlim H^*([H]_{U'}) \mbox{ as a subquotient of }  \Tate^* \left(  \varinjlim H^*([G]_{K'})\right) $ 
in an $H_S \times \cH(H_V, U_V)$-equivariant fashion.    We explicate this a little:

Proceed as in the proof of Theorem \ref{thm:realtheorem},  but form the associated ``Smith double complex'' from 
  the direct limit of cochain complexes for the $[G]_{K'}$. 
The sequences are convergent because the cohomological dimension 
of $[G]_{K'}$ is bounded independent of $[K']$. Our reasoning as before
shows that the {\em hv}-complex converges to $ \varinjlim H^*([G]_{K'}^{\sigma})$
whereas the $E^2$ term of the  {\em vh}-complex is $\varinjlim \Tate^* H^*([G]_{K'})$. 
Moreover, $\varinjlim H^*([H]_{U'})$ is an $H_S \times \cH(H_V, K_V)$-summand
of $ \varinjlim H^*([G]_{K'}^{\sigma})$.

Now localizing at a character  of $\cH(H_V, K_V)$ we see that (as $H_S$-representations) 
\begin{quote}
Any irreducible constituent $\tau$  of 
$\pi(\chi)$ is  a composition factor of $\Tate^* \Pi(\psi')$
\end{quote}
where $\psi' = \chi \circ \Br$ is considered as a character of $\cH(G_V, K_V)^{\sigma}$. 
By an argument with Frobenius acting on the coefficients, similar to that given earlier, 
we see that $\tau^{(p)}$ is a composition factor of $\Tate^* \Pi( (\psi')^p)$. 
Finally, because $\chi \circ \NBr$ is the unique extension of $(\psi')^p$ 
to $\cH(G_V, K_V)$, we see that $  \Pi( (\psi')^p) = \Pi(\psi)$. 
   \end{proof}

\section{Satake parameters} \label{Satakeparam}

In this background section, we recall the Satake isomorphism and the notion of modularity.
 Then we reformulate the Brauer map in terms of Satake parameters.  
 This section is primarily to set up notation and give references
 for results which are standard over $\C$ but less so over $k$.

 However the formulation of Theorem \ref{cSatake}
 and the accompanying discussion of the $c$-group may be of independent interest. Although Theorem \ref{cSatake} points
 the way to the most intrinsic way of formulating our results, we do not use it in the rest of the paper --- sticking instead
 to the $L$-group and taking  the {\em ad hoc} approach to the various square roots that occur. This is enough for our purposes,
 and  can be readily matched with the existing literature. 
 
\subsection{Restricted Weyl group}
\label{subsec:w0v}

Let $v$ be a good place for $\GG$, and let $A_v \subset B_v$ be a maximally split torus and Borel subgroup of $G_v$;
let $\mathbf{A}_v \subset \mathbf{B}$ be the corresponding algebraic groups, and $\mathbf{T}$ the quotient torus of $\mathbf{B}$. 
  The ``restricted Weyl group'' of $\GG$ at $v$ is the quotient $N_{G_v}(A_v)/\Cent_{G_v}(A_v)$, i.e. the normalizer of $A_v$ divided by the centralizer of $A_v$.  We denote it by $W_{0,v}$.  The correspondence between unramified characters of $T_v$ and  splittings of $\hat{T}(k) \rtimes \Frob_v \rightarrow \langle \Frob_v \rangle$ from \S\ref{subsec:cft} is compatible with the natural action  of $W_{0,v}$ on each side.  

The restricted Weyl group also acts on the dual torus $\hat{A}_v$ to $A_v$.  We will need the following assertion:

\begin{prop*}
The action of any $w \in W_{0,v}$ on $\hat{A}_v$ is induced by an element $n \in \hat{G}(k)$ normalizing $\hat{T}$ and fixed by $\Frob_v$.
\end{prop*}
This result is proven in \cite[Lemma 6.2]{Borel} over the complex numbers. 

\begin{proof}

As in \cite[\S 6.1]{Borel} we must show that each $\Frob_v$-fixed class in the Weyl group of $(\hat{G}, \hat{T})$
has a $\Frob_v$-fixed representative within $\hat{G}(k)$.

 It is proven by Steinberg \cite[p 173]{Steinberg}
that  $\sigma$-fixed points of the Weyl group $\hat{W}$
are generated by {\em basic reflections} $W_D$ indexed by orbits $D$ of $\Frob$ 
 on simple roots on $(\hat{G}, \hat{T})$.  The basic reflection $W_D$ is characterized (item (3) of {\em loc. cit.}) as the unique element of $\langle w_{\alpha} \rangle_{\alpha \in D} \subset \hat{W}$
with the property that $W_D$ sends $D$ setwise into $-D$.  
Equivalently, $W_D$ is the long element of the Weyl of the Levi subgroup $\hat{M}_D$
obtained when we adjoin each of the $\alpha \in D$ to $\hat{T}$.

We are reduced to a verification inside such a Levi group $\hat{M}$. 
This is a reductive group $\hat{M}$, equipped with a pinning, 
and the pinned automorphism $\Frob_v$ acts transitively on simple roots. We need to produce a representative for the long Weyl group element that is
$\Frob_v$-fixed.  It suffices to produce such a representative inside the derived group $[\hat{M}, \hat{M}](k)$,
and then inside $\hat{M}'(k)$ where $\hat{M}'$ is the simply connected cover of that derived group. 
The question can then be analyzed on each simple factor of $(\hat{M}', \Frob_v)$. The only nontrivial cases are as follows:
\begin{itemize}
\item[(a)]  $\hat{M}'=\SL(3)$:   the pinned automorphism of  $\SL(3)$ 
is given by $g \mapsto \omega (g^t)^{-1} \omega$, where $\omega = \left( \begin{array}{ccc} 0 & 0 & 1 \\ 0 & -1 & 0 \\ 1 & 0 & 0 \end{array}\right)$.
This fixes a representative for the long Weyl group element, namely  $\omega$ itself. 
\item[(b)] $\hat{M}' = \hat{M}_0^r$ and $\Frob_v$ permutes the simple factors. 
Note that $\Frob_v^r$ still acts transitively on the simple roots of $\hat{M}_0$, and by what we showed above,
there is a $\Frob_v^r$-fixed  representative $w \in \hat{M}_0$ for the long Weyl element; then $(w, \Frob(w), \dots ) \in \hat{M}_0^r$ gives the desired representative. 
\end{itemize}
\end{proof}

\subsection{The Satake isomorphism}
\label{Satake}

  We now describe the Satake isomorphism. We begin with a statement of the main ingredients, but presented
``over $k$'' and with no choices of square roots made.  For this statement, we will require the following twisted action of the Weyl group $W_{0,v}$ on $\hat{A}_v$:
\begin{equation}
\label{twisted}
w \ast a = wa  \cdot \sqrt{\frac{  \Sigma^*_G}{w\Sigma^*_G}}(q_v) \quad \text{for $w \in W_0$ and   $a \in \hat{A}_v(k)$}
\end{equation}
and $\Sigma^*_G$ is the co-character of $\hat{T}$ given by the sum of all positive coroots. Note that $\Sigma^*/w \Sigma^*_G$ is divisible by $2$ in that cocharacter lattice; thus $ \sqrt{\frac{  \Sigma^*_G}{w\Sigma^*_G}}(q_v)$ makes sense, and we can then project
to $\hat{A}_v$ via $\hat{T} \rightarrow \hat{A}_v$. 
 
\begin{theorem*}
Let $G_v = \G(F_v)$ be a reductive $v$-adic group and let $K_v \subset G_v$ be a maximal compact subgroup satisfying the conditions of \S\ref{goodplace},
which is ``in good position'' with respect to $A_v$, i.e. $A_v \cap K_v$ is a maximal compact subgroup of $A_v$. 
Let  $\hat{A}_v$ be as in \S\ref{subsec:cft}, \S\ref{subsec:w0v}.  The following hold: \begin{itemize} 
\item[(i)] There is a natural isomorphism
\begin{equation}
\label{twistedSatake}
\begin{array}{rcl}
\cH(G_v,K_v) & \stackrel{\sim}{\longrightarrow} &
\begin{array}{c}
\mbox{\rm $(W_{0,v},\ast)$-invariant regular} \\
\mbox{\rm functions on $\hat{A}_v$}
\end{array}
\end{array}
\end{equation}

  In particular, $\cH(G_v, K_v)$ is a commutative integral domain. 
 
\item[(ii)] There is a natural identification
\begin{equation} 
\label{Invarianttheory}  
\begin{array}{rcl}
\begin{array}{c}
\mbox{\rm $(W_{0,v},\cdot)$-invariant regular}\\
\mbox{\rm functions on $\hat{A}_v$}
\end{array}
&
\stackrel{\sim}{\longleftarrow} &
\begin{array}{c}
 \mbox{\rm regular functions on} \\
 \mbox{\rm $\hat{G}\rtimes \Frob_v \git \hat{G}$}. 
 \end{array}
 \end{array}
 \end{equation} 
where the $W_{0,v}$-action on $\hat{A}_v$ is now the usual one. 
\end{itemize}
\end{theorem*}

\begin{proof}[Proof of (i)]  
We will prove (i) here, and (ii) in \S\ref{ReardenMetal}.
There does not seem to be a reference with coefficients in $k$, so we give some details. We emphasize
that we are working with characteristic $p$ coefficients and no square roots are chosen. 

We will deduce (i) from some properties of the standard Satake isomorphism (i.e. the Satake isomorphism over $\C$).  It is a standard observation that the coefficients of this isomorphism can be shrunk from $\C$ to $\Z[q_v^{\pm 1/2}] \subset \R$  (see e.g. \cite{Gross}).  By using the modified $W_{0,v}$-action of \eqref{twisted}, we can further shrink the coefficients to $\Z[q_v^{\pm 1}] \subset \Q$.

Let us first recall the standard Satake isomorphism. Let $B_v$ be the $F_v$-points of a Borel containing $A_v$, 
and $N_v$ the points of its unipotent radical. 
Let $\delta_{\R_{>0}}$ be the modular character of the Borel, i.e.   the composite
\begin{equation} \label{Deltadef}
B_v \to F_v^* \to q_v^{\Z} \subset \R_{>0},
\end{equation} 
where the first map is the sum $\Sigma_G$ of all positive roots, evaluated on $F_v$-points. 
 Let $\delta_{\R_{>0}}^{1/2}$ denote the positive square root of $\delta_{\R_{>0}}$.   The usual Satake transform $f \mapsto \mathcal{S}(f)$ produces from $f \in \cH(G_v,K_v)$ a function $\mathcal{S}(f)$ on $A_v$ given by
 \[
\mathcal{S}(f)(t) = \delta^{1/2}_{\R_{>0}}(t) \cdot \int_{N_v} f(K_v,tnK_v) dn
 \]
where the measure $dn$ is normalized so that $N_v \cap K_v$ has mass $1$. Then $\mathcal{S}(f)$ is compactly supported, and constant on $(A_v \cap K_v)$-cosets.  We may therefore regard $\mathcal{S}$ as an element of the group ring $\C[X_*(\mathbf{A}_v)] = \C[X^*(\hat{A}_v)]$.
As such \cite{Gross} $\mathcal{S}$ defines a ring isomorphism from $\cH$ to  the $W_{0,v}$-invariant subring of $\C[X^*(\hat{A}_v)]$, where $W_{0,v}$ acts as in \S\ref{subsec:w0v};
moreover,  \cite[Lemma 10.2.1]{HainesRostami}, with  respect to the basis of double $K_v$-cosets in the domain, and of $W_{0,v}$-orbits on $X^*(\hat{A}_v)$ in the codomain, $\mathcal{S}$ is upper triangular and the diagonal entries belong to $\Z[q_v^{\pm 1/2}]$. In other words, $\mathcal{S}$ defines an algebra isomorphism 
\[
\text{$\cH(G_v,K_v)$ with $\Z[q_v^{\pm 1/2}]$-coefficients} \stackrel{\sim}{\to} \Z[q_v^{\pm 1/2}][X^*(A_v)]^{W_{0,v}}
\]
where the $W_{0,v}$-action on $X^*(\hat{A}_v)$ is the ``untwisted'' one from \S\ref{subsec:w0v}.  

One obtains a form of the Satake isomorphism over $k$ by tensoring with $k$, but this requires choosing an embedding $\Z[q_v^{\pm 1/2}] \to k$, i.e. choosing a square root of $q_v$ in $k$.
To avoid that,  define a modified Satake transform $\mathcal{S}^*$ by 
\begin{equation} \label{modSatake}
 \mathcal{S}^*(f)(t) = \int f(K_v,tn K_v) dn
 \end{equation} 
 i.e. $\mathcal{S}^* := \delta^{-1/2} \mathcal{S}$.   Then $\mathcal{S}^*$ is also an injective ring homomorphism with values in $\C[X^*(\hat{A}_v)]$.  
 Since the $*$-action of $W_{0,v}$ on $\C[X^*(\hat{A}_{v})] $  \eqref{twisted} sends $\chi \in X^*(\hat{A}_{v})$ 
to 
\begin{eqnarray*}
w* \chi & = & w \chi \cdot q_v^{\langle \Sigma^*_G-w^{-1} \Sigma^*_G,  \chi \rangle/2} \\
 & = & w \chi \cdot q_v^{\langle w \Sigma^*_G- \Sigma^*_G,  w \chi \rangle/2},
\end{eqnarray*}
and  we deduce that $\mathcal{S}^*$  is an isomorphism onto the $(W_{0,v},*)$-invariant subring of $\C[X^*(\hat{A}_v)]$.
 
For the rest of the proof of (i), let us write $\Z' := \Z[q_v^{-1}]$ and $\Z'' := \Z[q_v^{- 1/2}]$, for short.  The $*$-action of $W_{0,v}$ on $\C[X^*(\hat{A}_v)]$ clearly leaves the subrings $\Z'[X^*(\hat{A}_v)]$ and $\Z''[X^*(\hat{A}_v]$ stable, and (d) for $\mathcal{S}$ implies that $\mathcal{S}^*$ lies in $\Z''[X^*(\hat{A}_v)]$.  To conclude that $\mathcal{S}^*$ is an isomorphism over $\Z'$ it is only necessary to check that $\mathcal{S}^*$ carries the basis element indexed by $K_v t_0 K_v$ into $\Z'[X^*(\hat{A}_v)]$, i.e. to check that  
the measure of $n \in N_v$ with $ K_v tn K_v = K_v t_0 K_v$ lies in $\Z'$; 
 that follows from the normalization of   $dn$. \end{proof}
 
\subsection{Invariant theory lemma}
\label{ReardenMetal}
The proof of part (ii) of the Theorem of \S\ref{Satake} depends on the following lemma:
\begin{lemma*}
Every $W_0$-invariant function on $\hat{A}$ arises from a $\hat{G}$-invariant function on $\hat{G} \rtimes \Frob_v$
\end{lemma*}

\begin{proof}

We will actually show that the ring of $W_0$-invariant regular functions on $\hat{A}$ is spanned by traces of representations of $\hat{G} \rtimes \Frob_v$.
(It will follow that  functions on  $\hat{G} \rtimes \langle \Frob_v \rangle \git \hat{G}$
are spanned by traces of representations of $\hat{G} \rtimes \langle \Frob_v \rangle$.)

The ring of regular functions on $\hat{A}$ is naturally identified with the group ring $k[X^*(\hat{T})^\Frob_v]$.  The $W_{0,v}$-invariant regular functions have a basis parameterized by $W_{0,v}$-orbits on $X^*(\hat{T})^{\Frob_v}$.   

Every $W_{0,v}$-orbit on $X^*(\hat{T})^{\Frob_v} = X_*(\TT)^{\Frob_v}$ contains a dominant element:
if we look on the dual side, the relative Weyl group for $\mathbf{A}$ (i.e., the Weyl group of the relative root system, that might not be reduced)   has a relatively dominant element in its orbit,
which implies it is dominant considered as a cocharacter of $\mathbf{T}$ --- although it may lie on a wall.  Let us denote the basis element corresponding to $W_{0,v} \nu$, where $\nu$ is dominant, by  $\omega_\nu$.

For each dominant weight $\nu$, let $V_\nu$ denote the corresponding Weyl module for $\hat{G}$, i.e. by Borel-Weil
\[
V_\nu = H^0(\hat{G}/\hat{B};\OO(\nu))
\]
Note that $V_{\nu}$ need not be irreducible, since we are in characteristic $p$, but it does not matter for us. 
If $\nu$ is $\Frob_v$-invariant then (as $\Frob_v$ leaves $\hat{B}$ stable) the line bundle $\OO(\nu)$ acquires a $\hat{G} \rtimes \langle \Frob_v \rangle$-equivariant structure, and $V_\nu$ is canonically a $\hat{G} \rtimes \langle \Frob_v \rangle$-module.  

Let $W$ be the full Weyl group for $\hat{T}$, i.e. the quotient $N(\hat{T})/\hat{T}$. 
Let $|W\nu| \subset X^*(\hat{T})$ denote the convex hull of the $W$-orbit of $\nu$.  For $\lambda \in |W\nu|$, let $V_\nu(\lambda)$ denote the corresponding weight space of $V(\nu)$.  For $t \in \hat{T}$, we compute
\begin{equation}
\label{eq:chinu1}
\chi_\nu(t \rtimes \Frob_v) = \sum_{\lambda \in |W\nu|^{\Frob_v}} \lambda(t) \mathrm{Tr}(\Frob_v\vert_{V_\nu(\lambda)}) 
\end{equation}
(where $|W\nu|^{\Frob_v} $ is the $\Frob_v$-fixed elements of $|W\nu|$). 
We will show that \begin{equation}
\label{eq:chinu2}
\sum_{\lambda \in (W\nu)^{\Frob_v}} \lambda(t) \mathrm{Tr}(\Frob_v\vert_{V_{\nu}(\lambda)}) = \mathrm{Tr}(\Frob_v\vert_{V_{\nu}(\nu)}) \omega_\nu
\end{equation}
The left-hand side is the dominant term of  the right-hand side of \eqref{eq:chinu1}, so that \eqref{eq:chinu2} implies 
$$\chi_v(t \rtimes \Frob_v) = \mathrm{Tr}(\Frob_v\vert_{V_\nu(\nu)}) \cdot \omega_{\nu} + \sum_{\nu'} a_{\nu'} \omega_{\nu'}$$
where every $\nu' \in X^*(\hat{T})^{\Frob_v}$ has the property that $|\nu'| < |\nu|$. 
After observing that $\mathrm{Tr}(\Frob_v\vert_{V_\nu(\nu)})$ is nonzero and $V_{\nu}(\nu)$ is one-dimensional, it follows by induction that $\omega_\nu$ can be written as a linear combination of characters of $V_{\nu'}$, where $\nu' \in X^*(T)$ and $\nu' \leq \nu$.

Let us prove \eqref{eq:chinu2}.  We may write $W\nu$ as $W/W_{\nu}$, where $W_{\nu}$ is the $\Frob_v$-stable parabolic subgroup fixing $\nu$.  Recall that each coset of $W_{\nu}$ has a minimum element in the Bruhat ordering on $W$.  As $\Frob_v$ preserves the Bruhat ordering on $W$, it follows that each $\Frob_v$-fixed coset is represented by a $\Frob_v$-fixed element of $W$, i.e. by an element of $W_0$.  As $\Frob_v$ and $W_0$ commute, the trace $\mathrm{Tr}(\Frob_v\vert_{V_\nu(\lambda)})$ is therefore constant on the unique $W_0$-orbit on $(W\nu)^{\Frob_v}$.  This completes the proof.
\end{proof}

\begin{proof}[Proof of part (ii)]
We are now ready for the proof of part (ii) of the Theorem of \S\ref{Satake}.

Any $\hat{G}$-invariant regular function on $\hat{G} \rtimes \Frob_v$ gives by restriction a regular function on $\hat{T} \rtimes \Frob_v$.  As $t\Frob_v$ and $t (t')^{\Frob_v} (t')^{-1} \Frob_v$ are conjugate by $t'$, this restricted function descends to a regular function on the $\Frob_v$-coinvariants on $\hat{T}_v$, that is to say, it descends to  $\hat{A}_v$.  By the Proposition of \S\ref{subsec:w0v}, this function is $W_{0,v}$-invariant as well.  After the Lemma, it only remains to check that any $\hat{G}$-invariant regular function on $\hat{G} \rtimes \Frob_v$ that induces the zero function on $\hat{A}_v$ (and thus on $\hat{T}$), is zero.

The image of the action map $\alpha:\hat{G} \times (\hat{T} \rtimes \Frob_v) \to \hat{G} \rtimes \Frob_v$ is of dimension $\dim(\hat{G}) + \dim(\hat{T}) - \dim(M)$, where $M$ is the set-wise stabilizer of $\hat{T}\rtimes \Frob_v$ in $\hat{G} \rtimes \Frob_v$.  A Lie algebra computation
shows that the identity component of $M$ is $\hat{T}$; consequently $\alpha$ is dominant.   This completes the proof.
 \end{proof}

\subsection{Local pseudoroots}
\label{subsec:localpseudoroot}

We now address the ``mismatch'' between \eqref{twisted} and \eqref{Invarianttheory} of Theorem \ref{Satake}. After a suitable choice, we can identify the $*$-action and the usual action of $W_{0,v}$
on $\hat{A}_v$, thus obtaining the usual statement of the Satake isomorphism.  A more intrinsic approach is described in Theorem \ref{cSatake} but, for most of this paper,
we will follow the {\em ad hoc} approach outlined below, because it is enough for our purposes and much closer to the literature. 

With notation as in the Theorem, a {\em pseudoroot at $v$} is a fixed point of $(W_{0,v},*)$ 
on $\hat{A}_v$ whose square is normalized in a natural way. More precisely, a pseudoroot is a 
choice of element $\alpha_0 \in \hat{A}_v$ such that 
 \begin{enumerate}
\item[(a)] $\alpha_0^2 = \Sigma^*_G(q_v)$
\item[(b)] $\alpha_0$ is invariant under the $*$-action of $W_{0,v}$ (see \eqref{twisted}). 
\end{enumerate}
Such always exist: If we choose a square root $\sqrt{q_v} \in k^*$ we could take $\alpha_0 = \Sigma^*_G(\sqrt{q_v})$. 
If $\Sigma_G^*$ is divisible by two, then $\sqrt{\Sigma_G^*}(q_v)$ gives a  particularly natural choice. 

    The rule 
 $a \mapsto a \alpha_0$ defines an isomorphism 
\begin{equation} 
\label{twistedID}\hat{A}_v /(W_{0,v}, \mbox{ usual action})    \longrightarrow \hat{A}_v /(W_{0,v}, \mbox{ twisted action})  
\end{equation}
Thus, by composing \eqref{twistedSatake}, pullback under \eqref{twistedID}, and \eqref{Invarianttheory}, we arrive at an identification
\begin{equation} \label{ChosenPseudoSatake}
\text{characters of }\cH(G_v,K_v) \simeq \hat{G}\rtimes \Frob_v \git \hat{G}
\end{equation} 
which we refer to, in short, as the Satake isomorphism. 
In particular, having fixed a pseudoroot, each element of $\hat{G}(k)$ gives a character $\cH(G_v, K_v) \rightarrow k$,  
and two elements $g_1, g_2$ give  the same character precisely when $g_1 \cdot \Frob, g_2 \cdot \Frob$
have the same projection to $\LG \git \hat{G}$.

There are two other important equivalent ways to think of pseudroots,
via the identifications of \S \ref{subsec:cft}: 
\begin{equation} \label{cftredux}  \hat{A}_v \simeq \mbox{ conj. classes of splittings $\LT_v \rightarrow \langle \Frob_v \rangle$ } \simeq \mbox{ unramified characters of $T_v$} \end{equation}  
\begin{itemize}

\item[(a)] As ``square roots of the modular character'' of $B_v$: 

A pseudoroot gives rise via \eqref{cftredux}  to an unramified character of $T_v$, thus also an unramified character  of $B_v$. 
In this way, pseudoroots are identified with certain preferred square roots of the modular character for $B_v$:

Let $v$ be a good place.  By proceeding as in \eqref{Deltadef}
but using the natural map $q_v^{\Z} \rightarrow k^*$ instead of the inclusion $q_v^{\Z} \subset \R$ we get
the ``$k$-valued modular character'' 
\begin{equation}
\label{eq:delta}
\delta: B_v \to F_v^* \stackrel{|\cdot|_v}{\longrightarrow} q_v^{\Z} \to k^*
\end{equation}
or $\delta = |\Sigma_G|_v$ for short.   Then a  local pseudroot for $\G$ at $v$ corresponds to an unramified character $\delta^{1/2}:B_v \to k^*$ that squares to $\delta$, and that obeys
\[
(\delta^{1/2})^2 = \delta \qquad \frac{w\delta^{1/2}}{\delta^{1/2}} = \left| \sqrt{ \frac{w \Sigma_G}{\Sigma_G}}\right|_v \text{ for $w \in W_{0,v}$}
\]

\medskip

\item[(b)] A pseudoroot is uniquely determined by the associated splitting $\rho: \langle \Frob_v \rangle  \rightarrow \LT_v$;
  this splitting sends $\Frob_v$ to $\widetilde{\alpha_0} \cdot \Frob_v$, where $\widetilde{\alpha_0} \in \hat{T}$
  is a lift of $\alpha_0 \in \hat{A}_v$. 
  
  We will often simply say ``let $\rho: \langle \Frob_v \rangle  \rightarrow \LT_v$ be a pseudoroot'', meaning that  
$\rho(\Frob_v) \cdot \Frob_v^{-1}$ is a pseudoroot in the above sense.
  
   It will also be convenient to use the symbol
  $\inverserho$ for the parameter of the {\em inverse} pseudoroot, i.e.  
  \begin{equation} \label{inverserhodef} 
  \inverserho: \Frob_v \mapsto \widetilde{\alpha_0^{-1}} \Frob_v.\end{equation}
 i.e. $\inverserho$ parameterizes a square root of the {\em negative} modular character for $B_v$. \end{itemize}

\subsection{Global pseudoroots and canonical pseudoroots}
\label{subsec:globalpseudoroot} 
We define a \emph{global} pseudoroot to be a $k^*$-valued idele class character of $\Tcan_G$ (the canonical $F$-torus of $\G$, \S\ref{subsec:cantorus}) which restricts to a local pseudoroot at almost every good place $v$ --- that is to say, the associated character of a Borel subgroup $B_v \subset \G(F_v)$, via $B_v \rightarrow \Tcan_G(F_v)$, 
is a pseudoroot in the sense of  the discussion around \eqref{eq:delta}.

At least in our cases (and presumably always, by a global Langlands for tori with $k^*$ coefficients, but we did not verify the validity of this) 
a global pseudoroot  yields a conjugacy class of sections $\rho_G$ of $\LT \to \Gamma_F$ with the property that for almost every $v$, the element $\rho_G(\Frob_v) \Frob_v^{-1} \in \hat{T}(k)$ projects to an element of $\hat{A}_v$ satisfying conditions (a) and (b) of \S\ref{subsec:localpseudoroot}.

There are two situations when there is a canonical choice of global pseudoroot:
\begin{itemize}
\item[(a)] When the characteristic of $k$ is two, where the trivial character is the canonical choice ---  equivalently, the trivial
splitting where $\rho_G(\Frob_v) = \Frob_v$;
\item[(b)] When the half-sum of positive roots $\frac{\Sigma_G}{2}: \Tcan_G \rightarrow \mathbb{G}_m$ exists in the character lattice for $\G$. 
Then we pull back the ``cyclotomic'' idele class character of $\mathbb{G}_m(\adele) \rightarrow k^*$ via $\frac{\Sigma_G}{2}$. (The ``cyclotomic'' idele class character corresponds to  the Hecke character that sends a  prime-to-$p$ ideal to its norm in $k^*$.)    
The associated splitting is 
\begin{equation} \label{writeitout} \rho_G: \gamma \mapsto \frac{\Sigma_G^*}{2}(\mathrm{cyclo} (\gamma)) \rtimes \gamma,
\end{equation}
where $\mathrm{cyclo}:\Gamma_F \to \mathbf{F}_p^*$ is the action of $\Gamma_F$ on $p$th roots of unity in $\overline{F}$.
\end{itemize}

 When both these apply, these canonical choices agree: the half-sum of positive roots determines the trivial character.  
We shall simply say ``there is a canonical pseudoroot'' in these cases. 

Our focus will be on cases where there is a canonical pseudoroot for both $\G$ and $\HH$, but in general there need not be any global pseudoroot at all, e.g. \S\ref{subsec:Cgroup}.  
 \subsection{Parabolic induction and Satake parameters}
\label{piSat}
The isomorphism of \eqref{Invarianttheory} induces a bijection between maximal ideals, i.e. a bijection between semisimple twisted conjugacy classes in $\hat{G}(k)$ and characters of the Hecke algebra.  To clarify the role of pseudoroots, let us give an explicit formula for this bijection.

Let $\theta$ be an unramified character of $T_v$. We can form the unnormalized parabolic induction $\mathrm{J}_B^G(\theta)$. 
This is the submodule of $C_c^{\infty}(G_v;k)$ given by those $s:G_v \to k$ that obey $s(bg) = \theta(b) s(g)$ for $b \in B_v$. The $K_v$-invariant subspace is a one-dimensional $k$-vector space (because $G_v = B_vK_v$), generated by the  vector $v^0$  
whose restriction
 to $K_v$ is identically $1$. For $h \in \cH(G_v, K_v)$ we have\footnote{Indeed we compute directly  $ h v^0(e) = \int_{n \in N_v, a \in A_v , k \in K_v} h( an)  \theta(a) dn_v da_v $; note that the measure on $G_v$, normalized so that the measure of $K_v$ is $1$, also decomposes \cite[\S 4.1]{Cartier} 
 via $g=ank$ as $da \cdot dn \cdot dk$, where the measures on $A_v, N_v$ are normalized so that the measures of $A_v \cap K_v, N_v \cap K_v$ is $1$.
} \begin{equation} \label{JohnGalt} h v^0 =  \langle \mathcal{S}^* h , \theta \rangle  \cdot v^0\end{equation} 
with $\mathcal{S}^*$ as in \eqref{modSatake}. 
 In other words, the Hecke algebra acts on the $K_v$-fixed vector by the character  obtained by pulling back
$\theta$ via $\mathcal{S}^*$.      
Let $\chi_{\theta}$ be this character of the Hecke algebra.

Now $\theta$ determines a point $a_{\theta} \in \hat{A}_v$ as  in \eqref{cftredux}, and thus 
 a well-defined conjugacy class $C(\theta) \subset \hat{G}(k) \rtimes \Frob_v \git \hat{G}$ --- the $\hat{G}$-conjugacy class of any element of the form $t_{\theta} \Frob_v$
 where $t_{\theta} \in \hat{T}$ lifts $a_{\theta} \in \hat{A}_v$. 
  
Now fix a pseudoroot, 
which we think of, by the discussion around 
\eqref{eq:delta}, as a square root $\delta^{1/2}$ of the modular character. The Satake correspondence is then given by \begin{equation} C(\theta) \longleftrightarrow \chi_{\theta \delta^{1/2}} \end{equation}

Or, to say a different way: the Hecke character arising from the induction $\mathrm{J}_B^G(\theta)$
has parameter  given by the class of 
$$  t_{\theta} \cdot \inverserho_G(\Frob_v) $$ 
 inside
$ \hat{G}(k) \rtimes \Frob_v \git \hat{G}$.

For later use, let us examine the situation when we induce from a parabolic that is not minimal.
Suppose $P_v $ is a parabolic subgroup and $\theta$  an unramified  character of $\mathbf{L}^{\ab}$, the abelianized Levi subgroup for the parabolic $P_v$,
and its Langlands parameter therefore is a twisted conjugacy class in the dual torus $Z(\hat{L})$ (cf. \S \ref{parabelly});
let $l_{\theta} \in Z(\hat{L})$ be a representative.  
 Then the Satake parameter of the character of $\cH(G_v,K_v)$ on $\mathrm{J}_P^G(\theta)$ is
\begin{equation}
l_{\theta} \cdot \inverserho_G(\Frob_v)
\end{equation}
 Indeed to verify this we just choose a Borel subgroup $B_v \subset P_v$ and note that $J_P^G(\theta) \subset J_B^G(\theta)$, 
and use the previous formula.

\subsection{The $\sigma$-dual homomorphism}
\label{subsec:sigmadual}

In many cases,  the transfer of eigenvectors of Theorem \ref{thm:realtheorem}  arises from a  homomorphism of $L$-groups we call the ``$\sigma$-dual homomorphism.''  
As in \S\ref{subsec:dgalg}, let $\hat{G}$ and $\hat{H}$ be the dual groups to $\G, \HH$, and let  $\LG$ and $\LH$ be the $L$-groups to $\G,\HH$.

We have defined  at a $\sigma$-good place the normalized Brauer map
\[
\NBr: \cH(G_v, K_v) \rightarrow \cH(H_v, U_v).
\]
It seems very likely  that this arises from an {\em algebraic homomorphism}
\begin{equation}
\label{eq:sigmadualv}
\psi_v : \LH_v \rightarrow \LG_v
\end{equation}
covering the identity map on $\langle \Frob_v \rangle$. 
Indeed such a $\psi_v$ induces by Theorem \ref{Satake} a homomorphism of Hecke algebras.  
The existence of such a $\psi_v$ depends, a priori,  on the choice of local pseudoroot.

Even better, we can ask for a single homomorphism
\begin{equation} \label{eq:sigmadual}
\Lpsi: \LH \rightarrow \LG
\end{equation}
which induces
$\psi_v$ for almost all places $v$.   Again, there is an implicit choice of global pseudoroots. 

 We will call such an $\Lpsi$ a {\em $\sigma$-dual} homomorphism. 
 We will prove that these exist in ``most'' cases when $\GG$ is simply connected
 and $\HH$ semisimple (see Theorem \ref{Gsctheorem}). 
 In the presence of a $\sigma$-dual homomorphism we may reformulate  Theorem \ref{thm:realtheorem} as a functorial lift:

\begin{theorem*} \label{theoremv2} Suppose there is a $\sigma$-dual homomorphism 
$\LH \rightarrow \LG$.  If $\rho:\Gamma_F \to \LH$
is modular for $\HH$,  then $\Lpsi \circ \rho$ is modular (for $\GG$). 
\end{theorem*}
Here, we  say that $\rho: \Gamma_F \rightarrow \LG(k)$ is {\em modular}, 
with respect to a fixed choice of global pseudoroot for $\G$, if
there is a level structure $K$ and  a class $h \in H^*([G]_K)$ such that
\begin{quote}
For all but finitely many good places $v$, the class $h$ is an eigenvector for the $\cH_v$-action whose eigenvalue $\chi:\cH_v \to k$ coincides with $\rho (\Frob_v)$ under Satake.
\end{quote}

\subsection{The $C$-group and the $c$-group} 
\label{subsec:Cgroup}
 It is possible that there is no $\sigma$-dual homomorphism at all. A simple example
is provided by $\GG = \PGL_2$ over $F=\Q$ and $\sigma$ an inner automorphism of order $3$, 
with fixed points isomorphic to $\HH = \mathrm{PSO}(x^2+3y^2)$.  In that case, there is no homomorphism of $L$-groups:
$$ \mathbb{G}_m \rtimes \Gamma_{\Q} \rightarrow \SL_2 \times \Gamma_{\Q},$$
because the image of complex conjugation (considered in $\Gamma_{\Q}$, on the left-hand side)
when projected to $\SL_2$ must be an order $2$ element that normalizes but does not centralize a nontrivial torus, and none such exists. 
 
Deligne has introduced a mild modification of the $L$-group that allows one to bypass the issues of square roots.
It is termed by ``$C$-group'' by the Buzzard and Gee \cite{BG}; we will denote it $\CG$. It is the quotient
$\LG \times \Gm$ by the order $2$ element $e:= (\Sigma^*_G(-1), -1) \in  \LG \times \Gm$.    Here $\Sigma^*_G$ denotes the cocharacter $k^* \to \hat{T}(k)$ corresponding to the sum of positive roots of $\G$, as in \S\ref{subsec:localpseudoroot}.  

 Note that $\Sigma^*_G(-1)$ is always central, since the pairing of the sum of positive coroots with a root is always even.
As $\Sigma^*_G$ is a product of coroot homomorphisms, which extend to $\SL_2$, it always takes values in the commutator subgroup of $\hat{G}$.  
There is a natural projection $\CG \rightarrow \Gamma \times \Gm$: if we identify the set $\CG(k)$ with $\{1,e\}$-cosets of $\hat{G}(k) \times \Gamma \times k^*$, it carries the class of $(g,\gamma,c)$ to $(\gamma,c^2)$.

   Even when there is no global pseudoroot, one expects (see \cite[Conjecture 5.3.4]{BG} for characteristic zero) Hecke eigenclasses in $H^*([G])$ to match homomorphisms $\Gamma_F \to \CG$ whose composition with this projection is $\gamma \mapsto (\gamma, \mathrm{cyclo}(\gamma))$. One might hope that in such cases there would be a homomorphism of $C$-groups $\CH \rightarrow \CG$ over $\Gamma \times \Gm$.  But  the same example ($\GG=\PGL_2$ over $\Q$, $\sigma$ of order $3$) also contradicts this hope:

The  $C$-group
for $\PGL_2$ identified with the product  $\GL_2 \times \Gamma_{\Q}$
and the $C$-group for $\HH$ is $\left( \mathrm{SO}_2 \rtimes \Gamma_{\Q}\right) \times \Gm$. We seek
 $$\Cpsi: \left( \mathrm{SO}_2 \rtimes \Gamma_{\Q}\right) \times \Gm  \rightarrow \GL_2 \times \Gamma_{\Q}$$
where projection to $\Gm$ on the left 
 should correspond to determinant on the right.
 The projected map $\Gm \rightarrow \GL_2$ must be of the form $x \mapsto \left(  \begin{array}{cc} x^n & 0\\ 0 & x^m \end{array}\right)$
 where $n+m=1$; in particular, the centralizer of this image is the diagonal torus in $\GL_2$,
 and so $\mathrm{SO}_2 \rtimes \Gamma_{\Q}$ must map into the diagonal torus in $\GL_2$.
 But every element of $\mathrm{SO}_2$ is a commutator in $\mathrm{SO}_2 \rtimes \Gamma_{\Q}$, so that 
  $\Cpsi$ must be trivial on $\mathrm{SO}_2$.

 This problem  can be fixed by reducing the size of the $C$-group, retaining only the essential part of the $\Gm$ factor:  Suppose that the characteristic of $k$ is not $2$. We replace the $C$-group by the subgroup $\cG$ with
 \[
 \cG := \frac{\{(g, \gamma, c) \in \hat{G}(k) \times \Gamma \times k^* : c^2 = \mathrm{cyclo}(\gamma)\}}{ (\Sigma_G(-1), 1, -1)}
 \]
 Observe that this naturally fits in an extension $\hat{G} \rightarrow \cG \rightarrow \Gamma$; 
 the extension class corresponds to the $2$-cocycle that is the image
 $$\mathrm{cyclo} \in H^1(\Gamma, k^*) \stackrel{\mathrm{Bockstein}}{\longrightarrow} H^2(\Gamma, \{ \pm 1\}) \stackrel{\Sigma_G}{\longrightarrow}  \left( \mbox{extensions $\hat{G} \rightarrow ? \rightarrow \Gamma$}\right),$$
 where the last map simply constructs the extension associated to a cocycle in $H^2(\Gamma, \{ \pm 1\})$, by using $\Sigma_G: \{\pm 1\} \rightarrow \hat{G}(k)$.

With this notation, there exists a $\sigma$-dual homomorphisms ``for $c$-groups'' in the case described above:  the $c$-group of $\PGL_2$ becomes identified with
 $$ \{ (g \in \GL_2(k), \gamma \in \Gamma): \det(g) = \cyclo(\gamma) \},$$ and
the $c$-group of the torus $\HH$ is simply the semidirect product $\mathrm{SO}_2 \rtimes \Gamma$.
Then

$$(h, \gamma) \in \mathrm{SO}_2 \rtimes \gamma \mapsto (\iota(h)  w^{\chi(\gamma)}, \gamma) \in \cG$$
defines a $\sigma$-dual morphism; here $\iota: \mathrm{SO}_2 \rightarrow \GL_2$
is the standard inclusion,  $w $ is any element of $\mathrm{O}_2 - \mathrm{SO}_2$ and  $\chi: \Gamma \rightarrow \Z/2\Z$ is the quadratic character associated to $\Q(\sqrt{-3})$.

As in \S\ref{subsec:llg}, we can define a variant $\cG_v$ of the $c$-group, using only the discrete cyclic group generated by Frobenius at $v$ rather than the full Galois group of $F_v$.
This allows us to close this section with the following intrinsic form of the Satake isomorphism: 
 
\begin{theorem}[Satake isomorphism in terms of $c$-groups] \label{cSatake} 
Let $G_v = \G(F_v)$ be a reductive $v$-adic group and let $K_v \subset G_v$ be a maximal compact subgroup satisfying the conditions of \S\ref{goodplace}.    Then there is a natural isomorphism
\begin{equation}
\label{twistedcSatake}
\begin{array}{rcl}
\cH(G_v,K_v) & \stackrel{\sim}{\longrightarrow} &
\begin{array}{c}
 \mbox{\rm regular functions on} \\
 \mbox{\rm $\mbox{preimage of $\Frob_v$ in $\cG_v$} \git \hat{G}$}. 
 \end{array}
\end{array}
\end{equation}
 \end{theorem}
 
\proof  
Choose a square root  $\sqrt{q}$ of $q$ inside $k^*$. This gives, firstly, an isomorphism
\begin{equation} \label{frou-frou}  \LG_v \stackrel{\sim}{\longrightarrow} \cG_v\end{equation} 
given by $ g \rtimes \Frob_v^k \mapsto (g, (\sqrt{q})^k, \gamma).$
It also gives us a pseudoroot, namely, 
$$b \in B_v \mapsto (\sqrt{q})^{v (\delta(b))}.$$

 In this way we obtain 
the isomorphism of \eqref{twistedcSatake} by comparing   \eqref{frou-frou} with the Satake isomorphism, in the sense previously described.  If we change
$\sqrt{q}$ by $-1$ both the identification  of \eqref{twistedcSatake}, and the Satake isomorphism
change by multiplication by $\Sigma^*_G(-1) \subset \hat{G}(k)$ inside $\LG$. Thus, \eqref{twistedcSatake} is independent of choice of $\sqrt{q}$.
\qed

 \medskip
 
 In these terms, we may precisely formulate a ``better'' version of our earlier question:
 \begin{quote}
Suppose $\GG$ is semisimple and $\HH=\GG^{\sigma}$ connected.
Does there always exist a map $^c{}\hat{H} \rightarrow \cG$ which induces via Theorem \ref{cSatake}
for almost every place $v$, the normalized   Brauer map
$
\NBr: \cH(G_v, K_v) \rightarrow \cH(H_v, U_v)$?
\end{quote}

   \section{The Satake parameters of the Brauer homomorphism}  \label{BrauerSatakeproof}

\subsection{Computing the Brauer homomorphism}
\label{BrauerSatake}
Let $\G$ and $\HH = \GG^\sigma$ be as in \S\ref{sigmaax}.  Suppose that $\HH$ is connected.  Let $v$ be a $\sigma$-good place (\S\ref{subsec:sigmagood}), let $\gamma \in \Gamma_F$ be a Frobenius element at $v$, and choose a local pseudoroot at $v$ (\S\ref{subsec:localpseudoroot}) for both $\G$ and $\HH$.  
 
By Theorem \ref{Satake} and \eqref{ChosenPseudoSatake}, the normalized Brauer homomorphism gives a map
\begin{equation}
\label{specBrauer}
\Spec(\NBr):\hat{H} \rtimes \gamma \git \hat{H} \to \hat{G} \rtimes \gamma \git \hat{G}
\end{equation}
We may write the domain and codomain of this map as quotients of the maximal tori $\hat{T}_H \subset \hat{H}$ and $\hat{T}_G \subset \hat{G}$, respectively.  In this section we discuss the extent to which this map is covered by a homomorphism $\hat{N}:\hat{T}_H \to \hat{T}_G$.
The map $\hat{N}$ is not canonically specified --- there is a finite set of choices for it, indexed by what we call ``$\gamma$-admissible Borel classes'' (\S\ref{subsec:ab}).   The map $\hat{N}$  is the ``dual norm'' associated to the admissible Borel class (\S\ref{normsec}).  We will  also attach to $\gamma$ a Levi subgroup $\hat{L}_{\gamma} \subset \hat{G}$, called the associated dual Levi  (\S \ref{lem:generate} and \S\ref{normsec}). With these notations:
\begin{theorem*}
Let $\hat{N}:\hat{T}_H \rightarrow \hat{T}_G$ be the dual norm homomorphism attached to a $\gamma$-admissible Borel class, and let $\hat{L}_{\gamma}$ be the associated dual Levi to $\gamma$.
\begin{enumerate}
\item
For every $k$-point $x$ of $\hat{H} \rtimes \gamma \git \hat{H}$, there is a $t \in \hat{T}_H(k)$ such that $t\inverserho_H(\gamma) \in \LT_H$ is a representative for $x$, and $\hat{N}(t) \inverserho_G(\gamma)$  is a representative for the image of $x$ under the map \eqref{specBrauer}.
\item Moreover $t$ can be chosen such that $\hat{N}(t)$ lies in the center of $\hat{L}_{\gamma}$.
\end{enumerate}
\end{theorem*}

Note that part (1) of the Theorem does \emph{not} assert that \eqref{specBrauer} can be extended to a commutative square of the form
\begin{equation} \label{nottruesquare}
\xymatrix{
\hat{T}_H \ar@{-->}[r] \ar[d] & \hat{T}_G \ar[d] \\
\hat{H} \rtimes \gamma \git \hat{H} \ar[r] & \hat{G} \rtimes \gamma \git \hat{G}
}
\end{equation}

Roughly speaking, the proof of the Theorem goes like this:
We compute the effect of Tate cohomology on suitable spherical representations and deduce the computation of the Brauer homomorphism from \eqref{Brauercompat}. 
In turn, spherical representations are realized in the spaces of sections of suitable
line bundles over flag varieties;   the main technical step is 
extending a Borel subgroup of $H_v$ to a $\sigma$-stable parabolic subgroup of $G_v$, 
to produce ``compatible'' flag varieties for $G_v$ and $H_v$.  
It is at this step that we need to make choices (the ``$\gamma$-admissible Borel classes'') from above. 

 The arguments in this section resemble the arguments used to prove Theorem 3.3.A of \cite{KottwitzShelstad}, which produces a map similar to \eqref{specBrauer}, but on the dual side and in characteristic zero. We are grateful to Laurent Clozel for bringing this to our attention; it would be interesting to investigate further.

\subsection{Outline of this section}
In \S\ref{subsec:title}--\S\ref{normsec}, we will work with $\GG$ and $\HH$ over $\overline{F}$.  The Galois group makes its mark through its image in $\Out(\HH)$ and $\Out(\GG)$.  For a fixed element $\gamma \in \Out(\HH)$, we define certain subgroups $\TT_H^{\gamma}$, $\LL_G(\gamma)$, and a restricted class of Borels $\BB_G$ and parabolics $\PP_G$ (the former called ``$\gamma$-admissible Borels'').  Until \S\ref{subsec:rat}, \emph{all these groups are defined over $\overline{F}$}.  
In \S\ref{subsec:rat} we return to rationality issues.  If $\gamma$ is a Frobenius element at $v$, some of these groups (but not $\BB_G$) are defined over $F_v$.  
\newline

\begin{prop}
\label{subsec:title}

Let $\TT_H$ be a maximal torus in $\HH$.  Then the centralizer of $\TT_H$ is a maximal torus in $\G$.
\end{prop}

(We repeat: in \S\ref{subsec:title}--\S\ref{normsec}, all subgroups are to be taken as defined over  $\bar{F}$.) 

\begin{proof}
Let $x \in \TT_H$ be a regular semisimple element of $\HH$.  By \cite[\S 8.9]{SteinbergEndomorphisms}, or by \cite[Lemma 3.2]{Reeder}, we may find a $\sigma$-stable maximal torus and Borel of $\G$ containing $x$, and therefore containing $\TT_H$.  
Let $\TT_G \subset \G$ be such a maximal torus.  Then $\TT_H$ is the identity component of $\TT_G^{\sigma} = \TT_G \cap \HH$.  (In fact since $\HH$ is connected and $\TT_H$ is a maximal torus in $\HH$, we actually have $\TT_H = \TT_G^{\sigma}$).  The roots of $\TT_G$ on the centralizer of $\TT_H$ are those $\beta \in \Phi(\TT_G,\G)$ that are trivial on $\TT_H$, or equivalently that vanish on the Lie algebra of $\TT_H$.  We will show that there are no such $\beta$, and therefore the centralizer is equal to $\TT_G$.

The map $\mathrm{Lie}(\TT_G)^* \to \mathrm{Lie}(\TT_H)^*$ identifies the codomain with the $\sigma$-coinvariants of the domain.  Thus any $\beta$ that vanishes on $\mathrm{Lie}(\TT_H)$ belongs to the image of $1 - \sigma$, or equivalently to the kernel of $1 + \sigma + \cdots + \sigma^{p-1}$.  If $\beta$ is a positive (resp. negative) root, then each $\sigma^i(\beta)$ is also positive (resp. negative), and in particular $\beta + \sigma(\beta) + \cdots + \sigma^{p-1}(\beta) \neq 0$.  This completes the proof.
\end{proof}

\subsection{The torus $(\TT_H^{\gamma})^{\circ}$}
\label{subsec:tttth} 
We continue with the notations of the previous subsection. 
Let $\BB_H \subset \HH$ be a Borel containing $\TT_H$.
Let $\gamma$ be an outer automorphism of $\HH$.  (For instance, the image of an element of $\Gamma_F \to \mathrm{Out}(\HH)$ induced by the $F$-rational structure of $\HH$).  Then 
there exists a unique representative for
$\gamma$ in $\Aut(\HH)$ which preserves $\TT_H$ and $\BB_H$; we denote
this representative also by $\gamma$.  Let $\TT_H^{\gamma}$ denote the group of $\gamma$-fixed points, and $(\TT_H^{\gamma})^{\circ}$ the identity component of $\TT_H^{\gamma}$.

\begin{prop*}
 \begin{enumerate}
\item The centralizer of $(\TT_H^{\gamma})^{\circ}$ in $\HH$ is $\TT_H$.
\item The cone of coweights in $(\TT_H^{\gamma})^{\circ}$ that are positive on $\Phi(\TT_H,\BB_H)$ is ``open'', i.e. it does not lie in any hyperplane in $X_*((\TT_H^{\gamma})^{\circ})$
\end{enumerate}
\end{prop*}

\begin{proof}
The proof of Prop. \ref{subsec:title}, with $\sigma$ replaced by $\gamma$ and $\GG$ replaced by $\HH$, establishes (1).  
Let us prove (2).  Let $m$ denote the order of $\gamma$ in $\mathrm{Out}(\HH)$, and consider the operator  $\nu$ on  $X_*(\TT_H)$ carrying $\chi$ to $\chi + \gamma \circ \chi+ \cdots +  \gamma^{m-1} \circ \chi$.  After tensoring with $\Q$, the image of $\nu$ coincides with the kernel of $1-\gamma$, in particular the image of $\nu$ is not contained in any hyperplane
of $X_*(\TT_H^{\gamma})$.  Part (2) now follows from the fact that $\nu$ preserves the property of being positive on $\BB_H$.
\end{proof}

\subsection{The Levi $\LL_G(\gamma)$ and its derived group}
\label{C2}
We let $\LL_G(\gamma)$ denote the centralizer in $\G$ of the torus $(\TT_H^{\gamma})^{\circ}$.  It is a Levi subgroup, and $\Phi(\TT_G,\LL_G(\gamma)) \subset \Phi(\TT_G,\G)$ is given by those $\beta:\TT_G \to \Gm$ that are trivial on $(\TT_H^{\gamma})^{\circ}$.

Let us call a cocharacter $\chi:\Gm \to (\TT_H^{\gamma})^{\circ}$ \emph{generic} if its centralizer in $\GG$ is $\LL_G(\gamma)$.  Each $\delta \in \Phi(\TT_G,\G) - \Phi(\TT_G,\LL_G(\gamma))$ defines an orthogonal hyperplane  $H_{\delta} \subset X_*((\TT_H^{\gamma})^{\circ})$, and genericity is equivalent to $\chi \notin \bigcup_{\delta} H_{\delta}$.  

Note that $\LL_G(\gamma)$ is $\sigma$-stable, as is $[\LL_G(\gamma),\LL_G(\gamma)]$.  We have the following fixed-point computations:

\begin{prop*}
The following hold
\begin{enumerate}
\item $\LL_G(\gamma)^{\sigma} = \TT_H$.
\item $[\LL_G(\gamma),\LL_G(\gamma)]^{\sigma}$ is a maximal torus in $[\LL_G(\gamma),\LL_G(\gamma)]$.
\end{enumerate}
\end{prop*}
\begin{proof}
Part (1) follows immediately from part (1) of the proposition of \S\ref{subsec:tttth}.
We claim that $\sigma$ induces an inner automorphism of $[\LL_G,\LL_G]$.  Since the fixed points of an inner automorphism contain a maximal torus, and $[\LL_G,\LL_G]^{\sigma}$ is contained in a torus by part (1), we can conclude (2).

Let's prove the claim.  In fact we will prove that if $\mathfrak{g}$ is a semisimple Lie algebra over $\bar{F}$ and $a$ an automorphism of $\mathfrak{g}$ of prime order $p$, then if $a$ is not inner we cannot have $\mathfrak{g}^a$ contained in a Cartan subalgebra.  Indeed we may find a pinning $(\mathfrak{t},\mathfrak{b},\{e_{\alpha}\}_{\alpha \in I})$ of $\mathfrak{g}$ such that $a = \theta \mathrm{ad}_x$, where $\theta$ is a pinned automorphism and $x \in \mathfrak{t}$.  If $\theta$ is nontrivial then it also has order $p$, and there is a simple root $\alpha$ with $\alpha,\theta(\alpha),\theta^2(\alpha) \cdots,\theta^{p-1}(\alpha)$ all distinct.   

From the classification of semisimple Lie algebras by Dynkin diagrams, either we may choose $\alpha$ so that the $\theta^i(\alpha)$ are all orthogonal, or else $p = 2$ and $\mathfrak{g}$ has a factor of the form $\sl_3$ on which $\theta$ acts by transpose-inverse.  In the second case, take $\alpha$ to be one of the simple roots of the $\sl_3$-factor.  In either case one computes that the elements $e_{\alpha} + a(e_{\alpha}) + \cdots + a^{p-1}(e_{\alpha})$ and $e_{-\alpha} + a(e_{-\alpha}) + \cdots + a^{p-1}(e_{-\alpha})$ do not commute. 
 \end{proof}

\subsection{Admissible Borels}
\label{subsec:ab}
We will say that a Borel $\BB_G$ is \emph{$\gamma$-admissible} (with respect to $\sigma,\TT_H,\BB_H$) if it contains $\TT_G$ and there exists a cocharacter $\chi:\Gm \to (\TT_H^{\gamma})^{\circ}$ with the following properties:
\begin{enumerate}
\item[(i)] $\chi$ is positive for $\BB_H$: all nontrivial roots $\beta \in \Phi(\TT_H,\BB_H)$ satisfy $\langle \beta,\chi\rangle > 0$
\item[(ii)] $\chi$ is nonnegative for $\BB_G$: all nontrivial roots $\delta \in\Phi(\TT_G,\BB_G)$ satisfy $\langle \delta,\chi\rangle \geq 0$
\item[(iii)] $\chi$ is generic in the sense of \S\ref{C2}
\end{enumerate}

A tuple $(\TT_H,\BB_H,\TT_G,\BB_G)$, where $\TT_H \subset \BB_H$, $\TT_G = \Cent_\G(\TT_H)$, and $\BB_G$ is an admissible Borel, will be called a ``$\gamma$-admissible Borel tuple.''  The group $\HH(\overline{F})$ acts on $\gamma$-admissible Borel tuples by conjugation --- an orbit of this action is called a ``$\gamma$-admissible Borel class.''
Given a homomorphism $\Gamma \to \mathrm{Out}(\HH)$, we will say that a tuple $(\TT_H,\BB_H,\TT_G,\BB_G)$ is ``$\Gamma$-admissible'' if it is $\gamma$-admissible for every $\gamma$ in the image.  

The significance of admissibility is the following Lemma, which realizes $\BB_H$ as the $\sigma$-fixed points of a parabolic in $\GG$ (cf. discussion after \eqref{nottruesquare}). 

\begin{lemma}
\label{lem:generate}
If $\BB_G$ is a $\gamma$-admissible Borel, then $\LL_G(\gamma)$ and $\BB_G$ generate a $\sigma$-stable parabolic subgroup $\PP_G$ whose $\sigma$-fixed points are $\PP_G^{\sigma} = \BB_H$.
\end{lemma}

\begin{proof}
We'll construct $\PP = \PP_G$ by different means, and then show that its $\sigma$-fixed points are $\BB_H$ and that it is generated by $\LL_G(\gamma)$ and $\BB_G$.  

Let $\mathfrak{g}$ and $\mathfrak{h}$ denote the $\overline{F}$-linear Lie algebras of $\GG$ and $\HH$.  Similarly let $\mathfrak{t}_H$ and $\mathfrak{t}_G$ denote the Lie algebras of $\TT_H$ and $\TT_G$.  We have root space decompositions
\[
\mathfrak{g} = \mathfrak{t}_G \oplus \bigoplus_{\beta \in \Phi(\TT_G,\G)} \mathfrak{g}_{\beta} \qquad \qquad \mathfrak{h} = \mathfrak{t}_H \oplus \bigoplus_{\delta \in \Phi(\TT_H,\HH)} \mathfrak{h}_{\delta}
\]
Suppose $\chi$ witnesses the admissibility of $\gamma$, i.e. $\chi$ obeys (i), (ii), and (iii) of \S\ref{subsec:ab}.

Let $\PP \subset \G$ denote the parabolic subgroup containing $\BB_G$ whose Lie algebra is the sum of $\mathfrak{t}_G$ and those root spaces $\mathfrak{g}_{\beta}$ with $\langle \beta,\chi \rangle \geq 0$.  Then $\PP$ contains $\BB_G$.

It is clear that $\PP$ is $\sigma$-stable.  To see that $\PP^\sigma = \BB_H$, note that by assumption (i) for $\chi$, the Lie algebra of $\PP \cap \HH$ is the Lie algebra of $\BB_H$, i.e. $(\PP \cap \HH)^{\circ} = \BB_H$, and $\BB_H$ is its own normalizer in $\HH$.  

It remains to show that $\PP_G$ is generated by $\LL_G(\gamma)$ and $\BB_G$.  Since both $\LL_G(\gamma)$ and $\BB_G$ are connected, the subgroup they generate is connected as well, so this can be checked on Lie algebras, i.e. it is enough to see $\mathfrak{p} = \mathfrak{l}_G(\gamma) + \mathfrak{b}_G$.  We already have $\mathfrak{p} \supset \mathfrak{b}_G$, and $\mathfrak{p} \supset \mathfrak{l}_G(\gamma)$ follows from
\[
\mathfrak{l}_G(\gamma) = \mathfrak{t}_G \oplus \bigoplus_{\beta \mid \langle \beta,\chi\rangle = 0} \mathfrak{g}_{\beta}
\]
As all three spaces contain $\mathfrak{t}_G$, to show that $\mathfrak{p} \subset \mathfrak{l}_G(\gamma) + \mathfrak{b}_G$ it is enough to prove that
 a root of $\TT_G$ on $\mathfrak{p}$ is either a root of $\mathfrak{l}_G(\gamma)$, or a root of $\mathfrak{b}_G$.  Suppose $\mathfrak{g}_\beta \subset \mathfrak{p}$ but $\mathfrak{g}_{\beta} \not\subset \mathfrak{l}_G(\gamma)$, then $\langle \beta,\chi\rangle > 0$.  By assumption (ii) for $\chi$, it follows that $\mathfrak{g}_{-\beta}$ is not a root for $\BB_G$.  But $\BB_G$ is a Borel subgroup: we have $\Phi(\TT_G,\G) = \Phi(\TT_G,\BB_G) \amalg \left(-\Phi(\TT_G,\BB_G)\right)$, so $\mathfrak{g}_{-\beta} \not\subset \mathfrak{b}_G$ implies $\mathfrak{g}_{\beta} \subset \mathfrak{b}_G$.  This completes the proof.
 
Note the proof has shown that  $\LL_G(\gamma)$ is the standard Levi factor of $\PP_G$, generated by $\TT$, the root subgroups for simple roots $\alpha_i$ of $\BB_G$ with $\langle \alpha_i,\chi\rangle$= 0, and the roots subgroups for $-\alpha_i$.
 \end{proof}

$\gamma$-admissible Borels always exist, in fact:
\begin{lemma}
\label{pokemon}
Fix a Borel subgroup and maximal torus $\BB_H \supset \TT_H$ in $\HH$. 
\begin{enumerate}
\item For any $\gamma \in \mathrm{Out}(\HH)$, there is a Borel subgroup $\BB_G \subset \GG$ that is $\gamma$-admissible with respect to $\BB_H,\TT_H$.
\item For any group homomorphism $\Gamma \to \mathrm{Out}(\HH)$ whose image is cyclic of prime order, there is a Borel subgroup $\BB_G \subset \GG$ that is $\Gamma$-admissible with respect to $\BB_H,\TT_H$.
\end{enumerate}
\end{lemma}

The second assertion of the Lemma is not used in the proof of Theorem \S\ref{BrauerSatake}, but has a further consequence (discussed in \S\ref{BrauerSatakeCor}) that will be useful  later.

\begin{proof}
By part (2) of the Proposition of \S\ref{subsec:tttth}, we may find a cocharacter of $\TT_H^{\gamma}$ that is positive on $\BB_H$, i.e. that obeys (i) and (iii) of the conditions for admissibility.  Fix such a $\chi$.  By ``perturbing $\chi$ in $X_*(\TT_H)$,'' we may find a Borel $\BB_G$ obeying (ii).  More specifically, let $\epsilon:\Gm \to \TT_H$ be any cocharacter that does not vanish any roots of $\Phi(\TT_G,\GG)$.  For $N \in \Z$ sufficiently large, the cocharacter $N\chi + \epsilon$ also does not vanish on any root of $\GG$ and therefore determines a positive system in $\Phi(\TT_G,\GG)$.  Let $\BB_G$ be the corresponding Borel, i.e. with $\delta \in \Phi(\TT_G,\BB_G)$ if and only if $\langle \delta,N\chi +\epsilon\rangle \geq 0$.  By taking $N$ sufficiently large, we have $\frac{1}{N} \langle \delta,\epsilon\rangle > -1$ and therefore $\langle \delta,\chi \rangle \geq 0$ for all $\delta \in \Phi(\TT_G,\BB_G)$.

To prove the second assertion it suffices to show that $\BB_G$ can be chosen simultaneously $\gamma$-admissible and $1$-admissible.   When $\gamma = 1$, we have $(\TT_H^{\gamma})^{\circ} = \TT_H$, so the cocharacter $N\chi + \epsilon$ also witnesses the $1$-admissibility of $\BB_G$.
\end{proof}

\subsection{The norm and dual norm homomorphisms}  \label{normsec} 
With $\TT_H$ a maximal torus of $\HH$ and $\TT_G$ its centralizer in $\GG$, we define the norm homomorphism $\mathrm{N}:\TT_G \to \TT_H$ by 
\[
\mathrm{N}(t) = t \cdot t^{\sigma} \cdot \cdots \cdot t^{\sigma^{p-1}}
\]
If we choose $\BB_H \supset \TT_H$ and $\BB_G \supset \TT_G$, we get an induced map
\[
X^*(\BB_H) \simeq X^*(\TT_H) \stackrel{\mathrm{N}^*}{\longrightarrow} X^*(\TT_G) \simeq X^*(\BB_G)
\]
which in turn induces a map $\hat{T}_H \to \hat{T}_G$ which we call the \emph{dual norm}.  Note that (as there is no direct identification of $\TT$ with $\Tcan$, \S\ref{subsec:cantorus}) the dual norm depends on $\BB_H$ and $\BB_G$.  

When $\BB_G$ is $\gamma$-admissible with respect to $\TT_H,\BB_H$, the parabolic $\PP_G$ of Lemma \ref{lem:generate} determines (\S\ref{parabelly}) a Levi subgroup $\hat{L}_{\gamma} \subset \hat{G}$ containing $\hat{T}_G$.  The natural projection $\TT_G \to \LL_G^{\ab}$  dualizes to an inclusion
\begin{equation}
\label{eq:zhatl}
Z(\hat{L}_{\gamma}) \hookrightarrow \hat{T}_G
\end{equation}

\subsection{Rationality} 
\label{subsec:rat}
Now fix a place $v$ of $F$ at which $\HH$ is quasisplit, then we may choose $\TT_H$ and $\BB_H$ to be $F_v$-rational.  The group $\TT_G := \Cent_{\G}(\TT_H)$ is also $F_v$-rational.

If $\gamma \in \mathrm{Out}(\HH)$ is the outer automorphism corresponding to the Frobenius, then $(\TT_H^{\gamma})^{\circ}$ is the maximal split subtorus of $\TT_H$, and its centralizer $\LL_G(\gamma)$ is $F_v$-rational.  When $v$ is implicit we write $\LL_G := \LL_G(\gamma)$, for short.

If $\BB_G$ is any $\gamma$-admissible Borel (relative to $\TT_H, \BB_H$) then the corresponding parabolic $\PP_G$ of Lemma \ref{lem:generate} is $F_v$-rational, because
any character into the split torus $\left(\TT_H^{\gamma}\right)^{\circ}$ is automatically $F$-rational and $\PP_G$ can be defined via the non-negative weight spaces for such a character. 
Note that we cannot necessarily arrange for $\BB_G$ itself to be $F_v$-rational (nor will we need it), even if $\GG$ is quasisplit at $v$. 
As before we write $\hat{L}_{\gamma}$ for the standard Levi subgroup of $\hat{G}$ associated to the parabolic $\PP_G$.
 
\begin{lemma}
\label{lem:811} Notation as above,  $\TT_G \cap [\LL_G,\LL_G]$ is an anisotropic maximal torus in $[\LL_G,\LL_G]$.
\end{lemma}

\begin{proof}
We appeal to the following basic structural properties of Levi subgroups: $\TT_G \cap [\LL_G,\LL_G]$ is a maximal torus in $[\LL_G,\LL_G]$, and $[\LL_G,\LL_G] \cap \Cent(\LL_G)$ is finite.      
As $\Cent(\LL_G)$ contains the maximal split torus $(\TT_H^{\gamma})^{\circ}$, we can prove that $\TT_G \cap [\LL_G,\LL_G]$ is anisotropic by proving
\[
\TT_H \cap [\LL_G,\LL_G] = \TT_G \cap [\LL_G,\LL_G]
\]
The left-to-right containment is obvious, and since $\TT_G^{\sigma} = \TT_H$, to show equality it is enough to show that any element of the right-hand group is $\sigma$-fixed. This follows from part (2) of the proposition of \S\ref{C2}.
\end{proof}

\begin{lemma} \label{sur-langlands} 
Let $\TT_1$ and $\TT_2$ be algebraic tori over $F_v$, and suppose that $f:\TT_1 \to \TT_2$  has  anisotropic kernel.  Then precomposition with $f$ induces a surjection 
\[
\{\text{unramified $k^*$-valued characters of $\TT_2$}\} \to \{\text{unramified $k^*$-valued characters of $\TT_1$}\}
\]
\end{lemma}

\begin{proof}

Indeed, let $T_{v,i}^{\circ}$ be the maximal compact subgroup of $\TT_i(F_v)$. 
Then $f$ induces a map $T_{1,v}/T_{1,v}^{\circ} \rightarrow T_{2,v}/T_{2,v}^{\circ}$ which is {\em injective}:
Its kernel is a compact subgroup of a free abelian   group, thus trivial.  Since $k^*$ is injective as an abelian group, the result follows. 
\end{proof}

\subsection{Extension of characters}
\label{subsec:ext}

Let us say that a homomorphism $\BB_H(F_v) \to k^*$ or $\PP_G(F_v) \to k^*$ is unramified if it factors through an unramified character of $\BB_H^{\mathrm{ab}}(F_v) = \TT_H(F_v)$ or $\PP_G^{\mathrm{ab}}(F_v) = \LL_G^{\mathrm{ab}}(F_v)$.

\begin{prop*}
With notation as in \S \ref{subsec:rat}, let $\chi$ be an unramified character of $\BB_H(F_v)$.
\begin{enumerate}
\item $\chi$  extends to a $\sigma$-invariant unramified character $\chi^*$ of $\PP_G(F_v)$. 
\item $\chi^*$ may be chosen in such a way that we may choose representatives  $t_{\chi} \in \hat{T}_H$ and $t_{\chi^*} \in \Cent(\hat{L}_{\gamma}) \subset \hat{T}_G$
for the Langlands parameters of $\chi$ and $\chi^*$, respectively,  with 
\begin{equation}
\label{eq:8131}
\hat{N}(t_{\chi}) =( t_{\chi^*})^p
\end{equation}
where $\hat{N}$ is the dual norm map associated to the admissible Borel $\BB_G$ (\S\ref{normsec}). 
\end{enumerate}
\end{prop*}

\begin{proof}

Let $\mathbf{K}(F_v)$ be the kernel of the natural projection $\TT_G  \hookrightarrow \LL_G  \to \LL_G^{\ab} $.  Let $\mathrm{N}(\mathbf{K}) \subset \TT_H$ be the image of $\mathbf{K}$ under the norm map of \S\ref{normsec}.  By Lemma \ref{lem:811}, $\mathbf{K}$ and $\mathrm{N}(\mathbf{K})$ are anisotropic tori.

 The composite $\TT_G \stackrel{N}{\rightarrow} \TT_H \rightarrow \TT_H/\mathrm{N}(\mathbf{K})$ is trivial on $\mathbf{K}$,
 so it factors through $\LL_G^{\ab}$. Consider the commutative squares of $F_v$-algebraic tori and of associated dual tori
\[
\xymatrix{
\TT_G \ar[d]_-{\mathrm{N}} \ar[r]^-{g} & \LL_G^{\ab} \ar[d]^-{f} \\
\TT_H \ar[r]_-{\pi} & \TT_H/\mathrm{N}(\mathbf{K})
}
\qquad  \qquad
\xymatrix{
\widehat{\TT_H/\mathrm{N}(\mathbf{K})} \ar[r]^-{\hat{\pi}} \ar[d]_-{\hat{f}} &\hat{T}_H \ar[d]^-{\hat{N}} \\
\Cent(\hat{L}_{\gamma}) \ar[r] & \hat{T}_G
}
\]
where we use the $\gamma$-admissible Borel class to identify the duals of $\TT_G,\TT_H,$ and $\LL_G^{\ab}$ with $\hat{T}_G,\hat{T}_H,$ and $\Cent(\hat{L}_{\gamma})$ as in \S\ref{normsec}.  

Since $\mathrm{N}(\mathbf{K})$ is anisotropic, there is (by Lemma
\ref{sur-langlands})
 an unramified character of $\TT_H/\mathrm{N}(\mathbf{K})$, call it $\overline{\chi}$, with $\pi^* \overline{\chi} = \chi$.
Set $\psi = f^* \overline{\chi}$, an unramified character of $\LL_G^{\ab}(F_v)$.
Then for $t \in \TT_H(F_v)$ we have 
$$ \psi(g(t))  = \overline{\chi}(\pi \circ \mathrm{N}(t)) = \chi(t)^p.$$
In other words, $\chi^* := \psi^{1/p}$ extends $\chi$.

  If $t_{\overline{\chi}}$ is a representative for the Langlands parameter of $\overline{\chi}$, 
 then its image $t_{\chi} = \hat{\pi}(t_{\overline{\chi}}) \in \hat{T}_H$ is a representative for the Langlands parameter for $\chi$, 
 $t_{\psi} = \hat{f}(t_{\overline{\chi}}) \in \Cent(\hat{L}_{\gamma})$ is a representative for the Langlands parameter of $\psi$,
 and $t_{\psi}^{1/p}$ is a representative for the Langlands parameter of $\chi^* = \psi^{1/p}$.  (These facts are all readily deduced
 from \eqref{XT}). 
So \eqref{eq:8131} is a consequence of the commutativity of the right-hand square.
\end{proof}

\subsection{Proof of the Theorem of \S\ref{BrauerSatake}}

As in the hypotheses of the Theorem, let $v$ be a place of $F$, let $\gamma \in \Gamma_F$ be a Frobenius element at $v$, and let $\TT_H,\BB_H,\TT_G,\BB_G$ be a $\gamma$-admissible Borel tuple.  Let $\PP_G$ be the corresponding parabolic (Lemma \ref{lem:generate}), and let $\hat{N}$ be the corresponding dual norm (\S\ref{normsec}).

Fix $\theta \in \hat{H} \rtimes \gamma \git \hat{H}$.  We have to find $t \in \hat{T}_H$ such that the image of $t \inverserho_H(\gamma)$ in $\hat{H} \rtimes \gamma \git \hat{H}$
coincides with $\theta$, and  is moreover carried by the Brauer homomorphism to the image of $\hat{N}(t) \inverserho_G(\gamma)$ in $\hat{G} \rtimes \gamma \git \hat{G}$. 

We may regard (via Satake \eqref{ChosenPseudoSatake})
 $\theta$ as a character $\cH(H_v,U_v) \to k$.  By \S\ref{piSat}, there is an unramified homomorphism $\chi:\BB(F_v) \to k^*$ such that $\theta$ is the character by which $\cH(H_v,U_v)$ acts on the $U_v$-invariants of the unnormalized induction $\mathrm{J}_{B_H}^{H}(\chi)$.  We will show that we may take $t = t_{\chi}$, the element of $\hat{T}_H(k)$ corresponding to $\chi$.

Indeed, by \S\ref{subsec:ext}, $\chi$ extends to a $\sigma$-invariant character $\chi^*$ of $\PP_G(F_v)$.  Restriction from $G_v$ to $H_v$ gives a surjection  (see \eqref{eq:steve})
\begin{equation}
\label{Tate2}
\Tate^0(\mathrm{J}_{P_G}^G(\chi^*)) \to \mathrm{J}_{B_H}^H(\chi)
\end{equation}
that carries a nonzero $K_v$-invariant vector on the left (the function which is identically $1$ on $K_v$) to a nonzero $U_v$-invariant vector on the right (the function which is identically $1$ on $U_v$).
Let $\Theta$ be the character by which $\cH(G_v,K_v)$ acts on $\mathrm{J}_{P_G}^G(\chi^*)^{K_v}$.  As $\PP_G$ and $\chi^*$ are $\sigma$-fixed, we may apply \eqref{Brauercompat} and conclude
\[
\Theta\vert_{\cH(G_v,K_v)^{\sigma}} = \theta \circ \mathrm{Br}
\]

Thus  $  (\theta \circ \mathrm{Br})^p$   is the character by which
$\cH(G_v, K_v; \mathbf{F}_p)^{\sigma}$ acts on $\mathrm{J}_P^G((\chi^*)^p)^{K_v}$, by \eqref{JohnGalt}. 
Recall \S\ref{nBr} that $\NBr$ is the linear extension of $\mathrm{Br}^p$ from the $\F_p$-valued Hecke algebra. 
Since $\theta \circ \NBr$
extends  the $\sigma$-invariant character $(\theta \circ \Br)^p: \cH(G_v, K_v; \mathbf{F}_p)^{\sigma} \rightarrow k$
and such an extension is unique (\S \ref{Taterings}),  it follows that  $\theta \circ \NBr$ is the character by which $\cH(G_v,K_v)$ acts on on $\mathrm{J}_P^G((\chi^*)^p)^{K_v}$.  In other words, $\theta \circ \NBr \in \hat{G} \rtimes \gamma \git \hat{G}$ has a representative of the form $t_{\chi^*}^p \inverserho_G(\gamma)$,
where  $t_{\chi^*} \in Z(\hat{L})$ is a Langlands parameter for $\chi^*$.

 The conclusion of the Theorem now follows from \eqref{eq:8131}

\subsection{A consequence of the Theorem of \S\ref{BrauerSatake}}
 \label{BrauerSatakeCor}
 In this section we show how to use the Theorem of \S\ref{BrauerSatake} to verify that a candidate homomorphism $\LH \to \LG$ is a $\sigma$-dual homomorphism, in the sense of \S\ref{subsec:sigmadual}. The criterion we derive here is not applicable in every case we check --- sometimes one has to use \S\ref{BrauerSatake} directly --- but it is nonetheless very useful.

Suppose that we may find a $\Gamma$-admissible Borel tuple (\S\ref{subsec:ab}), and let $\hat{N}$ be the associated dual norm map (\S\ref{normsec}).  Let $\Lpsi:\LH \to \LG$ be a homomorphism over $\Gamma_F$, not a priori known to be a $\sigma$-dual homomorphism.  If
\begin{enumerate}
\item[(a)] $\Lpsi$ agrees with $\hat{N}$ on $\hat{T}_H$, and 
\item[(b)] for every $\gamma \in \Gamma_F$, the elements $\Lpsi(\inverserho_H(\gamma))$ and $\inverserho_G(\gamma)$ project to the same element of $\hat{G} \rtimes \gamma \git \hat{L}_\gamma$,
\end{enumerate}
then we claim that $\Lpsi$ is a $\sigma$-dual homomorphism. 

 Whenever (a) holds, (b) is implied by the following condition, which we will often use to verify it:
\begin{enumerate}
\item[($\beta$)] For every $\gamma \in \Gamma_F$, the elements $\Lpsi(\inverserho_H(\gamma))$ and $\inverserho_G(\gamma)$ of $\hat{G} \rtimes \gamma$ centralize $\Cent(\hat{L}_\gamma)$ and have the same prime-to-$p$ part, i.e. $\Lpsi(\inverserho_H(\gamma))^{p^n} = \inverserho_G(\gamma)^{p^n}$ for some positive integer $n$.\end{enumerate}

  \proof 

   Indeed, choose a representative
  $t \inverserho_H(\gamma)$ for an element of $\hat{H} \rtimes \gamma \git \hat{H}$, as in the statement of Theorem \ref{BrauerSatake}.  Set $\alpha = \Lpsi(\inverserho_H(\gamma))$ and $\beta = \inverserho_G(\gamma)$.  We need to see that
  $ \Lpsi(t \inverserho_H(\gamma)) = \hat{N}(t)  \alpha $  and 
  $\hat{N}(t)  \beta$ define the same classes in  $\hat{G} \rtimes \gamma \git \hat{G}$, i.e
  $ f(\hat{N}(t) \alpha ) = f(\hat{N}(t) \beta )$
for any $\hat{G}$-invariant algebraic  function $f$ on $\hat{G} \rtimes \gamma$.  But note that $x \mapsto f(\hat{N}(t) x)$ actually 
defines a regular function on   $\hat{G} \rtimes \gamma \git \hat{L}_{\gamma}$, so
the claim follows from hypothesis.  

To see that ($\beta$) implies (b):  consider the centralizer $\hat{L}^{+}$ of  $Z(\hat{L}_{\gamma})$ inside
$\hat{G} \rtimes \langle \gamma \rangle$. Its identity component is $\hat{L}_{\gamma}$.
  Thus $\alpha, \beta \in \hat{L}^+ $. For any linear representation $\rho$ of $\hat{L}^+$ we 
 we have
$ \mathrm{tr} \ \rho(\alpha) = \mathrm{tr} \  \rho(\beta)$ because $\alpha^{p^n} = \beta^{p^n}$. 
But the traces of such representations span all $\hat{L}_{\gamma}$-invariant functions
on the component of $\hat{L}^+$ containing $\alpha$ and $\beta$;
that was proved in the Lemma of \S\ref{ReardenMetal} and concludes the proof. 
 \qed

Finally, we have the following Lemma, which will be useful later:

\begin{lemma} \label{ZNL}
Fix a $\gamma$-admissible class $(\TT_H,\BB_H;\TT_G,\BB_G)$ for some $\gamma \in \Gamma_F$.  Let $\hat{N}$ be the dual norm map associated to the Borel class, and
let $\hat{T}_H^{\gamma}$ denote the $\gamma$-fixed points on $\hat{T}_H$.  Suppose also that $\mathfrak{h}$ is semisimple and its absolutely simple factors 
have multiplicity at most $1$. 
Then
 \begin{equation} \label{incl} \Cent_{\hat{G}}\left(   \hat{N} (\hat{T}_H^{\gamma})^{\circ} \right)  =  \hat{L}_{\gamma}. \end{equation}   \end{lemma} 
\proof 
We need to prove that, for a root $\alpha$ of $\GG$, 
it is equivalent for $\alpha$ to be trivial on $\left( \TT_H^{\gamma}\right)^{\circ}$,
and for the associated root $\alpha_*^{\vee}$ of $\hat{G}$
to be trivial on $\hat{N}(\hat{T}_H^{\gamma})^{\circ}$. Here, the subscript $\circ$ denote connected component.

The latter condition is equivalent to 
$$ \alpha^{\vee}_*  \circ \hat{N} \mbox{ is trivial on }  \left(  \hat{T}_H^{\gamma}\right)^0.$$
or equivalently
$N \circ \alpha^{\vee}: \mathbb{G}_m \rightarrow \TT_H$ 
should project trivially to coinvariants $(\TT_H)_{\gamma}$, or, what is the same,
$N \circ \alpha^{\vee}$ takes image in $(\gamma-1)\TT_H$.

Write $N_{\gamma}: \TT_H \rightarrow \TT_H$ for the ``$\gamma$-norm''   $ \sum_{j=1}^{o} \gamma^j$
 where $o$ is the order of $\gamma$ acting on $\hat{H}$.  Also write $\iota$ for the inclusion of $\TT_H$ into $\TT_G$.  
 Thus $(\gamma-1)\TT_H$ is the connected component of the kernel of $N_{\gamma}$;
 on the other hand, 
 $\alpha$ is trivial on $\left( \TT_H^{\gamma}\right)^{\circ}$ if and only if
 $\alpha \circ \iota N_{\gamma}$ is trivial on $\TT_H$, i.e.
 if $\alpha \circ  \iota N_{\gamma}  N$ is trivial on $\TT_G$. 
 So we need to check that 
 $$ \mbox{
 $\alpha \circ \iota N_{\gamma}  N = 0$ if and only if $ \underbrace{\iota}_{\TT_H \rightarrow \TT_G} \underbrace{N_{\gamma}}_{\TT_H \rightarrow \TT_H}   \underbrace{N}_{\TT_G \rightarrow \TT_H}  \circ \underbrace{\alpha^{\vee}}_{\G_m \rightarrow \TT_G}= 0$. } $$
 
Computing with  Lie algebras, we see that  it is enough to check that $\iota N_{\gamma} N$,
considered as a map $\mathfrak{t}_G \rightarrow \mathfrak{t}_G$,
is self-adjoint with respect to
the Killing form on $\mathfrak{g}$. Equip  $\mathfrak{t}_H$ with the restricted Killing form from $\mathfrak{g}$. 
Factor this map as 
$$ \mathfrak{t}_G \stackrel{N}{\rightarrow} \mathfrak{t}_H \stackrel{N_{\gamma}}{\rightarrow} \mathfrak{t}_H \stackrel{\iota}{\hookrightarrow} \mathfrak{t}_G.$$
The adjoint of $N: \mathfrak{t}_G \rightarrow \mathfrak{t}_H$ with respect to the Killing form is
$p \iota$, and the adjoint of $\iota$ is $N/p$.

So it is enough to verify that $N_{\gamma}: \mathfrak{t}_H \rightarrow \mathfrak{t}_H$ is self-adjoint with respect
to the restricted Killing form; and for that it is enough to see that $\gamma$  preserves the restricted Killing form. 
But on each simple factor of $\mathfrak{h}$, the restricted Killing form from $\mathfrak{g}$
is proportional to the Killing form for $\mathfrak{h}$, and (by virtue of the assumption on $\mathfrak{h}$) the action of $\gamma$ preserves the splitting into simple factors.
   \qed  

 We could have reached the same conclusion under the following assumption:
   Split $\mathfrak{h} = \bigoplus \mathfrak{h}_i$, 
 where each $\mathfrak{h}_i$ is a sum of simple Lie algebras of the same type (and the types of different $\mathfrak{h}_i$ are distinct).
 Then we require that the action of $\gamma$ on each $\mathfrak{h}_i$ be induced by an element of $\Aut(\mathfrak{g})$. 
 
\section{Construction of $\sigma$-dual homomorphisms 1: Preliminaries} \label{section:prelim}

The remainder of this paper is devoted to the proof of the second main theorem, i.e.
  \begin{theorem} \label{Gsctheorem}
 Suppose that $\GG$ is simply connected and $\HH=\GG^{\sigma}$ is semisimple, where $\sigma$ is a
 $F$-rational automorphism of $\GG$ of prime order.  
 Suppose also that $(\mathrm{Lie}(\GG), \mathrm{Lie}(\HH))$ does not contain a factor isomorphic to
 $(\mathfrak{e}_6, \mathfrak{sl}_3^3  \mbox{ or } \mathfrak{sl}_6 \times \mathfrak{sl}_2 \mbox{ or } \mathfrak{sp}_8)$.
  Then there exists a $\sigma$-dual homomorphism, with respect
 to  canonical pseudoroots (see Proposition \ref{canpseudo}).   \end{theorem} 

In this   somewhat miscellaneous section, we reduce  to the case where $\GG$ is absolutely almost simple.
 Part of this reduction involves proving the Conjecture in the important case of cyclic base change.  
 After these reductions, what is left is a finite list of cases, summarized in the Tables at the end of the section.

\subsection{Cyclic base change}
\label{CBC}
Let $E/F$ be a cyclic extension, let $\HH$ be a simply connected group over $F$,  and let $\GG = \mathrm{Res}_{E/F} (\HH \otimes_F E)$ (thus also simply connected). 
Then a generator $\sigma$ for $\Aut(E/F)$ induces an automorphism of $\GG$ fixing $\HH$.  
In this case there is a unique \sigmastable Borel class. 

We claim that the canonical ``diagonal-restriction'' map on $L$-groups $\LH \rightarrow \LG$ 
is a $\sigma$-dual homomorphism.  (For generalities on the $L$-group of a restriction of scalars, we refer to \cite[\S 5]{Borel}). That follows from \S\ref{BrauerSatake}
and the following facts:  with respect to the canonical map $\iota: \LT_H \rightarrow \LT_G$,
the canonical pseudoroots pseudoroots satisfy $\inverserho_G = \iota \circ \inverserho_H$
and also moreover $\iota|_{\hat{T}_H}$ is simply the dual norm. 

\begin{prop} \label{anyformworks} 
Suppose that $\GG$ has simply connected cover $\GG'$.
Let $\hat{G}'$ be the dual group to $\GG'$.  Suppose that   canonical pseudoroots exist for $\GG, \GG', \HH, \HH'=\mathrm{fix}(\sigma)$.  If $(\GG, \sigma)$ has a $\sigma$-dual homomorphism $\Lpsi: \LH \rightarrow \LG$, extending
a dual norm map,  then so  also does
$(\GG', \sigma)$. 
\end{prop}
  Note that the converse statement does not hold: if there is a $\sigma$-dual homomorphism
 for $\widetilde{\GG}$, we cannot make any deduction about $\GG$. 
 
\proof 
Let $\HH'$ be the fixed points of $\sigma$ on $\GG'$.
The Hecke algebra for $\GG'$ is identified with a subalgebra of the Hecke algebra for $\GG$
and similarly for $\HH$.  Indeed, 
let  
$K_v \subset G_v$ be a hyperspecial maximal compact subgroup,
with preimage $K_v'$. We may identify $(G'_v/K_v' )$ with a subset of $(G_v/K_v)$; 
moreover, if two elements in $G_v'/K_v'$ are in the same $K_v$-orbit, they are also in the same $K_v$-orbit. 
Then the identification is given by ``extension by zero.'' 

For almost all $v$, the following diagram commutes, where the middle vertical arrows are the inclusions just mentioned
\begin{equation} \label{passagetocover1} 
\xymatrix{
 k[   \hat{G}' \rtimes \Frob_v \git  \hat{G}] \ar[d]   &  \mathcal{H}(G_v')  \ar[l]_{\qquad \simeq} \ar[d]\ar[r]^{\NBr}  & \mathcal{H}(H'_v) \ar[d] \ar[r]^{\simeq \qquad}  & k[   \hat{H}' \rtimes \Frob_v \git  \hat{H}]  \ar[d] \\
 k[ \hat{G} \rtimes \Frob_v  \git \hat{G}_v]    &\mathcal{H}(G_v) \ar[l]^{ \qquad \simeq}  \ar[r]_{\NBr}   &   \mathcal{H}(H_v)  \ar[r]_{\simeq \qquad}  &  k[   \hat{H} \rtimes \Frob_v \git  \hat{H}] 
}
\end{equation}  We discuss only commutativity of the central square: It clearly commutes
if $\NBr$ is replaced by the un-normalized $\Br$ (vertical arrows are extension by zero, the horizontal arrows are then restriction). The two compositions in the central square are thus seen to be the same on $\cH(G_v', K_v'; \mathbf{F}_p)^{\sigma}$
and by \S \ref{Taterings} the same on $\cH(G_v', K_v')$. 
 
 It follows from these remarks that, if the map $\Lpsi: \LH \rightarrow \LG$
 covers a  homomorphism $\Lpsi': \LH' \rightarrow \LG'$, then $\Lpsi'$ is also a $\sigma$-dual homomorphism. 
To verify that $\Lpsi$ indeed descends to such a $\Lpsi'$, it is enough to check that its projection to  
$\LG'$ is trivial on the kernel of $\LH \rightarrow \LH'$, that is, the kernel of
$\hat{H} \rightarrow \hat{H}'$.  (This kernel is regarded as a group-scheme; it may have a single $k$-point.) But that kernel is contained inside $\hat{T}_H$;
and the dual norm map $\hat{T}_H \rightarrow \hat{T}_G$ visibly 
  covers a dual norm map $\hat{T}_H' \rightarrow \hat{T}_G'$, 
  in particular, is trivial on $\ker(\hat{T}_H \rightarrow \hat{T}_H') \simeq \ker(\hat{H} \rightarrow \hat{H}')$. 
\qed

\subsection{Reduction to the case where $\G$ is absolutely almost simple}
\label{subsec:rttcwGiaas}

 Suppose now that $\G$ is simply connected over $F$. 
We claim that the existence of a $\sigma$-dual homomorphism reduces
to the absolutely simple case.  Start by noting: 

\begin{itemize}
\item[(i)] If $(\G, \sigma) = (\G_1, \sigma_1 ) \times (\G_2, \sigma_2)$, 
and the $\G_i$ admit $\sigma$-dual homomorphisms, then so does $\G$.

\item[(ii)]  If $E \supset F$ is an extension field
and $(\G, \sigma) = \mathrm{Res}_{E/F} (\G', \sigma')$,  
and $\G'$ admits a $\sigma$-dual homomorphism, then so does $\G$.  
(Here the  $L$-group of $\G$ is the ``induction'' of the $L$-group of $\G'$, 
and we just do direct computations with induced groups.)

 \end{itemize}

By a version of  \cite[\S 6.21 (ii)]{BorelTits}, there
is an {\'e}tale $F$-algebra $F'$,
an $F'$-group $\GG'$ with {\em absolutely almost simple} geometric fibers,   
so that $\GG = \mathrm{Res}_{F'/F} \GG'$, 
and moreover an  $F$-automorphism $\tau$ of $F'$
and an $\tau$-linear automorphism $\widetilde{\tau}$ of $\GG'$
which induce $\sigma$.

Now split $(F', \tau)$ as a product of irreducible factors.  
There's a corresponding splitting of $\GG'$, and it is thus enough to consider the case that $(F', \tau)$ is irreducible, i.e.
either   $F'$ is a field, or
  $F'$ is a sum of $p$ isomorphic fields, permuted by $\tau$; in that case there is also
a corresponding decomposition of $\GG' \simeq \GG_0^p$, the factors being permuted by $p$.  
 The second case is easy to check by hand (indeed, it amounts to \S  \ref{CBC} in the degenerate case where $E = F^{\oplus p}$). 

 In the first case ($F'$ a field)
 we use $\widetilde{\tau}$ to descend $\GG'$ to a group $\GG^*$ over the fixed field $F^* \subset F'$ so that there's an isomorphism
$$ \G^* \otimes_{F^*} F' \simeq \G',$$
and now $\widetilde{\tau}$ on the right-hand side is induced by $\tau \in \mathrm{Aut}(F'/F^*)$.  

Set $\G'' = \mathrm{Res}_{F'/F^*} (\G^* \otimes_{F^*} F')$, and let $\sigma''$ be the automorphism
of $\G''$ defined by $\tau \in \Aut(F'/F^*)$. 
Then
$(\G, \sigma)   =\mathrm{Res}_{F'/F} (\G', \widetilde{\tau}) = \mathrm{Res}_{F^*/F} (\G'' , \sigma'')$. Together with reduction (ii), this is cyclic base change, which we have already discussed (\S \ref{CBC})

\begin{prop}[Canonical pseudoroots] \label{canpseudo} 
When $\G$ is simply connected and $\HH$ is semisimple, both $\G$ and $\HH$ have canonical pseudoroots
(see \S \ref{subsec:globalpseudoroot} for definition). 
\end{prop}
 
\begin{proof}
When $p = 2$, the trivial section is a canonical pseudoroot and the proposition is trivial, so let us suppose $p > 2$.  As $\G$ is simply connected, the sum of roots for $\G$ is always divisible by $2$ in the weight lattice, but there is something to check for $\HH$.  

As in \S\ref{subsec:rttcwGiaas}, we may reduce to the case where $\G$ is almost simple.  We consult the classification given in the tables of $\S\ref{zippy}$.
Either
\begin{enumerate}
\item  $\sigma$ is an order $3$ outer automorphism of $\mathrm{Spin}_8$ and $\HH$ is one of $\mathrm{G}_2$, $\PGL_3$.   
\item $\G$ is an exceptional group and $\sigma$ is an inner automorphism of order $3$ or $5$.
\end{enumerate}
The outer cases are readily checked.  In the inner cases (2), we see from the classification that in every case but one, the sum of positive roots of $\HH$ is divisible by $2$ even in the adjoint form (because the Lie algebra of $\HH$ is a product of copies of $\mathfrak{sl}_k$ for $k$ odd.)  The exception is the order $3$ automorphism of $\mathrm{E}_7$, whose fixed points have $\mathfrak{h} \simeq \mathfrak{sl}_3 \times \mathfrak{sl}_6$.  We verify that $\Sigma_H$ is even by hand, using \S\ref{DTT1}.
\[
\Sigma_H = 96\alpha_{1} + 84\alpha_{2} + 96\alpha_{3} + 160\alpha_{4} +
132\alpha_{5} + 96\alpha_{6} + 52\alpha_{7}
\]
where $\alpha_1,\ldots,\alpha_7$ are the simple roots of the simply-connected form of $\mathrm{E}_7$ with their Bourbaki numbering.
\end{proof}

\subsection{Overview of the remainder of the proof} 
\label{zippy}

We have verified in \S\ref{subsec:rttcwGiaas} that it suffices to treat the case where $\G$ is almost simple simply connected.  In that case, automorphisms of $\G \times_F \overline{F}$ whose fixed points are a semisimple group have been classified.   

We now recall the classification of prime order automorphisms of simple Lie algebras over an algebraically closed field   as a brute list (omitting the several interesting automorphisms of composite order), and refer to \cite{Reeder} for a more subtle discussion.   For the outer cases, the ``comments'' column list the name of the automorphism given in loc. cit., \S 4.

Let $\mathfrak{g}$ be a simple Lie algebra over $\overline{F}$, and let $\sigma$ be an automorphism of prime order whose fixed point subalgebra $\mathfrak{g}^{\sigma}$ is semisimple.
We tabulate the cases below (the two tables give inner and outer cases) 
together with references to where Theorem \ref{Gsctheorem} is proved, i.e.
we verify the existence of a $\sigma$-dual homomorphism for some group $\GG$ (not necessarily the simply connected one) 
with the corresponding Lie algebra. This is enough by
 Proposition \ref{anyformworks}.

{\tiny 
\begin{figure}[!htbp]
\subfloat{
  \begin{tabular}{|c|c|c|c|c|}
 \hline \hline
 group &  fixed subgroup   & comments & $p$ & ref \\ 
 \hline
 $\sl_{2n+1}$ & $\so_{2n+1}$ &$\sigma_0 $   & 2 & \ref{crit4} \\ 
 $\sl_{2n}$ & $\sp_{2n}$ & $\sigma_0$   & 2 & \ref{crit4} \\ 
 $\sl_{2n}$ & $\so_{2n}$ &  $\sigma_n$   & 2 &  Th \ref{casebash4thm} \\ 
 $\so_{2n+2}$ & $\so_{2a+1} \times \so_{2b+1}$ &   & 2 &   Th \ref{casebash4thm} \\
 $\mathfrak{so}_8$ & $\mathfrak{g}_2$ & $\sigma_0$ & $3$ & \ref{crit4}  \\ 
  $\mathfrak{so}_8$ & $\mathfrak{sl}_3$ & $\sigma_2$ & $3$ &  Th \ref{casebash4thm} \\ 
  $\mathfrak{e}_6$ & $\mathfrak{f}_4$ & $\sigma_0$ & 2 & \ref{crit4}  \\
  $\mathfrak{e}_6$ & $\mathfrak{sp}_8$ &$\sigma_4$ & 2 &  \\
  \hline \hline 

 \end{tabular}
  }
 \subfloat{
   \begin{tabular}{|c|c|c|c|c|}
 \hline \hline
 group &  fixed subgroup   & comments & $p$ & ref \\ 
 \hline
 $\mathfrak{so}_{2n}$ & $\mathfrak{so}_{2a} \times \mathfrak{so}_{2b}$ &    $a,b \geq 2$ & 2 &  Th \ref{casebash4thm}   \\
$ \mathfrak{so}_{2n+1}$ & $ \mathfrak{so}_{2a+1} \times \mathfrak{so}_{2b}$ &  $b \geq 2$ & 2 & \ref{crit3}  \\
$\mathfrak{sp}_{2n}$ &  $\mathfrak{sp}_{2a} \times \mathfrak{sp}_{2b}$  &   $a,b \geq 1$ &2 & \ref{crit3}  \\
$\mathfrak{g}_2$ & $\mathfrak{sl}_3$ &     & 3 & \ref{crit6}  \\
$ \mathfrak{g}_2$ &  $ \mathfrak{so}_4$ &    & 2& \ref{crit3}  \\ 
$\mathfrak{f}_4$ & $\mathfrak{sl}_3 \times \mathfrak{sl}_3$ &  	& 3 & \ref{crit6} \\ 
$  \mathfrak{f}_4$ &  $ \mathfrak{sp}_6 \times \mathfrak{sl}_2$ &      	&2&  \ref{crit3} \\
 $  \mathfrak{f}_4$  & $  \mathfrak{so}_9$ & 	&2& \ref{crit3} \\
 $  \mathfrak{e}_6$ &  $  \mathfrak{sl}_3^{\times 3}$  & 	&3& \\ 
$\mathfrak{e}_6$  &  $ \mathfrak{sl}_6 \times \mathfrak{sl}_2$. & 		& 2 &\\
  $  \mathfrak{e}_7$ &  $  \mathfrak{sl}_3 \times \mathfrak{sl}_6$ & & 3 &\ref{crit6}   \\ 
  $ \mathfrak{e}_7$  &  $  \mathfrak{sl}_8$ & 	& 2 & \ref{crit3}   \\
  $ \mathfrak{e}_7$  &  $ \mathfrak{sl}_2 \times \mathfrak{so}_{12}$ & 	& 2 & \ref{crit3}  \\ 
$ \mathfrak{e}_8$  & $ \mathfrak{so}_{16}$ & 	& 2 &  \ref{crit3} \\ 
$ \mathfrak{e}_8$ &   $ \sl_9$ & 	 & 3 & \ref{crit6}  \\
$\mathfrak{e}_8$  & $  \sl_5 \times \sl_5$ & 	& 5 &  \\
$\mathfrak{e}_8$ &$  \sl_2 \times \mathfrak{e}_7$ &	& 2 & \ref{crit3} \\
$\mathfrak{e}_8$ & $ \sl_3 \times \mathfrak{e}_6$& 	& 3 &  \ref{crit6} \\
\hline\hline
 \end{tabular}
}
\end{figure}
}

\section{Construction of $\sigma$-dual homomorphisms 2:  Inner cases}
\label{sec9}

In this section, which is part of the proof of the second main theorem \ref{Gsctheorem}, we handle
inner cases with $p=3$, except for $\mathrm{E}_6$.
Throughout this section, suppose $\GG$ is a simply connected algebraic group over $F$, and that $\sigma$ is an inner automorphism of $\GG$ whose fixed points $\HH$ are semisimple.

By a \emph{Borel triple} for $\GG, \HH$ we will mean a triple $(\TT,\BB_H,\BB_G)$ where $\TT$ is a maximal torus in $\HH \times_F \overline{F}$ and $\BB_H$, $\BB_G$ are Borel subgroups of $\HH, \GG$ containing $\TT$.  Note that a Borel triple leads to an identification of canonical tori $\Tcan_G \simeq \TT \simeq \Tcan_H$, and therefore of dual tori $\hat{T}_H \simeq \hat{T}_G$ --- we denote this identification by $\psi_1:\hat{T}_H \stackrel{\sim}{\to} \hat{T}_G$.  (This differs from the dual norm \S\ref{normsec}, which in the inner case is a composition of $\psi_1$ with a Frobenius endomorphism.)

\begin{prop}
\label{psi1sec}
Let $(\TT,\BB_H,\BB_G)$ be a Borel triple and let $\psi_1$ denote the corresponding identification $\hat{T}_H \stackrel{\sim}{\to} \hat{T}_G$.  Then $\psi_1$ extends to an injection of dual groups $\psi_1':\hat{H} \hookrightarrow \hat{G}$.  This extension is moreover unique up to $\hat{T}_H$-conjugacy. 
\end{prop}

As any two maps $\psi_1'$ are $\hat{T}_H$-conjugate, the image of $\hat{H}$ and $\hat{B}_H$ under $\psi_1'$ are well-defined.  We will denote them by $\hat{H}_1$ and $\psi_1'(\hat{B}_{H})$.  If we fix a pinning of $\hat{H}_1$ (with $\psi_1'(\hat{B}_{H})$ as the pinned Borel) there is a unique $\psi_1'$ inside the $\hat{T}_H$-conjugacy class of morphisms $\hat{H} \to \hat{G}$ preserving the pinning.

\begin{proof}
By \S\ref{DTT1} the pair $(\GG, \HH)$ is isomorphic over $\overline{F}$ 
to a pair as considered in  Theorem \ref{DTT1}. Now 
Theorem \ref{DTT2} establishes the claim for a possibly different choice of Borel subgroups
$\BB_H, \BB_G$.  To obtain $\psi_1'$, one just applies suitable Weyl group elements.  

As to uniqueness, suppose $\psi_1', \psi_1''$ are two possible extensions.
Consider any one-parameter unipotent root group $u: \mathbb{G}_a \rightarrow \hat{H}$ associated to a simple root. 
Then $\psi_1 u, \psi_1' u: \Ga \rightarrow \hat{G}$ both transform under the same character of $\hat{T}_G$, 
and must differ then by a scaling automorphism of $\Ga$. Modifying by a suitable element of $\hat{T}_H$, we may suppose
that $\psi_1, \psi_1'$   act the same way on all 
the unipotent root groups in $\hat{H}$ for simple roots, and then they are the same on all of $\hat{H}$. 
\end{proof}

\subsection{Weyl groups and Galois action} 
\label{wgaga}
 
Let $(\TT,\BB_H,\BB_G)$ be a Borel triple.  Let $W_G$ and $W_H$ be the Weyl groups of $\GG$ and $\HH$.  For each $w \in W_G$, the intersection $\mathrm{Ad}(w) \BB_G \cap \HH \subset \HH$ is a Borel subgroup of $\HH$, and there is therefore a unique $\nu_w \in W_H$ for which 
\[\mathrm{Ad}(w) \BB_H \cap \HH = \mathrm{Ad}(\nu_w) \BB_H.\]
 
The map $\nu: W_G \rightarrow W_H$ is 
 left $W_H$-equivariant. The preimage $\nu^{-1}(1) \subset W_G$
corresponds exactly to those $w$ for which   $\Ad(w) \BB_G \supset \BB_H$
and is a set of coset representatives for $W_H \backslash W_G$. 
We call this set $W_{H \backslash G}$.  
Then the composite 
 $$W_G \rightarrow   W_H \backslash W_G \simeq W_{H \backslash G},$$
will be denoted by $w \mapsto \bar{w}$; explicitly, 
$\bar{w} = \nu_w^{-1} w$.

Now we examine the Galois action.  
Suppose $\TT$ to be chosen $F$-rational. 
There is now a  unique $w(\gamma) \in W_G$ so that
$\BB_G^{\gamma} = \Ad(w(\gamma)) \cdot \BB_G$. 
Then also $\BB_H^{\gamma} = \Ad(\nu_{w(\gamma)}) \BB_H$. 
The action of $\gamma$ on $X_*(\BB_G)$ is given by the composite:
\begin{equation} \label{BBGaction}  X_*(\BB_G) \simeq X_*(\BB_G^{\gamma}) \xrightarrow{\Ad(w(\gamma))^{-1}}  X_*(\BB_G). \end{equation}
  and now comparing this with the corresponding sequence for $\BB_H$, we see that 
\begin{equation} \label{dalecooper} \mbox{ action of $\gamma$ on $X_*(\BB_H)$} = \Ad(\overline{w(\gamma)}) \circ  \mbox{ action of $\gamma$ on $X_*(\BB_G)$} \end{equation}%
where we emphasize that  we have identified $X_*(\BB_H)$ and $X_*(\BB_G)$ using
$(\TT, \BB_H, \BB_G)$. 
This implies that \begin{equation} \label{moof} \gamma \mapsto \overline{w(\gamma)} \rtimes \gamma \in W_G \rtimes \Gamma_F \simeq W_{\hat{G}} \rtimes \Gamma_F \end{equation}  is a homomorphism,
where $\Gamma_F$ is acting on the Weyl group of $\GG$ according to its outer action on $\GG$.

 \begin{prop} \label{LpsiInner}
Let $\GG$ be a simply-connected semisimple group over $F$ and let $\HH$ be the fixed points of an inner automorphism of order $p$ of $\GG$.  Suppose moreover that
\begin{itemize}
\item[(i)]  There exists a  $\Gamma_F$-admissible Borel class in the sense of \S \ref{subsec:ab}. 

\item[(ii)] The center of $\hat{H}$ has only one $k$-point (e.g., it is a group scheme of $p$-power order). \end{itemize} 
Then there is  a homomorphism $\Lpsi:\LH \rightarrow \LG$ extending a dual norm map of \S\ref{normsec}.
    \end{prop}
  
  \proof

  Let $(\TT, \BB_H; \TT, \BB_G)$ be a $\Gamma_F$-admissible class,
  and let $\psi_1':\hat{H} \rightarrow \hat{G}$ be 
  as in \S \ref{psi1sec}, with $\hat{H}_1 $ the image of $\psi_1'$; fix a  pinning of $\hat{H}_1$
  which makes $\psi_1'$ into a pinned map. 
  
  Then $\psi_1'$ extends the identification $\hat{T}_H \rightarrow \hat{T}_G$
  that arises from the choices of Borel subgroups; on the other hand, the dual norm map
  $\hat{T}_H \rightarrow \hat{T}_G$ is
 the $p$th power of this identification.  As $\hat{H}$ is defined in characteristic $p$ and has a natural $\mathbf{F}_p$-rational structure, the $p$th power map on $\hat{T}_H$ extends to a $k$-linear Frobenius homomorphism $\Frob:\hat{H} \to \hat{H}$ splitting $\hat{T}_H$.  So it suffices to show that $\psi_1' \circ \Frob: \hat{H} \rightarrow \hat{G}$
 can be extended to $L$-groups.

Let   $\alpha$ be a root of $\hat{T}_H$ on $\hat{H}$, 
let  $U_{\alpha} \subset  \hat{H}$ be the unipotent root group for $\hat{H}$ associated to $\alpha$, 
 and let $\beta = \gamma \alpha$, the image of $\alpha$ under the 
 pinned automorphism of $\hat{H}$  corresponding to some $\gamma \in \Gamma_F$. Then  inside $\hat{G}$ we have an equality  
$$  \Ad(\overline{w(\gamma)} \rtimes \gamma) \cdot \psi_1' U_{\alpha}= \psi_1'  U_{\beta}$$
because they  are both unipotent subgroups
 corresponding to the same character of $\hat{T}$ (by \eqref{dalecooper}). 
  In particular,  the elements $\overline{w(\gamma)} \rtimes \gamma \in W_{\hat{G}} \rtimes \Gamma_F$ for $\gamma \in \Gamma_F$ normalizes $\hat{H}_1$
  because it permutes the nontrivial root groups.

   Moreover, because of  (ii), the element   $\overline{w(\gamma)} \rtimes \gamma \in W_{\hat{G}} \rtimes \Gamma_F$   can be lifted {\em uniquely} to an element   $\varpi_{\gamma} \rtimes \gamma  \in N_{\hat{G}}(\hat{T})(k) \rtimes \Gamma_F$ inducing a pinned automorphism of $\hat{H}_1$.  (In many cases, we will write down a formula for $\varpi_{\gamma}$
in the course of proving Proposition \ref{crit6}). 
In particular,
we have an equality of maps $\hat{H} \rightarrow \hat{H}_1$: 
\begin{equation} \label{equ} \psi_1' \circ \gamma  =     \Ad( \varpi_{\gamma}  \rtimes \gamma) \cdot \psi_1'\end{equation}  
 since both sides act the same way on roots and respect pinnings. 
Now  define 
\begin{equation} \label{LpsiInnerDef} \Lpsi: h \rtimes \gamma  \in \LH  \mapsto \psi_1'(\Frob( h)) \left( \varpi_{\gamma} \rtimes \gamma \right)  \in \LG\end{equation}   where $\Frob$ is the Frobenius. 
  To verify this really is a homomorphism just amounts to checking that 
$
  \psi_1' (\Frob( \gamma h \gamma^{-1}) ) = \Ad \left( \varpi_{\gamma} \rtimes \gamma \right) \cdot \psi_1'(\Frob(h))$.
  But $\Frob(\gamma h \gamma^{-1}) = \gamma \Frob(h) \gamma^{-1}$, because the pinned automorphism of $\hat{H}$ defined by $\Gamma_F$
  is defined over $\mathbf{F}_p$, and so the result follows from \eqref{equ}.  
   \qed

   \subsection{The Weyl group element $U$ in the case $p=3$}  \label{specialcase}

 We now suppose that $p=3$.  If $\GG$ is almost simple then, 
 examining the classification (see tables in the next section), $\GG$ must be of exceptional type $\mathrm{EFG}$
 if $\HH$ is to be semisimple.

\begin{prop*}
Suppose that $\GG$ is simply connected almost simple not of type $\mathrm{E}_6$, that $\sigma$ is inner of order $3$, and that the fixed points $\HH$ are semisimple. Let $(\TT,\BB_H,\BB_G)$ be a Borel triple.  Then 
\begin{enumerate}
\item Over $\overline{F}$, the quotient $N_{\GG}(\HH) / \HH$ has precisely two elements.
\item There is a unique element $U \in W_G$ that has a representative that normalizes both $\HH$ and $\BB_H$ and moreover acts by $-1$ on $\Cent(\HH)/\Cent(\GG)$.
\end{enumerate}
\end{prop*}

\begin{proof}
If $\sigma$ is inner and of order 3, then $\Cent(\HH)/\Cent(\GG)$ is of order $3$, because it is  dual to the quotient of the root lattice of $\GG$ by the root lattice of $\HH$, which by \S\ref{DTT1}(3) is of order $p$.

We have already remarked that $\GG$ is of exceptional type.  If it is not of type $\mathrm{E}_6$, then $\Cent(\GG)$ has order prime to $3$.  It follows that if $n \in N_{\GG}(\HH)$ acts trivially on $\Cent(\HH)/\Cent(\GG)$, then $n \in \HH$, i.e. that the index of $\HH$ in $N_{\GG}(\HH)$ has order $\leq 2$.
A case by case check of the exceptional groups reveals that the normalizer contains $\HH$ with index precisely two.  (See \cite[Table 10.3]{LiebeckSeitz}.  Note also that when $\GG$ is a form of $\mathrm{E}_6$ and $\HH$ a form of $\SL_3^{\times 3}/\Delta \mu_3$, $N_{\GG}(\HH)/\HH$ is the symmetric group of order $3$.)

To prove (2), note that $\HH(\overline{F})$ acts transitively on the set of pairs $\BB'_H \supset \TT'_H$, where $\BB'_H \supset \TT'_H$ are a Borel and maximal torus in $\HH \times_F \overline{F}$.  
Thus the nontrivial element of $N_{\GG}(\HH)/\HH$ has a representative  $\tilde{U} \in \GG(\overline{F})$  that normalizes $\HH$, $\BB_H$, and $\TT$ simultaneously.  Any two such must necessarily differ by an element of $\TT$, so $\tilde{U}$ determines a unique element of $W_G$.
\end{proof}

We record, case-by-case, the element $U$ with respect to the Borel triple of \S\ref{DTT1}.  Being of order two in $W_G$, it can (see \cite[Exercise V.3.3]{BourbakiLie}) be written as a product of commuting root reflections for a finite set of orthogonal positive roots $\{r_1, \dots, r_n\}$ of $\GG$. In our cases this can be taken to be the full set of positive $(-1)$-eigenroots of $U$.
 
{\tiny 
\hspace*{-.7cm}  
\begin{center}
\begin{tabular}{|c|c|c|c|c|}
\hline
$\GG$ & $\mathfrak{h}$ & $\begin{array}{c}\text{node} \\ \text{of} \\ \Delta_\GG\end{array}$ & $\{r_1, \dots, r_n\}$ = positive $(-1)$-eigenroots of $U$ & $\Delta_{\HH}^U$\\
\hline
$\mathrm{G}_2$& $\mathfrak{sl}_3$ & 1 &  $\alpha_1 + \alpha_2$ & $\varnothing$ \\
\hline
$\mathrm{F}_4$& $\mathfrak{sl}_3 \times \mathfrak{sl}_3$  & 2 & 
$
\begin{array}{c}
\alpha_1 + \alpha_2 + \alpha_3 + \alpha_4 \\
\alpha_1 + \alpha_2 + 2 \alpha_3
\end{array} 
$  & $\varnothing$ \\
\hline
$\mathrm{E}_7$& $\mathfrak{sl}_3 \times \mathfrak{sl}_6$& 3 & 
$\begin{array}{c}
\alpha_{1} + \alpha_{2} + \alpha_{3} + 2\alpha_{4} + \alpha_{5} \\
\alpha_{1} + \alpha_{3} + \alpha_{4} + \alpha_{5} + \alpha_{6} + \alpha_{7} \\
\alpha_{1} + \alpha_{2} + \alpha_{3} + \alpha_{4} + \alpha_{5} + \alpha_{6} \\
\end{array}$ & $\alpha_5$
 \\
\hline
$\mathrm{E}_7$&$\mathfrak{sl}_3 \times \mathfrak{sl}_6$ & 5 &
$
\begin{array}{c}
\alpha_{1} + \alpha_{2} + \alpha_{3} + \alpha_{4} + \alpha_{5} + \alpha_{6} \\
\alpha_{2} + \alpha_{3} + 2\alpha_{4} + \alpha_{5} + \alpha_{6} \\
\alpha_{1} + 2\alpha_{2} + 2\alpha_{3} + 3\alpha_{4} + 2\alpha_{5} + \alpha_{6} + \alpha_{7}
\end{array}
$ & $\alpha_3$
 \\
\hline
$\mathrm{E}_8$&$\mathfrak{sl}_9$ & 2 &
$ \begin{array}{c}
\alpha_{1} + \alpha_{2}  + \alpha_{3} + 2\alpha_{4} + 2\alpha_{5} +
\alpha_{6} + \alpha_{7} + \alpha_{8}\\
\alpha_{1} + \alpha_{2} + 2\alpha_{3} + 2\alpha_{4} + 2\alpha_{5} +
\alpha_{6} + \alpha_{7}\\
\alpha_{1} + \alpha_{2} + 2\alpha_{3} + 3\alpha_{4} + 2\alpha_{5} +
\alpha_{6}\\
2\alpha_{1} + 2\alpha_{2} + 3\alpha_{3} + 4\alpha_{4} + 3\alpha_{5} +
3\alpha_{6} + 2\alpha_{7} + \alpha_{8}
\end{array}$ & $\varnothing$
\\
\hline 
$\mathrm{E}_8$&  $\mathfrak{sl}_3 \times \mathfrak{e}_6$& 7 & 
$\begin{array}{c}
\alpha_{1} + \alpha_{2} + 2\alpha_{3} + 2\alpha_{4} + \alpha_{5} +
\alpha_{6} + \alpha_{7} + \alpha_{8} \\
\alpha_{1} + \alpha_{2} + \alpha_{3} + 2\alpha_{4} + 2\alpha_{5} +
\alpha_{6} + \alpha_{7} + \alpha_{8}\\
\alpha_{2} + \alpha_{3} + 2\alpha_{4} + 2\alpha_{5} + 2\alpha_{6} +
\alpha_{7} + \alpha_{8}
\end{array}
$ & $\alpha_2,\alpha_4$
\\
\hline
\end{tabular}
\end{center}}

Note that in this table the two subgroups of $\mathrm{E}_7$ are conjugate to each other --- we include both  for completeness.  The $\alpha_i$s are the simple roots of $\TT$ on $\BB_G$ in their Bourbaki numberings \cite[Plates I--IX]{BourbakiLie}.  The outer action of $U$ on $\HH$ is determined by its fixed points  $\Delta_{\HH}^U$ on $\Delta_\HH$, which are recorded in the rightmost column.
Finally,  in all the cases that occur, these
orthogonal roots mentioned above are in fact strongly orthogonal, i.e. $r_i \pm r_j$ are not roots, as can be verified from the table.

\begin{lemma}
\label{lem:windomearle}
Let $\G$ be an almost simple, simply connected group of exceptional type, but not of type $\mathrm{E}_6$.  Let $\HH$ be the semisimple fixed points of an inner automorphism of order $3$, let $(\TT,\BB_H,\BB_G)$ be a Borel triple, and let $U \in W_G$ be the corresponding Weyl group element of \S\ref{specialcase}.  The following hold:
\begin{enumerate}
\item $U = s_1 \cdot \ldots s_k$, where the $s_i$ are commuting reflections associated to a set of strongly orthogonal roots $r_1,\ldots,r_k$ of $\TT$ on $\BB_G$.  The $r_i$ are uniquely determined as the positive roots $r$ such that $Ur=-r$. 
\item If $\BB_G$ is $\gamma$-admissible with respect to $\BB_H$,\footnote{Note this assumption does not apply to the Borel triple used in the table above.} where $\gamma$ is the outer automorphism of $\HH$ determined by $U$, then the roots $r_i$ are simple for the positive system determined by $\BB_G$.  
\item There is at least one $\Gamma$-admissible (in the sense of (3)) Borel $\BB_G$ such that if $\beta$ is a simple root of $\TT$ on $\BB_H$ that is fixed by $U$, then $\langle \rho_G,\beta^\vee\rangle$ is odd, where $\rho_G$ denotes half the sum of the positive roots of $\TT$ on $\BB_G$.  \end{enumerate}

\end{lemma}

\begin{proof}
We already proved (1) (we verified it for a different choice of Borel triple, but the statement is independent of that choice). 

 Let us verify (2).  
As $\BB_G$ is $\gamma$-admissible, we may find $\chi \in X_*(\TT^{\gamma})$ (i.e., orthogonal to each of the $r_i$) that is positive on simple roots of $\BB_H$ and nonnegative on simple roots of $\BB_G$.  Also, $\chi$ is generic, in that its centralizer in $\GG$ coincides with $\TT^{\gamma}$.
Write $r_i$ as a sum of positive simple roots $\sum s_i$; since $\chi(r_i) = 0$ each
of those simple roots must be orthogonal to $\chi$ 
and  then (by genericity)
orthogonal to $\TT^{\gamma}$, i.e $r_i$ is a sum of simple roots each in the $-1$ eigenspace of $U$.  
By our prior discussion, each of these simple roots must actually be one of the $r_i$s, so every $r_i$ is in fact simple.

Let us prove (3).  From the table above, we see that unless $\G,\mathfrak{h}$ is one of $\mathrm{E}_8,\mathfrak{sl}_3 \times \mathfrak{e}_6$ or $\mathrm{E}_7,\mathfrak{sl}_3 \times \mathfrak{sl}_6$, there are no $U$-fixed elements of $\Delta_{\HH}^U$ at all and the statement is vacuous.  Let us treat the remaining cases:

{Case $\GG = \mathrm{E}_8$.}  Let $\BB_G, \BB_H$ be as in \S\ref{DTT1}.  $\BB_G$ is not admissible with respect to $\BB_H$, but in the basis $\alpha_i$ of simple roots of $\TT$ on $\BB_G$, the Borel whose simple roots are
{\tiny \[
\begin{array}{rcl}
\alpha'_1 & = & 2\alpha_{1} + 2\alpha_{2} + 3\alpha_{3} +
4\alpha_{4} + 3\alpha_{5} + 3\alpha_{6} + 2\alpha_{7} + \alpha_{8} \\
\alpha'_2 & = & \alpha_{1} + \alpha_{2} + \alpha_{3} +
2\alpha_{4} + 2\alpha_{5} + \alpha_{6} + \alpha_{7} + \alpha_{8} \\
\alpha'_3 & = & -\alpha_{1} 
-\alpha_{2} -2\alpha_{3} 
-2\alpha_{4} -\alpha_{5} 
-\alpha_{6} -\alpha_{7} 
-\alpha_{8} \\
\alpha'_4 & = & -\alpha_{1} 
-2\alpha_{2}  -2\alpha_{3} 
-4\alpha_{4}  -4\alpha_{5} 
-3\alpha_{6}  -2\alpha_{7} 
-\alpha_{8} \\
\alpha'_5 & = & \alpha_{2} + \alpha_{3} + 2\alpha_{4} +
2\alpha_{5} + 2\alpha_{6} + \alpha_{7} + \alpha_{8} \\
\alpha'_6 & = & \alpha_{1} + \alpha_{2} + 2\alpha_{3} +
2\alpha_{4} + 2\alpha_{5} + \alpha_{6} + \alpha_{7} \\
\alpha'_7 & = & \alpha_{4} \\
\alpha'_8 & = & \alpha_{2}
\end{array}
\]}
is admissible with respect to $\BB_H$.  The $U$-fixed roots of $\Delta_H$ are $\alpha_2$ and $\alpha_4$, and we compute $\langle \rho_G, \alpha_2^{\vee}\rangle =   \langle \rho_G,\alpha_4^{\vee}\rangle = 1$ are both odd.

{Case $\GG = \mathrm{E}_7$}.  Let $\BB_G,\BB_H$ be as in \S\ref{DTT1}.  In the basis $\alpha_i$ of simple roots of $\TT$ on $\BB_G$, the Borel whose simple roots are
{\tiny \[
\begin{array}{rcl}
\alpha_1' & = & \alpha_1+2\alpha_2+2\alpha_3+3\alpha_4+2\alpha_5+\alpha_6+\alpha_7 \\
\alpha_2' & = & \alpha_1+\alpha_2+\alpha_3+\alpha_4+\alpha_5+\alpha_6 \\
\alpha_3' & = & -\alpha_1-\alpha_2-\alpha_3-2\alpha_4-\alpha_5 \\
\alpha_4' & = & -\alpha_1-\alpha_2-2\alpha_3-2\alpha_4-2\alpha_5-2\alpha_6-\alpha_7 \\
\alpha_5' & = & \alpha_1+\alpha_3+\alpha_4+\alpha_5+\alpha_6+\alpha_7\\
\alpha_6' & = & \alpha_2+\alpha_3+2\alpha_4+\alpha_5+\alpha_6 \\
\alpha_7' & = & \alpha_5
\end{array}
\]}
is $\gamma$-admissible.  The only $U$-fixed root of $\Delta_H$ is $\alpha_5$, and  $\langle \rho_G,\alpha_5^\vee \rangle =1$ is odd.
\end{proof}

  \begin{prop} \label{crit6} 
  Suppose that $\sigma$ is an inner ($F$-rational) automorphism
  of order $p=3$ of a simply connected group $\GG$,  and over the algebraic closure
   $(\mathrm{Lie}(\GG),\mathrm{Lie}(\HH))$ is one of 
($\mathfrak{g}_2 \supset \mathfrak{sl}_3$), 
( $\mathfrak{f}_4 \supset \mathfrak{sl}_3 \times \mathfrak{sl}_3$), 
( $\mathfrak{e}_7 \supset \mathfrak{sl}_3 \times \mathfrak{sl}_6$), 
($\mathfrak{e}_8 \supset \mathfrak{sl}_9$) or 
($\mathfrak{e}_8 \supset \mathfrak{sl}_3 \times \mathfrak{e}_6$).
  Then there exists a $\sigma$-dual homomorphism.   
\end{prop}
 
 \begin{proof}

 Note that in all cases listed $\GG$ has no nontrivial outer automorphisms.  By the discussion of \S\ref{specialcase} the index of $\HH$ in its normalizer is $2$.  Then \eqref{moof} shows that the image of $\Gamma_F \to \Out(\HH)$ is of order $\leq 2$, so there is a $\Gamma_F$-admissible Borel class by Lemma \ref{pokemon}(2).  Fix such  a $\Gamma$-admissible Borel class $(\TT,\BB_H,\TT,\BB_G)$.

 As in the first sentence of the proof of \ref{LpsiInner}, 
if we choose any pinning of $\hat{H}_1$ then we can specify $\Lpsi$ in such a way
that it preserves pinnings when restricted to $\hat{H}$.  
We will in fact choose a pinning of $\hat{H}_1$ only in 
 in {\em Step 6} below; all the prior steps
will be independent of the choice of pinning, because they will depend only on $\psi_1'$ restricted
to $\hat{T}_H$. 

In each case the center of $\hat{H}$ is a $p$-group ($p = 3$), so the hypotheses of Proposition \ref{LpsiInner} apply and produce a homomorphism $\Lpsi:\LH \to \LG$ extending the dual norm map associated to the $\Gamma$-admissible class. Explicitly, this is given by
\begin{equation} \label{LpsiInnerDef2} \Lpsi: h \rtimes \gamma  \in \LH  \mapsto \psi_1'(\Frob( h)) \left( \varpi_{\gamma} \rtimes \gamma \right)  \in \LG\end{equation}   where $\Frob$ is the Frobenius. 
  To prove that $\Lpsi$ is a $\sigma$-dual map, we verify (b) from \S\ref{BrauerSatakeCor}.  That is, for each $\gamma \in \Gamma$, we must see that 
\begin{equation}
\label{showthis}
\text{$\inverserho_G(\gamma)$ and $\Lpsi(\inverserho_H(\gamma))$ project to the same 
element of $\hat{G} \rtimes \gamma \git \hat{L}_{\gamma}$,}
\end{equation}
where $\hat{L}_{\gamma}$ is defined as in  \S \ref{normsec}, and
$\inverserho_G, \inverserho_H$ are the canonical pseudoroots for $\HH, \GG$. 

{\it Step 1.}  We claim that $\Gamma_F \to \mathrm{Out}(\HH)$ is the $\mathbf{F}_3^*$-valued cyclotomic character.  That is $\gamma$ acts nontrivially on $\HH$ if and only if $q_{\gamma} = -1$ mod 3, where $q_{\gamma} \in \hat{\mathbf{Z}}^*$ is the image of $\gamma$ under the cyclotomic character.

Set $\mathbf{S} =  \mathbf{T}_{\HH}^{\mathrm{can}}/Z(\GG)$.  Now any $\gamma$ fixes $\sigma \in \mathbf{S}(\overline{F})$,  where $\sigma$
is the order $3$ element such that conjugation by $\sigma$ has $\HH$ for fixed points.
In other words, ``evaluation at $\sigma$,'' considered in $ \Hom(X^*(\mathbf{S})/3, \mu_{3})$  
is Galois-invariant. If $\gamma \in \Gamma_F$
acts nontrivially on $\HH$ its action is the same as that of $U$;
in particular, it acts as $-1$ on  $X^*(\mathbf{S})$,  so it acts as $-1$ on $\mu_3$ too.

{\it Step 2.}  Let us prove \eqref{showthis} in the case that $\gamma$ acts trivially on $\HH$, i.e. $q_{\gamma}=1$.  
The element $\overline{w(\gamma)}$, defined  in \eqref{dalecooper} is trivial here. That means, looking at the way we define $\Lpsi$ from \eqref{LpsiInnerDef},  that $\gamma$ is sent to $\gamma$.
It follows that $\inverserho_G(\gamma)$ and $\Lpsi(\inverserho_H(\gamma))$ coincide already in $\hat{G} \rtimes \gamma$.

{\it Step 3.}  When $\gamma$ acts nontrivially on $\HH$, i.e. $q_{\gamma} = -1$, it acts via the element $U$ of \S\ref{specialcase}.  Let us denote the Levi $\hat{L}_{\gamma}$ of \S\ref{normsec}  by $\hat{L}$.  It is obtained from $\hat{T}$ by adjoining the roots $\pm r_{i,*}^\vee$, where $r^{\vee}_{i,*}$ are the orthogonal roots of Lemma \ref{lem:windomearle}.

{\it Step 4.}  We prove an identity \eqref{1074} between the coweights $\Sigma_G^*$ and $\Sigma_H^*$, that give rise to the canonical pseudoroots.
Let $\Sigma_G$ and $\Sigma_H$ be the sum of all positive roots of $\TT_G$ on $ \BB_G$ and $\TT_H$ on $\BB_H$ respectively.  
Let $\Sigma_G^*$ and $\Sigma_H^*$ be the corresponding cocharacters of $\hat{T}_G$ and $\hat{T}_H$ respectively.  We identify $\hat{T}_H$ and $\hat{T}_G$ with each other via the homomorphism $\psi_1$ of \S\ref{psi1sec}, and:
\begin{equation}
\label{1074}
\Sigma_G^* - \psi_1 \circ \Sigma_H^* = \sum r_{j,*} + \delta
\end{equation}
where $\delta \in X_*(\hat{T}_G)$ is orthogonal to all $r_{j,*}^\vee$, i.e. it is in the center of $\hat{L}$.  To prove \eqref{1074}, it is equivalent to show that $\langle \Sigma_G^* -  \psi_1 \circ \Sigma_H^*,r_{j,*}^\vee\rangle = 2$ for each $j$, which is a consequence of
\begin{enumerate}
\item $\langle \Sigma_G,r_j^\vee\rangle = 2$ as the $r_j$ are simple (Lemma \ref{lem:windomearle}(3)).
\item $\langle \Sigma_H,r_j^\vee\rangle = 0$, where we consider $\Sigma_H$ as a character of $\TT_G=\TT_H$. 
This is true since $U$ fixes $\Sigma_H$ and acts on $r_j^{\vee}$ by $-1$
\end{enumerate}

{\it Step 5.}  Let $\hat{U} \in W_{\hat{G}}$ correspond to $U \in W_G$.  We give an explicit formula for an element $\varpi \in N_{\hat{G}}(\hat{T})$ that projects to $\hat{U}$.  As $U$ leaves stable the coroots of $\HH$, we see that $\hat{U}$ leaves stable the roots of $\hat{H}_1$, and so it will follow in particular that $\varpi$ normalizes $\hat{H}_1$ and induces on it the outer automorphism   corresponding to $\gamma$.

Fix a primitive fourth root of unity $\sqrt{-1} \in k$, i.e. a square root of $q_{\gamma}$.  By \eqref{1074}, we have
\[
\Sigma_G^*\left(\sqrt{-1}\right) = \psi_1' \circ \Sigma^*_H\left(\sqrt{-1}\right) \cdot \left(\prod r_{j,*}\left(\sqrt{-1}\right)\right)
 \delta\left(\sqrt{-1}\right)\]
Let $\iota_j:\SL_2 \to \hat{L}$ be the coroot homomorphisms corresponding to the $r_{j,*}$, and set %
\begin{equation} \label{JDef}
\varpi  = \delta(\sqrt{-1}) \prod \iota_j(J) \qquad \text{ where }J := \left(
\begin{array}{rr} 0 & 1 \\ -1 & 0 \end{array}
\right) \in \SL_2(k)
\end{equation} 

Both $\varpi$ and $\Sigma_G^*(\sqrt{-1})/\psi_1' \circ \Sigma_H^*(\sqrt{-1})$ are elements of $\hat{L}$, and by \eqref{1074} they are $\hat{L}$-conjugate.  In particular, since both $\Sigma_H^*$ and $\Sigma_G^*$ are actually divisible by $2$ inside the cocharacter group 
(see Proposition \ref{canpseudo}), $\varpi$ has order $2$.  Since $\delta(\sqrt{-1}) \in \hat{T}_G$, the conjugation action of $\varpi$ on $X^*(\hat{T})$ is the product of the conjugation actions of $\iota_j(J)$, each of which induces the root reflection through $r_{j,*}^\vee$.  In other words, $\varpi$ acts as $U$ on $\hat{T}$.

{\it Step 6.}
We show that there exists a choice of pinning of $\hat{H}_1$ such that 
 $\varpi$ acts as a pinned automorphism on $\hat{H}_1$,  i.e. for each simple root $\beta$ of $\hat{T}$ on $\psi_1'(\hat{B}_H)$, there is a nonzero vector $X_j \in \mathrm{Lie}(\hat{H}_1)_{\beta}$, and the $X_j$ are permuted by the conjugation action of $\varpi$.
As $\varpi$ has prime order two, this is possible if and only if $\varpi$ acts trivially on $\mathrm{Lie}(\hat{G})_{\beta}$
whenever $\beta$ is actually fixed by $\varpi$.
Such a $\beta$ corresponds
(under $\hat{\psi}_1$) to a simple root $\alpha$ of $\hat{H}$ that is $\gamma$-fixed.   Consulting the table of \S\ref{specialcase}, for $\mathrm{G}_2$, $\mathrm{F}_4$, and $(\mathrm{E}_8,\mathfrak{sl}_9)$, there is nothing to prove since there are no $U$-fixed simple roots.  The remaining two cases are simply-laced, in which case the $\beta$-root group commutes with each $\iota_j(\SL_2)$, because $\beta$ is orthogonal (therefore strongly orthogonal in these simply laced cases) to  $r_j^{\vee,*}$ inside the root system of $\hat{G}$:
   $$ \langle \beta, r_j^{*} \rangle = \langle \beta^*, r_j \rangle  =0$$
   where $\beta^*$ is the associated coroot for $\hat{G}$ --- the last     because $U$ fixes $\beta^*$ and negates $r_j$. 
   
     Finally, we must check that the $\beta$-root group commutes with $\delta(\sqrt{-1})$, or equivalently that $\langle \delta,\beta \rangle$ is divisible by $4$.  From \eqref{1074}, and the fact that $\langle \Sigma_H^*,\beta \rangle = 2$ and $\langle r_j^*,\beta \rangle = 0$, this is equivalent to showing that $\langle \frac{1}{2} \Sigma_G^*,\beta \rangle$ is odd, which is part of Lemma \ref{lem:windomearle}(4).

 {\it Step 7.}  Finally we verify \eqref{showthis} when $q_{\gamma} = -1$.  
 Note that, e.g., $\inverserho(\gamma) = \Sigma_G^*(\sqrt{-1})$ in this case. 
 Referring to
 \eqref{LpsiInnerDef2}  we must show that 
  \[ \psi_1' \Sigma_H^*(\Frob(\sqrt{-1}))  \varpi \rtimes \gamma = \psi_1' \circ \Sigma_H^*(\sqrt{-1}) \varpi \rtimes \gamma \qquad \text{and} \qquad
 \Sigma_G^*(\sqrt{-1}) \rtimes \gamma\]
 of $\LG$
 are conjugate under $\hat{L}_{\gamma}.$  Because $\G$ is inner, it is sufficient to show that $\Sigma_G^*(\sqrt{-1})$ and $\psi_1' \Sigma_H^*(\sqrt{-1})\varpi$ are $\hat{L}$-conjugate.  This follows from the remarks of {\it Step 5.} (note that, as $\langle \Sigma_H,r_j^\vee\rangle = 0$, the map $\psi_1' \circ \Sigma_H^*$ takes values in the center of $\hat{L}_{\gamma}$).
  \end{proof}

\section{Construction of $\sigma$-dual homomorphisms 3: exceptional isogenies} \label{casebash3}

In this section we treat two large ``clusters'' of cases.  One of these involves heavily the ``exceptional isogenies'' of semisimple groups that exist only in special characteristic. The remaining cases are checked individually, 
and they are discussed in the final section.

\subsection{Remark on twisting} \label{Twisting}

We will use repeatedly the following remark: 
Suppose that $(\GG, \sigma, \HH= \mathrm{fix}(\sigma))$ admits a form $(\GG', \sigma', \HH'=\mathrm{fix}(\sigma'))$
defined over $F$ with $\GG'$ and $\HH'$  both $F$-split groups.  Then $\GG, \sigma$
is obtained by twisting  from $(\GG', \sigma')$ by an element of $$H^1(\Gamma, \Aut(\GG', \sigma', \HH')).$$
This group of automorphisms is precisely the centralizer $\Cent_{\Aut(\GG')}(\sigma')$ of $\sigma'$ inside $\Aut(\GG')$.

In particular, the image of $\Gamma \rightarrow \Out(\HH) $
lies inside the image of $\Cent_{\Aut(\GG')}(\sigma')$ in $\Out(\HH') \simeq \Out(\HH)$ and 
similarly for the image of $\Gamma \rightarrow \Out(\GG)$. 
For example, if $\Cent_{\Aut(\GG')}(\sigma')$ is {\em connected}, 
then both $\HH$ and $\GG$ are necessarily inner.   
 
\begin{prop} \label{crit3}
Suppose the following conditions hold:  \begin{itemize}
\item[(i)] $\Out(\GG)$ is trivial (so that $\GG$ and $\sigma$ are both inner) and $p=2$; 
\item[(ii)] The image of $\Gamma$ in $\Out(\HH)$
is of order $\leq 2$;
thus we may choose (Lemma \ref{LpsiInner}) a $\Gamma$-admissible Borel class. 

\item[(iii)] $\pi_1 \HH$ is a $2$-group.

\end{itemize}
 then $\Lpsi$ as defined by Proposition \ref{LpsiInner} is a $\sigma$-dual homomorphism. These conditions hold when the Lie algebras of $\G \supset \HH$ are isomorphic over $\overline{F}$ to a pair from the following list:
 ($\mathfrak{so}_{2n+1} \supset \mathfrak{so}_{2u+1} \times \mathfrak{so}_{2v}$),  or 
 ($\mathfrak{sp}_{2n} \supset \mathfrak{sp}_{2u} \times \mathfrak{sp}_{2v}$)
for $n \geq 2, u,v \geq 1$, and $u+v= n$; or 
  $(\mathfrak{g}_2 \supset \mathfrak{so}_4)$,  $(\mathfrak{f}_4 \supset \mathfrak{sp}_6 \times \mathfrak{sl}_2 \mbox{ or } \mathfrak{so}_9)$,
  $(\mathfrak{e}_7 \supset \mathfrak{sl}_2 \times \mathfrak{so}_{12} \mbox{ or } \mathfrak{sl}_8)$, $(\mathfrak{e}_8 \supset \mathfrak{so}_{16} \mbox{ or } \mathfrak{sl}_2 \times \mathfrak{e}_7)$.

 \end{prop}

\begin{proof}
By inspection all cases have properties (i), (ii); and for (iii)  $\Out(\HH)$ itself has at
 most two elements,  except in the first-listed case with $v = 4$.  But then the image of $\Cent_{\mathrm{Aut}(\Spin_{2n+1})}(\sigma) = \Cent_{\SO_{2n+1}}(\sigma)$ in $\Out(\HH)$ has order $\leq 2$ (the centralizer has two connected components), and that $\Gamma \to \Out(\HH)$ has image of size $\leq 2$ follows from \S\ref{Twisting}.

  In all cases, an application of Lemma \ref{LpsiInner}  now shows that 
  the dual norm map extends to $\LH \rightarrow \LG$.
It remains to check \S\ref{BrauerSatakeCor}(b).  Let us show that the stronger condition \S\ref{BrauerSatakeCor}($\beta$) holds.  The pseudoroots are trivial, 
since $p=2$; that is to say $\rho_H(\gamma) = \gamma \in \LH$ and $\rho_G(\gamma) = \gamma \in \LG$.  The condition to be checked is that $\gamma$ and $\Lpsi(\gamma)$ both centralize $Z(\hat{L}_{\gamma})$, and that 
\begin{equation}
\label{eq:gamma2n}
\gamma^{2^n} = \Lpsi(\gamma)^{2^n} \quad \text{ for $n \gg 0$}
\end{equation}
Recall that on the left $\gamma$ denotes an element of $\LG$ and on the right it denotes an element of $\LH$.

 The hypotheses of Lemma \ref{ZNL}   apply in all cases  except $\mathfrak{g}_2 \supset \mathfrak{so}_4$: this is clear except in the cases  when $\mathfrak{h}$ has 
a repeated factor
i.e.  ($\mathfrak{so}_{2n+1} \supset \mathfrak{so}_{2u+1} \times \mathfrak{so}_{2v}$)
with $v=2$, or   ($\mathfrak{sp}_{2n} \supset \mathfrak{sp}_{2u} \times \mathfrak{sp}_{2v}$)
with $u=v$.  In each of these cases, the criterion remarked after the proof of Lemma \ref{ZNL} applies. 
For the case $\mathfrak{g}_2 \supset \mathfrak{so}_4$, the normalizer of $\HH$ in $\GG$ is trivial, so by \eqref{dalecooper} the conclusion of Lemma \ref{ZNL} is trivially true. 

As $\G$ is inner, $\LG = \hat{G} \times \Gamma_F$ and $\gamma \in \LG$ automatically centralizes $Z(\hat{L}_\gamma)$.  Meanwhile $\gamma \in \LH$ centralizes $\hat{T}^{\gamma}_H$ by definition, so that $\Lpsi(\gamma)$ centralizes $\hat{N}(\hat{T}_H^{\gamma})$, the image of $\hat{T}_H^{\gamma}$ under the dual norm.  
Now conjugation by $\Lpsi(\gamma)$ gives an inner automorphism of $\hat{G}$,
and according to Lemma \ref{ZNL} that inner automorphism is represented by an element of 
 $\hat{L}_{\gamma}$;
 so  $\Lpsi(\gamma)$ centralizes $Z(\hat{L}_{\gamma})$ as well.

Since the image of $\Gamma_F$ in $\mathrm{Out}(\HH)$ and $\mathrm{Out}(\GG)$ both have order at most $2$, both $\gamma^2$ and $\Lpsi(\gamma)^2$ must centralize $\hat{H}_1 = \Lpsi(\hat{H})$, i.e. they must differ by an element of $\hat{G}$ that centralizes  the image of $\hat{H}_1$.   Now,  
by our explicit construction (\S \ref{sec9}), $\hat{H}_1 \subset \hat{G}$ is a subgroup of maximal rank. 
Thus, any element that centralizes the image of $\hat{H}_1$ lies inside $\hat{H}_1$; but then it lies in the center of $\hat{H}_1(k) = \{e\}$ (because  $\pi_1(\HH)$ is a 2-group). 
\end{proof}

\subsection{When $\sigma$ is pinned}
We now study the  case where $(\GG, \sigma)$ is isomorphic, over the algebraic closure,
to a  pinned automorphism of an almost simple simply connected group $\GG$, leaving stable a maximal torus and Borel $\TT \subset \BB \subset \GG \times_F \overline{F}$.

  Put $\TT_H := \TT^{\sigma}$ and $\BB_H := \BB^{\sigma}$ --- as $\GG$ is simply-connected, they are a maximal torus and Borel subgroup of $\HH$.  As $\GG$ is almost simple, $\mathrm{Out}(\HH)$ is trivial and $\BB$ is $\Gamma_F$-admissible (i.e. $1$-admissible) with respect to $\TT_H, \BB_H$.

For each root $\alpha$ of $\GG$,  let $\bar{\alpha}$ be its restriction 
to $\left( \TT^{\sigma} \right)^{\circ}$. 
   In general, the roots of $\HH$ are in one-to-one correspondence with the ``$\sigma$-equivalence classes'' of roots of $\TT$ on $\GG$ --- here $\alpha$ and $\beta$ are $\sigma$-equivalent if $\bar{\alpha} = r \bar{\beta}$ for some $r > 0$.     In particular, the number of roots of $\HH$ is the number of distinct rays in $\R \otimes X^*(\TT^{\sigma})$ of the form $\R_{>0} \bar{\alpha}$.  For further discussion of this and what follows,  we refer to \cite[\S\S 3.3--3.4]{Reeder} and \cite[p. 177]{Steinberg}.
    
Any $\sigma$-equivalence class $a$ is the positive roots of a system of type either $A_1^r$ or $A_2$ 
(\cite[p. 177]{Steinberg}). We say accordingly that $a$ (or any $\alpha \in a$) is `` of type $A_1^r$''
or ``of type $A_2$.'' 
Type $A_2$ occurs only in type $\SL_{2n+1}$, namely 
$a = \{e_i - e_{n+1}, e_{n+1}- e_{i^*}, e_i - e_{i^*} \}$ where $i+i^*=2n+2$ and we use the standard
realization of the root system $A_{2n}$.  Moreover, if a root is of type $A_1^r$, then $r$
is the size of the $\sigma$-orbit of $\alpha$.  
   
    If $a$ is a $\sigma$-equivalence class of roots, the corresponding root of $\TT^{\sigma}$ on $\HH$ is 
   $e \bar{\alpha}$ for a unique $e  \in \Q_{>0}$; here $\alpha$ is any element of $a$. We set $e_{\alpha}=e$ for any $\alpha\in a$. 
   Also, we write $\theta_{\alpha}$ for the angle between $\alpha$ and $\bar{\alpha}$ with respect to a $W$-invariant inner product
   on $X^*(\TT) \otimes \R$.   Finally, we let $o_{\alpha}$
   to be the order of the $\sigma$-orbit of $\alpha$, so that $o_{\alpha}=1$ or $p$
   according to whether $\alpha$ is $\sigma$-fixed or not.

The pinned automorphism $\sigma$ also acts as a pinned automorphism on $\hat{G}$.  
Let $\hat{G}^{\sigma}$ be the fixed locus with its  reduced subscheme structure. 
If $\alpha_*^\vee$ is a root of $\hat{T}$ on $\hat{G}$, we let $\overline{\alpha^\vee_*}$ denote its restriction to $\hat{T}^{\sigma} = (\hat{T}^{\sigma})^\circ$.   In general a positive multiple
$f_{\alpha}\overline{\alpha^\vee_*}$ of $\overline{\alpha^\vee_*}$ is a root of $\hat{T}^{\sigma}$ on $\hat{G}^\sigma$.

In particular $\hat{G}^{\sigma}$ and $\HH$ are both semisimple, and have the same rank and the same number of roots --- however, they are not always dual groups, but they are isogenous:

 \begin{lemma}\label{pinnedcase} 
 Suppose that for all roots $\alpha$ of $\TT$ on $\GG$,
 \begin{equation}
 \label{eq:zaz0}
  p  e_{\alpha} f_{\alpha} \cos^2(\theta_{\alpha})    \in \{1,p,p^2, \dots \}
  \end{equation} Then the dual norm defined by $(\BB_H, \BB_G)$ extends to an  
  isogeny 
 $ \hat{H} \rightarrow \hat{G}^{\sigma}.$  
 \end{lemma}
 
 For example, for pinned triality acting on $\Spin(8)$, the dual norm 
composes the inclusion $\mathrm{G}_2 \rightarrow \mathrm{PSO}_8$
with an exceptional isogeny $\mathrm{G}_2 \rightarrow \mathrm{G}_2$. 
 
 By explicit computation in the type $\mathrm{A}_2$ case, 
\begin{equation} \label{Comps} (e_{\alpha}, f_{\alpha}, \cos^2(\theta_{\alpha}) ) =   \begin{cases} (\frac{1}{2} o_{\alpha}, o_{\alpha}, o_{\alpha}^{-2}), \mbox{ type $\mathrm{A}_2$}, \\
(1, 1, o_{\alpha}^{-1}) \mbox{ else}. \end{cases} \end{equation}
and so the assumption \eqref{eq:zaz0} always applies.
Indeed, when $\GG$ is not a form of $\SL_{2n+1}$, this isogeny is a ``special isogeny'',  switching short and long roots from $\hat{H}$ to $\hat{G}^{\sigma}$.

\begin{proof}
 The norm is a $\sigma$-invariant map $\TT_G$ to $\TT_H$, and the dual norm $\hat{N}$ 
 takes $\hat{T}_H$ to $\hat{T}_G^{\sigma}$, a maximal torus of $\hat{G}^{\sigma}$.  We will apply Chevalley's isogeny theorem, which says that $\hat{N}$ extends to an isogeny so long as for each root $\alpha$ of $\hat{G}^\sigma$, there is a root $\alpha'$ of $\hat{H}$ and a $q_{\alpha} \in \{1,p,p^2,\ldots\}$ such that $\hat{N}$ induces
\begin{equation}
\label{eq:davidbrooks}
\text{(A)   }\quad \alpha \mapsto q_{\alpha} \alpha', \quad \text{ and   } \quad \text{(B) }\quad (\alpha')^\vee \mapsto q_{\alpha} \alpha^\vee
\end{equation}
Before proceeding we make a reduction.  Suppose that the condition (A) on the left holds, and that the quantity $q_{\alpha}$ depends only on whether $\alpha$ is long or short.  Then the condition (B) on the right also holds.  Indeed under these assumptions the tensor $\sum_{\alpha} \alpha \otimes \alpha$ pulls back under $\hat{N}$ to $q_1\sum_{\alpha' \text{ short}} \alpha' \otimes \alpha' + q_2 \sum_{\alpha'\text{ long}} \alpha' \otimes \alpha'$ --- these are both Weyl invariant and (as as both $\hat{H}$ and $\hat{G}^\sigma$ are semisimple) positive definite inner products.  Using them to identify weights with coweights, we have $(\alpha')^\vee = 2\alpha'/\langle \alpha',\alpha'\rangle \mapsto 2 q_{\alpha} \alpha/\langle \alpha,\alpha\rangle = q_{\alpha} \alpha^\vee$, as desired. 

Now let us prove that \eqref{eq:davidbrooks}(A) holds.
 Let $\mathfrak{t}_G = X_*(\TT_G) \otimes \R$.
 If we identify $\mathfrak{t}_G$ with the the Lie algebra
 of a maximal torus inside the real points of the split form of $\GG$,
the Killing form endows it with an inner product, and $\sigma$ acts on $\mathfrak{t}_G$ and preserves this inner product.

  Make the following identifications:   \begin{itemize}
 \item[(a)] $X^*(\TT_G) \otimes \R = \mathfrak{t}_G^* \simeq \mathfrak{t}_G$ via Killing form;
for each root $\alpha$ let $H_{\alpha} \in \mathfrak{t}_G$ be the corresponding element
 representing the root.
 \item[(b)] $X_*(\hat{T}_G) \otimes \R = X^*(\TT_G) \otimes \R \simeq \mathfrak{t}_G$  via (a). 
 \item[(c)] $X_*(\hat{T}_G^{\sigma}) \otimes \R  \simeq \mathfrak{t}_G^{\sigma}$ via (b). 
 \item[(d)] $X^*(\hat{T}_G^{\sigma}) \otimes \R \simeq (X_*(\hat{T}_G^{\sigma}) \otimes \R)^* \simeq \mathfrak{t}_G^{\sigma}$ via (c) and the Killing form. 
 \item[(e)] $X_*(\TT_H) \otimes \R \simeq \mathfrak{t}_G^{\sigma}$.
 \item[(f)] $X^*(\TT_H) \otimes \R = (\mathfrak{t}_G^{\sigma})^* \simeq \mathfrak{t}_G^{\sigma}$ via Killing form again.
 \item[(g)] $X_*(\hat{T}_H) \otimes \R \simeq X^*(\TT_H) \otimes \R \simeq \mathfrak{t}_G^{\sigma}$ via (f). 
 \end{itemize}
 
 With these identifications the dual norm $\hat{T}_H \rightarrow \hat{T}_G^{\sigma}$
 corresponds to ``multiplication by $p$'' on $\mathfrak{t}_G^{\sigma}$.
  The associated root $\alpha^{\vee}_*$ for $\hat{G}$ corresponds to $2 H_{\alpha}/ \langle H_{\alpha}, H_{\alpha} \rangle \in \mathfrak{t}$.
  Then $  \frac{2 f_{\alpha}}{p} \frac{\sum H_{\sigma^i \alpha}}{ \langle H_{\alpha}, H_{\alpha} \rangle} \in \mathfrak{t}$
 represents a root of $(\hat{G})^{\sigma}$. Pulling back under $\hat{T}_H \rightarrow (\hat{T}_G)^{\sigma}$, 
 we get the weight of $\hat{T}_H$ represented by 
\begin{equation} \label{zaz1}  2 f_{\alpha}  \frac{\sum H_{\sigma^i \alpha}}{ \langle H_{\alpha}, H_{\alpha} \rangle}, \end{equation}
  On the other hand, the root $\alpha$ gives rise to a root of $\HH$, namely $e_{\alpha} \bar{\alpha}$, which is represented by 
 $\frac{e_{\alpha}}{p} \sum H_{\sigma^i \alpha} \in \mathfrak{t}_G^{\sigma}$. The associated coroot  for $\HH$ is given by
\begin{equation} \label{zaz2} \frac{2}{p e_{\alpha}  (\cos^2 \theta_{\alpha})} \sum \frac{ H_{\sigma^i \alpha}}{\langle H_{\alpha}, H_{\alpha} \rangle}\end{equation} 
By our assumption \eqref{eq:zaz0}, the weight \eqref{zaz1} is a $p$-power multiple of \eqref{zaz2}.

In fact, examining \eqref{Comps}, the multiple is given by $1/2$ in the type $A_2$ case, and $o_{\alpha}^{-1}$ otherwise. 
From this we see that the multiple depends only on whether the induced root of $\hat{G}^{\sigma}$ is long or short;
in particular, it is constant on Weyl orbits.

\end{proof}

\begin{prop} \label{crit4} 
Assume   that $\HH$ is inner and that $\sigma$ is conjugate over $\bar{F}$ to a pinned automorphism of $\GG$. 
Then there is a $\sigma$-dual homomorphism.
 These conditions hold for each of the following pairs $(\GG, \HH)$:
 $(\SL(2n+1) ,\SO(2n+1))$ for $n \geq 1$, $(\SL(2n),\Sp(2n))$ for $n \geq 2$,  $(\Spin(2n+2) ,\Spin(2n+1))$ for $n \geq 1$,  $(\Spin(8) ,\mathrm{G}_2)$
and  $(\mathrm{E}_6,\mathrm{F}_4)$. 
\end{prop}

\begin{proof}

In each case $\HH$ is inner and (by \S  \ref{Twisting})
 the image of $\Gamma$ in $\Out(\GG)$  is contained in the subgroup generated by $\sigma$. 
In this setting  the Borel class $(\BB_H, \BB_G)$ is $\Gamma$-admissible.

The dual norm extends by Lemma \ref{pinnedcase} to $\psi_0: \hat{H} \rightarrow \hat{G}^{\sigma}$  and we take
\begin{equation} \label{Lpsisimple} \Lpsi : \hat{h} \rtimes  \gamma  \in \LH \longrightarrow \psi_0(\hat{h})   \rtimes \gamma \in \LG. \end{equation} 
 
 This satisfies the conditions of  \S \ref{BrauerSatakeCor}. In fact, in all cases 
  $\Lpsi(\rho_H^-(\gamma)) = \rho_G^-(\gamma)$. That is clear in the cases where $p=2$.
For  $(\Spin_8, \mathrm{G}_2)$ with $p=3$: The cyclotomic character is valued in $\mathbf{F}_3^*$, and we must verify that
  (see \eqref{writeitout})  
  $\hat{N} \circ (\Sigma_H^*/2) $ and $\Sigma_G^*/2$ are equal when evaluated on $\mathbf{F}_3^*$;
it is enough to check that the difference $\Sigma_G^* - \hat{N} \Sigma_H^*$ is divisible by $4$,  i.e. 
that $\Sigma_G$ and $\Sigma_H \circ N$ differ by a multiple of $4$, which can be checked by hand.
\end{proof}

\section{Construction of $\sigma$-dual homomorphisms 4: computations in remaining cases}
 \label{casebash4} 
In discussing the remaining examples we observe the following: Suppose that $\psi_1: \hat{H} \rightarrow \hat{G}$ has trivial centralizer in $\hat{G}(k)$. Then it extends in at most one way
to a morphism $\LH \rightarrow \LG$ over $\Gamma_F$.  This applies in all cases of the following:

\begin{theorem} \label{casebash4thm}
A $\sigma$-dual homomorphism exists  when $(\GG, \HH)$ or $(\mathfrak{g}, \mathfrak{h})$ is a
form of any of the following:  (when we specify the Lie algebra, we suppose that $\GG$ is simply connected).
  \begin{itemize}
\item[(i)] $(\mathfrak{so}_8, \mathfrak{sl}_3)$,  where $\psi_1: \mathrm{SL}_3 \rightarrow \mathrm{PSO}_8$
which is the composition of the Frobenius and the adjoint representation. 
\item[(ii)]  $(\SL_{2n}, \SO_{2n})$ or  $(\SO(2n), \SO(2a) \times \SO(2b))$ with $n=a+b$;
 where $\psi_1$ 
  is the composition of the Frobenius and the standard inclusion.
 \item[(iii)] 
$(\SO(2n+2), \SO(2a+1) \times  \SO(2b+1) )$, where $\psi_1:
 \mathrm{Sp}_{2a} \times \mathrm{Sp}_{2b} \hookrightarrow \mathrm{Sp}_{2a+2b}  \stackrel{\iota}{\hookrightarrow} \mathrm{SO}_{2n+2} $,
and $\iota$ is described below \eqref{embeddesc}.
\item[(iv)]   When $(\mathfrak{g}, \mathfrak{h})$ is a form of $(\mathfrak{e}_8,  \mathfrak{sl}_5^2)$.
Here $\Lpsi$ is described in detail below. 

\end{itemize}

\end{theorem}

The proof of all cases is similar, but lengthy. We give only details of the final case $(\mathfrak{e}_8, \mathfrak{sl}_5^2)$, after
describing the morphism $\iota$ from case (iii).  Note that in some of the other cases it is necessary to use
directly the Theorem of \S \ref{BrauerSatake}, not only its corollary from \S  \ref{BrauerSatakeCor}.

\subsection{The map $\Sp(2n) \rightarrow \SO(2n+2)^{\mathrm{pinned}}$ in characteristic $2$.} 
We describe the map $\iota$ cited above. 
 Write $Q$ for a  quadratic form on the vector space $V=k^{2n}$, with
associated nondegenerate bilinear form $B=\langle -, - \rangle$.  Let $q$ be the form on the two
dimensional vector space $\langle f_1, f_2 \rangle$ with $q(z_1 f_1+z_2 f_2)=z_1z_2$.
We describe an embedding \begin{equation} \label{embeddesc} \Sp(B) \rightarrow \SO(Q \bigoplus  q)^{\tau} \end{equation} 
where  $\tau$ is the identity on $V$ swaps $f_1, f_2$. 
For any  $A \in \Sp(B)$ there exists a unique linear functional $\ell =\ell_A \in V^*$
such that $Q (Ax)-Q(x) = \ell^2$.   Write $\ell_A(v) =  B(Av, w_A)$ for some $w_A \in V$. 
Note that, if we use $B$ to give an isomorphism $V \rightarrow V^*$ and thus
regard $Q$ as a quadratic form also on $V^*$, we have $Q(\ell_A) =Q(w_A)$.
After all, $\ell_A(A^{-1} w_A) = B(w_A, w_A) = 0$, 
so that $Q(w_A) = Q(A^{-1} w_A)  = Q(\ell_A)$. 

 Then there is a unique regular function \cite{Dieudonne} $D$ on $\Sp(B)$,  
extending the $\{0,1\}$-valued Dickson invariant on $\mathrm{O}(Q)$, with the property that 
$$D(A)^2 - D(A) = Q(w_A)^2=Q(\ell_A)^2$$ 
 i.e. the Artin--Schreier covering of $\Sp(B)$ defined by this last equation has a canonical splitting.
With this in hand,  the map \eqref{embeddesc}  is $A \mapsto \mathrm{Frob}(\tilde{A})$  where $\tilde{A}$ is given by
$$ \label{weirdembed}\tilde{A}v= Av + \ell(v) (f_1+f_2),  \\ \tilde{A} f_i = s f_i + (1-s) f_j + w, \{i,j\}=\{1,2\}.
$$
and $s=\sqrt{D(A)}, \ell=\ell_A, w =w_A$ are as above. Note that the map 
 $A \mapsto \tilde{A}$ is not regular, because it involves square roots, but $\Frob(\tilde{A})$ does not.

  \begin{prop}
 Suppose that $(\mathrm{Lie}(\G), \mathrm{Lie}(\HH))$ is a form of $(\mathfrak{e}_8, \mathfrak{sl}_5^2)$. Then there is a $\sigma$-dual homomorphism. 
 \end{prop}

Let $\GG$ be a form of $\mathrm{E}_8$ over $F$, let $\sigma \in \GG = \mathrm{Aut}(\GG)$ be an element of order $5$ whose centralizer (denoted $\HH$) is semisimple.  We may find a maximal torus $\TT \subset \HH \times_F \overline{F} \subset \GG \times_F \overline{F}$ 
over the algebraic closure of $F$, such that $X^*(\TT)$ is naturally identified with the $\mathrm{E}_8$ lattice in $\R^8$.  The simple roots of this may be taken to be the columns of the matrix
{ \begin{equation}
\label{eq:E8simple}
{\tiny 
\left(\begin{array}{rrrrrrrr}
-{1}/{2} & -1 & 1 & 0 & 0 & 0 & 0 & 0 \\
{1}/{2} & -1 & -1 & 1 & 0 & 0 & 0 & 0 \\
{1}/{2} & 0 & 0 & -1 & 1 & 0 & 0 & 0 \\
{1}/{2} & 0 & 0 & 0 & -1 & 1 & 0 & 0 \\
{1}/{2} & 0 & 0 & 0 & 0 & -1 & 1 & 0 \\
{1}/{2} & 0 & 0 & 0 & 0 & 0 & -1 & 1 \\
{1}/{2} & 0 & 0 & 0 & 0 & 0 & 0 & -1 \\
-{1}/{2} & 0 & 0 & 0 & 0 & 0 & 0 & 0
\end{array}\right)}
\end{equation}}
Let $\alpha_i$ denote the $i$th column of this matrix --- this numbering agrees with the Bourbaki numbering.  Let $\alpha_0$ denote the highest root:
\begin{equation}
\label{eq:E8alpha0}
\alpha_0 = 2 \alpha_1 + 3 \alpha_2 + 4 \alpha_3 + 6 \alpha_4 + 5 \alpha_5 + 4\alpha_6 + 3\alpha_7 + 2\alpha_8
\end{equation}
Then $\Delta_H := \{-\alpha_0,\alpha_8,\alpha_7,\alpha_6,\alpha_1,\alpha_3,\alpha_4,\alpha_2\}$ may be taken to be simple roots of $\TT$ on $\HH \times_F \overline{F}$. A visual aid:
{\tiny 
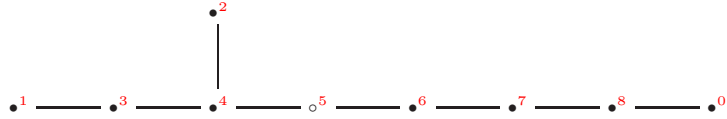
\[
\xymatrix{
& & \bullet^{\color{red} 2} \ar@{-}[d] \\
\bullet^{\color{red} 1} \ar@{-}[r] & \bullet^{\color{red} 3} \ar@{-}[r] & \bullet^{\color{red} 4} \ar@{-}[r] & \circ^{\color{red} 5} \ar@{-}[r] & \bullet^{\color{red} 6} \ar@{-}[r] & \bullet^{\color{red} 7} \ar@{-}[r] & \bullet^{\color{red} 8} \ar@{-}[r]& \bullet^{\color{red} 0}
}
\]}
The red numbers denote the indices of the $\alpha$s, but note the node labeled $0$ corresponds to \emph{minus} $\alpha_0$.  Let $\BB_H \subset \HH \times_F \overline{F}$ denote the Borel subgroup containing $\TT$ whose simple roots are the black nodes of the diagram above.

Let us denote by $U$ the following matrix:
{\begin{equation}
\label{eq:E8U}
{\tiny
U = 
\left(\begin{array}{rrrrrrrr}
0 & 0 & 0 & 0 & {1}/{2} & -{1}/{2} &
-{1}/{2} & -{1}/{2} \\
0 & 0 & 0 & 0 & -{1}/{2} & {1}/{2} &
-{1}/{2} & -{1}/{2} \\
0 & 0 & 0 & 0 & -{1}/{2} & -{1}/{2} &
{1}/{2} & -{1}/{2} \\
-{1}/{2} & -{1}/{2} & -{1}/{2} & {1}/{2}
& 0 & 0 & 0 & 0 \\
{1}/{2} & {1}/{2} & -{1}/{2} & {1}/{2} &
0 & 0 & 0 & 0 \\
{1}/{2} & -{1}/{2} & {1}/{2} & {1}/{2} &
0 & 0 & 0 & 0 \\
-{1}/{2} & {1}/{2} & {1}/{2} & {1}/{2} &
0 & 0 & 0 & 0 \\
0 & 0 & 0 & 0 & {1}/{2} & {1}/{2} &
{1}/{2} & -{1}/{2}
\end{array}\right)}
\end{equation}}
One checks that $U$ belongs to the Weyl group of $\mathrm{E}_8$ --- that is, it is orthogonal and preserves the $\mathrm{E}_8$-lattice.  It moreover has order 4 and preserves $\Delta_H$, in particular we have $\{1,U,U^2,U^3\} \subset \mathrm{Out}(\HH)$. Finally, $\{1, U, U^2, U^3\}$ is precisely the normalizer of $\HH$ in the Weyl group. 

\begin{prop*}
The map $\Gamma_F \to \mathrm{Out}(\HH)$ induced by the $F$-rational structure on $\HH$ factors through $\{1,U,U^2,U^3\}$.  In fact  $\gamma$ induces $U^i$ exactly when $q_{\gamma}=2^i$,
where $q_{\gamma} \in \mathbf{F}_5^*$ is the cyclotomic character.  \end{prop*}

\begin{proof} 
The first assertion --- that the map factors through $\langle U \rangle$ --- follows as before, using \eqref{dalecooper}. 
Next, proceed as in Step 1 of Proposition \ref{crit6};  with notation as in there, 
 $ \mbox{ evaluation at $\sigma$} \in \Hom(X^*(\mathbf{S}), \mu_F)$.

 Suppose, for example, that $\gamma$ induces $U$. 
 Now $\alpha_5$ defines an element of $X^*(\mathbf{S}) = X^*(\BB_H)$. 
The image of $\sigma$ under $\alpha_5$ is a primitive 5th root of unity $\zeta_5$.
Then
$(\gamma \alpha_5) (\sigma) = U   \alpha (\sigma) = \zeta_5^2$  because
the coefficient of $\alpha_5$ in $U \alpha_5$ is $2$, and  and $\alpha_i(\sigma) = 1$ for $i \neq 5$. 
Similarly $q_{\gamma} = 2^j$ when $\gamma$ induces $U^j$. 
\end{proof}

To each of element $\gamma \in \{1,U,U^2,U^3\}$, we consider a Levi subgroup $\LL_G(\gamma) \subset \GG \times_F \overline{F}$ --- the centralizer of the identity component\footnote{though in this case, $\TT^\gamma$ is always connected} of $\TT^\gamma$.  In fact $\LL_G(U) = \LL_G(U^3) \supset \LL_G(U^2) \supset \LL_G(1) = \TT$.  Let us indicate what the roots of $\TT$ on these Levis are.  First, put
{\tiny \begin{eqnarray*}
x_1 & = & \alpha_2+\alpha_3+\alpha_4+\alpha_5+\alpha_6+\alpha_7 \\
x_2 & = & \alpha_1+\alpha_2+\alpha_3+2\alpha_4+\alpha_5 \\
x_3 & = & \alpha_1 + \alpha_2+2\alpha_3+2\alpha_4+2\alpha_5+2\alpha_6+\alpha_7+\alpha_8 \\
y_1 & = & \alpha_1+\alpha_3+\alpha_4+\alpha_5+\alpha_6+\alpha_7 \\
y_2 & = & \alpha_2+\alpha_3+2\alpha_4+\alpha_5+\alpha_6 \\
y_3 & = & 	-\alpha_1 -2\alpha_2 - 2\alpha_3 - 4\alpha_4 - 3\alpha_5-3\alpha_6-2\alpha_7 -\alpha_8
\end{eqnarray*}}
The roots generate a root system of type $\mathrm{A}_3\times\mathrm{A}_3$ inside $\mathrm{E}_8$, indeed they are a system of simple roots for such a root system.
Computing dot products between these six vectors, we find that their Dynkin diagram is
\[
\xymatrix{
x_1 \ar@{-}[r] & x_2 \ar@{-}[r] & x_3 &  & y_1 \ar@{-}[r] & y_2 \ar@{-}[r] & y_3
}
\]

\begin{prop*}
We have
\begin{enumerate}
\item $\LL_G(U) = \LL_G(U^3)$ is the Levi obtained by adjoining to $\TT$ the roots 
{\tiny \[
\begin{array}{cccccc}
\pm x_1 & \pm x_2 & \pm x_3 & \pm (x_1+x_2) & \pm (x_2+x_3) & \pm(x_1+x_2+x_3) \\
\pm y_1 & \pm y_2 & \pm y_3 & \pm (y_1 + y_2) & \pm(y_2 + y_3) & \pm(y_1 + y_2 + y_3)
\end{array}
\]}

\item $\LL_G(U^2)$ is the Levi obtained by adjoining the roots  to $\TT$
the roots
{\tiny
 \[
\begin{array}{cccc}
\pm x_1 & \pm y_3 & \pm y_1 & \pm y_3 
\end{array}
\] }
\end{enumerate}
In particular the derived subgroups $[\LL_G(U),\LL_G(U)]$ and $[\LL_G(U^2),\LL_G(U^2)]$ are isomorphic to $\SL_4^2$ and $\SL_2^4$, respectively.  We also record:
\begin{enumerate}
\item[(3)] If $s_r$ denotes the reflection across the hyperplane perpendicular to the root $r$, then
\begin{equation}
\label{eq:UCoxeter}
U = s_{x_1} s_{x_2+x_3} s_{-x_1-x_2} s_{y_1}s_{y_2+y_3} s_{-y_1 - y_2}
\end{equation}
\end{enumerate}
In particular $U$ is a Coxeter element of the Weyl group of $\LL_G(U)$.\footnote{As such, in Carter's classification, its conjugacy class in the Weyl group of $\mathrm{E}_8$ is the one labeled $(2 \mathrm{A}_3)''$, see Carter, Springer Lecture Notes in Mathematics {\bf 131}, Table 7}
\end{prop*}


The roots $x_1,x_2,x_3,y_1,y_2,y_3$ may be extended to a basis of simple roots of $\mathrm{E}_8$ in 1720 ways.  Let us consider one of them, the following:
{\begin{equation}
\label{eq:first_admissible_Borel}
{\tiny
\begin{array}{rcl}
\alpha_1' & = & x_1\\
\alpha_2' & = & -\alpha_1 - \alpha_2 -\alpha_3 - \alpha_4 -\alpha_5 - \alpha_6
\\
\alpha_3' & = & x_2\\
\alpha_4' & = & x_3\\
\alpha_5' & = &-2\alpha_1 -2\alpha_2 -4\alpha_3 -5\alpha_4 -4\alpha_5 -3\alpha_6 -2\alpha_7 -\alpha_8
 \\
\alpha_6' & = & y_1\\
\alpha_7' & = & y_2\\
\alpha_8' & = & y_3
\end{array}
}
\end{equation}}
Let us denote the Borel subgroup of $\GG \times_F \overline{F}$ whose simple roots are $\alpha_1',\ldots,\alpha_8'$ by $\BB_G$.
In  what follows, the identification $X^*(\TT) \simeq X_*(\hat{T})$ is understood with respect to $\BB_G$. 

\begin{prop*}
The Borel subgroup $\BB_G$ is $\Gamma$-admissible with respect to $\BB_H$.
\end{prop*}
\begin{proof}
The row vector $(1, -3, -7, 9, 7, 3, -1, 3)$  considered in $X_*(\TT)$ witnesses the admissibility of $\BB_G$
with respect to $U$ or $U^{-1}$; the vector
$(1, -3, -7, 8, 6, 3, 0, 2)$ witnesses the admissibility with respect to $U^2$. Finally one verifies
that $\BB_H \subset \BB_G$, so it is also admissible with respect to the trivial element. \end{proof}

In what follows, the identification $X^*(\TT_G) \simeq X_*(\hat{T})$ is understood with respect to this Borel $\BB_G$,
and {\em not} with respect to the Borel defined by $\alpha_1, \dots, \alpha_8$.

The triple $\BB_G \supset \BB_H \supset \TT$ induces an identification of $\hat{T}_H$ with $\hat{T}_G$, which extends (uniquely up to $\hat{T}_H$-conjugacy) to an injective homomorphism
\[
\psi_1':\hat{H} \hookrightarrow \hat{G}
\] 
The image of $\psi_1'$ contains $\hat{T}_G$, denote it by $\hat{H}_1$.  The center of $\hat{H}_1$ is isomorphic to $\mu_5$ --- in particular as $k$ has characteristic $5$, it has a single $k$-point.  By Proposition 10.3, it follows that $\psi_1' \circ \Frob$ extends to a homomorphism of $L$-groups
$\LH  \to \LG$.  There is some freedom in choosing this extension, one such extension for each lift of $U \in N_\G(\TT)/\TT = N_{\hat{G}}(\hat{T})/\hat{T}$ to $\varpi \in N_{\hat{G}}(\hat{T})$ that (1) normalizes $\hat{H}_1$  and (2) preserves a pinning of $\hat{H}_1$ and (3) has order 4.  Given such a $\varpi_U$, we define $\Lpsi$ be the following formula:
$$\Lpsi(h \rtimes \gamma) = \psi_1'(\Frob(h)) \varpi_U^j \rtimes \gamma \text{ if $q_{\gamma} = 2^j$}$$

Let us define such a $\varpi = \varpi_U$.    If $\alpha$ is a root of $\TT$ on $\GG$, and $\alpha_*$ is the corresponding coroot of $\hat{G}$, we let $\iota_\alpha$ denote the coroot homomorphism $\iota_\alpha:\SL_2(k) \to \hat{G}$.  Let $J \in \SL_2$ be as in \eqref{JDef}. 
Then

\begin{prop*}
Define $\varphi$ via
\begin{equation}
\label{eq:varphi}
 \varphi := \iota_{x_{1}}(J) \cdot \iota_{x_{2}+x_{3}}(J) \cdot \iota_{-x_{1}-x_{2}}(J)\cdot\iota_{y_{1}}(J)\cdot \iota_{y_{2}+y_{3}}(J) \cdot \iota_{-y_{1}-y_{2}}(J)
\end{equation}
Let $\varphi$ be as in \eqref{eq:varphi}.  Then $\varphi$ has order $8$,   normalizes $\hat{H}_1$, and lifts $U$; regarded as an element of $\SL_4 \times \SL_4$, its characteristic polynomial 
is  $(x^2 - 3)(x^2 - 2)$
on each factor. 
 \end{prop*}

Let $r_1 = \alpha_3 + \alpha_4 + \alpha_7 + \alpha_8$ and $r_2 = \alpha_1 + \alpha_2 + \alpha_3 + \alpha_4+ \alpha_6 + \alpha_7 +\alpha_8 - \alpha_0$.   Then $r_{1,*}$ and $r_{2,*}$, regarded as homomorphisms $\Gm \to \hat{T} \subset \hat{G}$, generate the center of $\hat{L}_U$.  
Let us take
\begin{equation}
\label{eq:tomcruise}
t= (-3 r_1 - r_2)_*(\sqrt{2})
 \end{equation}
and put $\varpi_U = \varphi t = t\varphi$.
Now $r_1+r_2 $ differs by $x_1+x_3+y_1+y_3$ by an element of $2 X^*(T)$. This means that $r_{1,*}+ r_{3, *}$
and $x_{1,*}+x_{3,*}+y_{1,*} + y_{3,*}$ take the same value at $-1$; from this we see that $\varpi_U^4=1$. 
  To verify that $\Lpsi$ is a $\sigma$-dual homomorphism, it suffices to prove that 
\[
\inverserho_G(\gamma) \text{ and } \psi_1'(\mathrm{Frob}(\inverserho_H(\gamma)))\varpi_U^{\log_{ 2}(q_{\gamma})} \text{ project to the same element of }\hat{G} \rtimes \gamma \git \hat{L}_\gamma
\]
Because $\GG$ is inner, it is equivalent to prove that 
$
\psi_1 \Sigma_H^*(\Frob(\sqrt{q_{\gamma}}))^{-1} \varpi_U^{\log_{ 2}(q_{\gamma})}
$ and $\Sigma_G^*(\sqrt{q_{\gamma}})^{-1}$
are conjugate in $\hat{L}_{\gamma}$.  Note that $\psi_1' \Sigma_H^* = \psi_1 \Sigma_H^*$ and, by computing,
\begin{eqnarray*}
\Sigma_G & = & 3x_1 + 4x_2+3x_3 + 3y_1 + 4y_2 + 3y_3  + 13 r_1 + 21 r_2 \\
\psi_1 \Sigma_H & = & 2r_1 + 4r_2
\end{eqnarray*}

Let us treat $q_{\gamma} =  2 $ first.  We have $\Frob(\sqrt{  2}) = -\sqrt{  2}$, so we want to check 
(where $\sim$ denotes $\hat{L}_{\gamma}$-conjugacy) 
\begin{equation}
\label{eq:boombox}
\begin{array}{rcl}
\psi_1(\Sigma_H^*(-\sqrt{  2}))^{-1} t \varphi & \sim & \Sigma_G^*(\sqrt{  2})^{-1}
\end{array}
\end{equation}
are conjugate in $\hat{L}_U$. 
In particular $\psi_1 \Sigma_H$ (and also $t$) is central in $\hat{L}_U$, so to show \eqref{eq:boombox} it is equivalent to show that 
$t \psi_1\Sigma_H^*(-\sqrt{  2})^{-1} \Sigma_G^*(\sqrt{  2})$ 
is $\hat{L}_U$-conjugate to $\varphi^{-1}$.  Since $\psi_1 \Sigma_H$ is even in the root lattice, $\psi_1 \Sigma_H^*(-\sqrt{2}) = \psi_1\Sigma_H^*(\sqrt{2})$.  So we are left with showing that
\[
t\frac{\Sigma_G^*}{\psi_1 \Sigma_H^*}(\sqrt{2}) \sim \varphi^{-1}
\]
Note $\Sigma_H - \Sigma_G = -3 x_1 - 4x_2 - 3 x_3 - 3y_1 - 4 y_2 - 3 y_3 - 11 r_1 - 17 r_2$.  So by \eqref{eq:tomcruise} the left-hand side is $(3 x_1 + 4 x_2 +3 x_3 + 3 y_1 +4 y_2 +3 y_3)_*(\sqrt{2})$. 
By a computation in $\SL_4 \times \SL_4$, this is conjugate both to $\varphi$ and $\varphi^{-1}$.

The case $q_{\gamma}=3$ is similar. We now check $q_{\gamma}=\pm 1$. For $q_{\gamma}=1$ there is nothing to check and for $q_{\gamma}=-1$ we must see that
$ t^2 \frac{\Sigma_G}{  \psi_1 \Sigma_H^*} (2) \sim \varphi^{-2}$ i.e. the square of the previous conjugacy, but now for the smaller group $\hat{L}_{\gamma}$.  
The left-hand side equals the diagonal matrix $D = \left(\begin{array}{cc}  D & \mathbf{0} \\ \mathbf{0} & D^{-1} \end{array}\right)$ in each $\SL_4$ factor
where $D = \left( \begin{array}{cc} 2 & 0 \\ 0 & 3 \end{array}\right) $; 
on the other hand $\varphi^2 =  \left( \begin{array}{cc} - J & \mathbf{0} \\ \mathbf{0} & - J \end{array}\right)  \in \SL_4,$
and indeed both $D$ and $D^{-1}$ are conjugate to $-J$ inside $\SL_2(k)$.

\bibliography{TVBib} 
\bibliographystyle{amsplain}

\end{document}